\documentclass{amsart}

\usepackage{amsmath,amssymb,amsthm,amsfonts}
\usepackage{hyperref}

\usepackage{color}

\newtheorem{lemma}{Lemma}[section]
\newtheorem{theorem}{Theorem}[section]

\newtheorem{proposition}{Proposition}[section]

\newtheorem{corollary}{Corollary}[section]

\numberwithin{equation}{section}

\newcommand{\dis}{\displaystyle}

\newcommand{\C}{\mathbb{C}}
\newcommand{\D}{\mathbb{D}}
\newcommand{\R}{\mathbb{R}}

\newcommand{\M}{\mathbb{M}}

\newcommand{\semiG}{\mathbb{A}}

\newcommand{\FM}{\mathbf{M}}

\newcommand{\FL}{\mathbf{L}}
\newcommand{\FI}{\mathbf{I}}

\newcommand{\CB}{\mathcal{B}}
\newcommand{\CC}{\mathcal{C}}
\newcommand{\CD}{\mathcal{D}}
\newcommand{\CE}{\mathcal{E}}
\newcommand{\CF}{\mathcal{F}}

\newcommand{\CM}{\mathcal{M}}
\newcommand{\CN}{\mathcal{N}}

\newcommand{\MRk}{\mathfrak{R}}

\newcommand{\na}{\nabla}

\newcommand{\al}{\alpha}
\newcommand{\be}{\beta}
\newcommand{\ga}{\gamma}
\newcommand{\om}{\omega}
\newcommand{\la}{\lambda}
\newcommand{\de}{\delta}
\newcommand{\si}{\sigma}
\newcommand{\pa}{\partial}

\newcommand{\eps}{\epsilon}

\newcommand{\De}{\Delta}

\newcommand{\lng}{{\langle}}
\newcommand{\rng}{{\rangle}}

\newcommand{\lag}{\langle}
\newcommand{\rag}{\rangle}

\begin{document}
\title[Green's function of the Navier-Stokes-Maxwell system]{Green's function and large time behavior of the Navier-Stokes-Maxwell system}
\author[R.-J. Duan]{Renjun Duan}
\address{Department of Mathematics, The Chinese University of Hong Kong,
Shatin, Hong Kong}
\email{rjduan@math.cuhk.edu.hk}
\date{\today}

\begin{abstract}
In this paper, we are concerned with the system of the compressible Navier-Stokes equations coupled with the Maxwell equations through the Lorentz force in three space  dimensions. The asymptotic stability of the steady state with the strictly positive constant density and the vanishing velocity and electromagnetic field is established under small initial perturbations in regular Sobolev space. For that, the dissipative structure of this hyperbolic-parabolic system is studied to include the effect of the electromagnetic field into the viscous fluid and turns out to be more complicated than that in the simpler compressible Navier-Stokes system. Moreover, the detailed analysis of the Green's function to the linearized system is made with applications to derive the rate of the solution converging to the steady state.
\end{abstract}

\maketitle

\vspace{-5mm}

\tableofcontents

\section{Introduction}

Plasma dynamics is a field of studying flow problems of electrically conducting fluids, particularly of ionized gases \cite{KT,Ni}. The scope of plasma dynamics is very broad. A complete analysis in this field consists of the study of the gasdynamic field, the electromagnetic field and the radiation field simultaneously \cite{MiW,Pai}. Little has been done on such a complete analysis. Thus, many studies at both the microscopic kinetic level and the macroscopic continuous level are made in the individual field or in a little complicated situation taking into account the interaction of two fields with the rest one as a subsidiary factor \cite{Bi,Da,KT, MiW,Ni,Pai}.

At the macroscopic level, the motion of plasma fluids, for instance ionized gases, is described in terms of the classical gasdynamic quantities, such as density, pressure, flow velocity, etc, under the influence of the electromagnetic field induced by the electrically conducting fluid itself.  The fundamental system to govern the time evolution of those macroscopic quantities consists of the gasdynamic equations, e.g. the Navier-Stokes equations in the viscous case, coupled with the equations of the self-consistent electromagnetic field, e.g. the Maxwell equations. In general, plasma is a mixture of various species: charged ions and electrons, and neutral particles. In many simplified cases when the variation of one composition is small and its effect is negligible, the plasma may be considered as a single fluid.

In this paper, we are interested in the motion of one fluid described by the compressible Navier-Stokes equations coupled with the Maxwell equations through the Lorentz force.
In the isentropic case, the motion equations read as
\begin{equation}\label{NSM}
    \left\{\begin{array}{l}
      \dis \pa_t n + \na \cdot (nu)=0,\\[1mm]
      \dis \pa_t (nu) + \na \cdot (n u\otimes u) +\na P(n)=- n (E+ u\times B) +\nu \De u,\\[1mm]
      \dis \pa_t E - \na \times B = nu,\\[1mm]
      \dis \pa_t B + \na \times E = 0,\\[1mm]
      \dis \na \cdot E = n_{\rm b}- n, \ \ \na\cdot B =0.
    \end{array}\right.
\end{equation}
Here, the unknowns are $n=n(t,x)\geq 0$, $u=u(t,x)\in \R^3$, $E=E(t,x)\in \R^3$ and $B=B(t,x)\in\R^3$ over $\{t>0,x\in \R^3\}$, denoting the fluid density, fluid velocity, electric field and magnetic field, respectively. $\nu>0$ is a constant denoting viscosity coefficient, and $n_{\rm b}>0$ is a constant denoting the uniform background density (e.g. of ions). $P$ depending only on $n$ denotes the pressure function with the usual assumption that $P$ is smooth in the argument and $P'>0$. Initial data of the system is given as
\begin{equation}\label{NSM.ID}
    n(0,x)=n_0(x), u(0,x)=u_0(x), E(0,x)=E_0(x), B(0,x)=B_0(x), x\in \R^3.
\end{equation}
For convenience of presentation, we call \eqref{NSM} the Navier-Stokes-Maxwell system. Notice that the same terminology was used in Masmoudi \cite{Ma} but for a different modelling system.

The consideration of system \eqref{NSM} is motivated by three issues that we shall address in more detail later on. The first issue is to generalize some known results of the viscous fluid, cf. Matsumura-Nishida \cite{MN1}, Kawashima \cite{Ka2}, Hoff-Zumbrun \cite{HZ1}, to include the effect of the electromagnetic field.
Notice that system \eqref{NSM} is of the hyperbolic-parabolic type. Indeed, from the later proof, the nonlinear coupling of system \eqref{NSM} makes it possible for the hyperbolic Maxwell system to gain some weak spatial regularity or energy dissipation rate from the diffusion effect of the viscous fluid. On the other hand, as far as the large time behavior is concerned, the speed of convergence for the fluid velocity has to be reduced because of the appearance of the electromagnetic field. The second issue is to compare the current  study of the Navier-Stokes-Maxwell system with that for other two relative models,
the Navier-Stokes-Poisson system and the magnetohydrodynamic equations, cf. Li-Matsumura-Zhang \cite{LMZ-NSP}, Umeda-Kawashima-Shizuta \cite{UKS}. The final issue is to expose the special dissipative structure and regularity-loss property of system \eqref{NSM}, which have been recently investigated for both fluid and kinetic models, cf. Hosono-Kawashima \cite{HK}, Ide-Haramoto-Kawashima \cite{IHK}, Duan \cite{Du-EM}, Duan-Strain \cite{DS-VMB}, Duan \cite{Du-1VMB}. Notice that we shall derive system \eqref{NSM} from the one-species Vlasov-Maxwell-Boltzmann system \cite{Du-1VMB} by employing the Liu-Yang-Yu's macro-micro decomposition \cite{LYY}. This also  provides a link between the macroscopic Navier-Stokes-Maxwell system and the kinetic Vlasov-Maxwell-Boltzmann system for the single fluid with the electromagnetic field; the similar dissipative property for both  systems will be also discussed later on.

\subsection{Main results}

The Navier-Stokes-Maxwell system \eqref{NSM} has a class of naturally existing constant steady states in which the density is $n_{\rm b}$, the magnetic field is an arbitrary constant vector $B_{\rm b}$ and both the velocity and the electric field are zero. Here and hereafter, we call $B_{\rm b}$ the uniform background magnetic field. In this paper, $B_{\rm b}=0$ is assumed, and we are concerned with the asymptotic stability of the constant steady state with the density $n_{\rm b}$ and the vanishing velocity and electromagnetic field for the Cauchy problem \eqref{NSM}, \eqref{NSM.ID}. We shall point out later on that part of results can be directly generalized to the case of non-vanishing constant magnetic field, that is, $B_{\rm b}\neq 0$.

To the end, set
\begin{equation*}
   \ga=\sqrt{P'(n_{\rm b})}>0,\ \be=\sqrt{n_{\rm b}}>0,\ \mu=\nu/n_{\rm b}>0,
\end{equation*}
and take change of variables
\begin{equation*}
 \rho=n-n_{\rm b},\  v=\frac{n_{\rm b}}{\sqrt{P'(n_{\rm b})}}u,\ \widetilde{E}=(\frac{n_{\rm b}}{{P'(n_{\rm b})}})^{1/2}E,\ \widetilde{B}=(\frac{n_{\rm b}}{{P'(n_{\rm b})}})^{1/2}B.
\end{equation*}
Then, the Cauchy problem \eqref{NSM}, \eqref{NSM.ID} can be reformulated as
\begin{equation}\label{eq}
    \left\{\begin{array}{l}
     \dis \pa_t \rho +\ga \na\cdot v =-\frac{\ga}{\be^2} \na \cdot (\rho v),\\
      \dis \pa_t v+\ga \na \rho +\be \widetilde{E}-\mu \De v=- \frac{\ga}{\be^2} v \cdot \na v -\frac{\be^2}{\ga}(\frac{\na P(\rho+n_{\rm b})}{\rho+n_{\rm b}}-\frac{P'(n_{\rm b})}{n_{\rm b}}
\na \rho)\\[3mm]
      \dis \hspace{4.5cm}- \frac{\ga}{\be} v\times \widetilde{B} +\nu(\frac{1}{\rho+n_{\rm b}}-\frac{1}{n_{\rm b}}) \De v,
    \end{array}\right.
\end{equation}
and
\begin{equation}\label{eq.2}
  \left\{\begin{array}{l}
       \dis \pa_t \widetilde{E} - \na \times \widetilde{B} -\be v = \frac{1}{\be} \rho v ,\\
      \dis \pa_t \widetilde{B} + \na \times \widetilde{E} = 0,\\
      \dis \na \cdot \widetilde{E} = -\frac{\be}{\ga}\rho, \ \ \na\cdot \widetilde{B} =0,
    \end{array}\right.
\end{equation}
with
\begin{equation}\label{eq.id}
    \rho(0,x)=\rho_0(x), v(0,x)=v_0(x), \widetilde{E}(0,x)=\widetilde{E}_0(x), \widetilde{B}(0,x)=\widetilde{B}_0(x), x\in \R^3.
\end{equation}
One of the main results concerning the global existence and large time behavior of solutions to the above Cauchy problem on the Navier-Stokes-Maxwell system is stated as follows.

\begin{theorem}\label{thm.main}
Let $N\geq 4$.  Assume that initial data $[\rho_0,v_0,\widetilde{E}_0,\widetilde{B}_0]$ satisfies the last equation of \eqref{eq.2} and $\|[\rho_0,v_0,\widetilde{E}_0,\widetilde{B}_0]\|_{H^N}$ is sufficiently small. Then the Cauchy problem \eqref{eq}, \eqref{eq.2}, \eqref{eq.id} admits a unique global solution $[\rho,v,\widetilde{E},\widetilde{B}]$ with
\begin{eqnarray*}
&&\dis [\rho,v,\widetilde{E},\widetilde{B}]\in C([0,\infty);H^N(\R^3)),\\
&& \rho\in L^2((0,\infty);H^N(\R^3)),\
\na v\in L^2((0,\infty);H^N(\R^3)),\\
&&\na \widetilde{E}\in L^2((0,\infty);H^{N-2}(\R^3)),\
\na^2 \widetilde{B}\in L^2((0,\infty);H^{N-3}(\R^3)),
\end{eqnarray*}
and
\begin{multline}\label{thm.main.1}
    \|[\rho(t),v(t),\widetilde{E}(t),\widetilde{B}(t)]\|_{H^N}^2+\int_0^t (\|\rho(s)\|_{H^N}^2+\|\na v(s)\|_{H^{N}}^2\\
    +\|\na \widetilde{E}(s)\|_{H^{N-2}}^2+\|\na^2 \widetilde{B}(s)\|_{H^{N-3}}^2)ds
    \leq C\|[\rho_0,v_0,\widetilde{E}_0,\widetilde{B}_0]\|_{H^N}^2
\end{multline}
for any $t\geq 0$. Moreover, if $\|[\rho_0,v_0,\widetilde{E}_0,\widetilde{B}_0]\|_{L^1\cap H^{N+2}}$ is further sufficiently small, then the obtained solution $[\rho,v,\widetilde{E},\widetilde{B}]$ satisfies
\begin{eqnarray*}
  \|\rho(t)\| &\lesssim& (1+t)^{-1},\\
    \|v(t)\| &\lesssim& (1+t)^{-\frac{5}{8}},\\
      \|\widetilde{E}(t)\| &\lesssim& (1+t)^{-\frac{3}{4}}\ln (3+t),\\
        \|\widetilde{B}(t)\| &\lesssim& (1+t)^{-\frac{3}{8}},
\end{eqnarray*}
for any $t\geq 0$.
\end{theorem}

We now give some remarks on the above theorem. First of all, the global existence result stated in Theorem \ref{thm.main} also holds for the case of $B_{\rm b}\neq 0$. In fact, the proof is easily modified to derive the a priori estimate as in \eqref{thm.main.1} by noticing that the additional linear term generated by  $B_{\rm b}\neq 0$ is $u\times B_{\rm b}$ which is orthogonal to the velocity $u$; see the proof of Theorem \ref{thm.ge}. Secondly, \eqref{thm.main.1} shows that the electromagnetic field satisfying the Maxwell system indeed has some time-space integrability property weaker than the fluid components, whereas its highest-order spatial derivative is not time-space integrable, which becomes a typical feature of this kind of systems with the regularity-loss property.

Thirdly, the time decay rates stated in  Theorem \ref{thm.main} depend  essentially on the analysis of the Green's function $G$ of the linearized Navier-Stokes-Maxwell system; see Theorem \ref{thm.gr}, Theorem \ref{thm.up}, Theorem \ref{thm.tdp} and Theorem \ref{thm.odecay}. Specifically, from Theorem \ref{thm.up}, the Fourier transform of $G$ have different behaviors over different frequency domains, in the rough way as
\begin{equation}\label{est.G1}
    |\hat{G}(t,k)|\lesssim\left\{\begin{array}{ll}
      \dis e^{-O(1)|k|^2 t} +e^{-O(1)|k|^4 t} &\dis \quad |k|\leq \eps,\\[3mm]
      \dis e^{-O(1)t} &\dis \quad \eps\leq |k|\leq L,\\[3mm]
      \dis e^{-O(1)t}+e^{-O(1)|k|^2 t} + e^{-\frac{O(1)t}{|k|^2}} &\dis \quad |k|\geq L,
    \end{array}\right.
\end{equation}
for two properly chosen constants $0<\eps\ll 1 \ll L<\infty$; more precise estimates on $\hat{G}$ can be found in Theorem \ref{thm.up}.
It should be emphasized those time rates obtained in Theorem \ref{thm.lr} on the basis of the elementary Lyapunov property of the  linearized system, which implies
\begin{equation}\label{est.G2}
   |\hat{G}(t,k)|\lesssim e^{-\frac{|k|^4t}{(1+|k|^2)^3}},\ \ k\in \R^3,
\end{equation}
fail to yield the ones of the nonlinear solution given in Theorem \ref{thm.main} above. Notice that \eqref{est.G2}, which is a simpler form of \eqref{est.G1}, only provides the partial information of $\hat{G}$.

Finally we point out some differences of time rates between the linearized and nonlinear cases.
From Corollary  \ref{cor.ld}, the solution $[\rho,v,\widetilde{E},\widetilde{B}]$ to the linearized homogeneous system corresponding to \eqref{eq}-\eqref{eq.2} decays as
\begin{eqnarray*}
  \|\rho(t)\| &\lesssim& (1+t)^{-\frac{5}{4}},\\
    \|v(t)\| &\lesssim& (1+t)^{-\frac{5}{8}},\\
      \|\widetilde{E}(t)\| &\lesssim& (1+t)^{-\frac{3}{4}},\\
        \|\widetilde{B}(t)\| &\lesssim& (1+t)^{-\frac{3}{8}},
\end{eqnarray*}
for any $t\geq 0$, provided that  $[\rho_0,v_0,\widetilde{E}_0,\widetilde{B}_0]$ belongs to $L^1\cap H^{s}$ for properly large $s$. Thus, in Theorem \ref{thm.main}, the time decay rates of $v$ and $\widetilde{B}$ are optimal in the sense that they are the same as those in the linearized case, but other two rates for $\rho$ and $\widetilde{E}$ are not optimal because they are strictly smaller than those in the linearized case. From the proof of Theorem \ref{thm.main} later on,  the slower time decay for $\rho$ and $\widetilde{E}$ mainly results from the nonlinear effect, particularly in the presence of a nonlinear term  $v\times \widetilde{B}$ in the momentum equation of system \eqref{eq}.

\subsection{Motivations and literature}

As mentioned before, there are three motivations resulting the current investigation of the Navier-Stokes-Maxwell system. The first motivation is to generalize some known results of the viscous fluid, particularly about the global existence and large time behavior of classical solutions near non-vacuum steady state, in order to include the effects of the electromagnetic field. In fact, there are extensive studies concerning those results for the Navier-Stokes equations in the context of gasdynamics,
in the form of
\begin{equation}\label{NS}
    \left\{\begin{array}{l}
      \dis \pa_t n + \na \cdot (nu)=0,\\[1mm]
      \dis \pa_t (nu) + \na \cdot (n u\otimes u) +\na P(n)=\nu \De u,
    \end{array}\right.
\end{equation}
where the electromagnetic field is absent and thus the Maxwell system is decoupled from the fluid dynamic equations. We here mention only some of results related to our interest.

One of the most important aspects of the Navier-Stokes system \eqref{NS} is due to its hyperbolic-parabolic character. It is crucial to understand the interaction between the hyperbolic conservative part and the parabolic diffusive part. Indeed, as we shall see even in the general situation \cite{Ka2,SK,Vi}, this kind of interaction can often induce some weak smoothness and dissipation of the hyperbolic component of the coupled system. It is Kanel who first observed this property in the proof of global existence of solutions to \eqref{NS} in one space dimension. Later on, in three space dimensions, global well-posedness of the near-constant-equilibrium solutions to the Cauchy problem or the initial boundary value problem was established by Matsumura-Nishida \cite{MN-CMP,MN0,MN1} using the classical energy method. Since then, the perturbation theory of the Navier-Stokes system has been investigated by many authors.

The study of the dissipative property of the system \eqref{NS} is not only important for the global existence of solutions under small perturbations but also useful for obtaining the rates of convergence of solutions trending towards the equilibrium. Still in \cite{MN0}, Matsumura-Nishida obtained the rate of convergence for the Navier-Stokes system in $\R^3$ as
\begin{equation*}
    \|[n-n_{\rm b},u]\|_{H^2}\leq C(1+t)^{-\frac{3}{4}}
\end{equation*}
if initial perturbation is small belonging to $H^4\cap L^1$. Ponce  \cite{Po} gave the time decay rate in $L^p$ with $2\leq p\leq \infty$ as
\begin{equation*}
    \|[n-n_{\rm b},u]\|_{L^p}\leq C(1+t)^{-\frac{3}{2}(1-\frac{1}{p})}
\end{equation*}
provided that small initial perturbation belongs to $H^s\cap W^{s,1}$ for $s$ large enough.

On the other hand,
a general approach for obtaining the optimal time decay of solutions in $L^p$ space with $p\geq 2$ in any space dimension was developed by Kawashima \cite{Ka2,Ka0} and Shizuta-Kawashima \cite{SK}. Precisely, the main tool used by Kawashima \cite{Ka2,Ka0} is the Fourier analysis applied to the linearized homogeneous system, and the key part in the proof is to construct some compensation function to capture the dissipation of the hyperbolic component in the solution and then obtain an estimate on the Fourier transform $\hat{U}$ of the solution $U$ as
\begin{equation}\label{U.bd}
|\hat{U}(t,k)|\leq C e^{-\frac{c|k|^2}{1+|k|^2}t}|\hat{U}_0|
\end{equation}
for any $t$ and $k$. Thus, from the above estimate, it is well known the solution over the high frequency domain decays exponentially while over the low frequency domain it decays polynomially with the rate of the heat kernel.
Moreover, in \cite{SK}, the authors proposed a condition, which is now called the Shizuta-Kawashima condition, in order to assure the time decay of solutions to linear systems of equations of the hyperbolic-parabolic type. The Shizuta-Kawashima condition also plays a key role in the study of stability and large time behavior of the solution to the nonlinear system; see Hanouzet-Natalini \cite{HN}, Yong \cite{Yo}, Bianchini-Hanouzet-Natalini  \cite{BHN} and references therein. For some extensions of the Shizuta-Kawashima condition, refer to Beauchard-Zuazua \cite{BZ} and Mascia-Natalini \cite{MN}.

About the time decay in $L^p$ space with $1\leq p\leq 2$, Zeng \cite{Z2} provided a complete analysis of the Green matrix with some sharp, pointwise bounds for the one-dimensional viscous heat-conductive fluid system. The difficulty for the rate in $L^p$ for $p$ close to one lies in the fact that the density satisfies the mass conservation law.    The result of \cite{Z2} was later extended by Liu-Zeng \cite{LZ1} to the case of zero heat conductivity and also by Liu-Zeng \cite{LZ2,LZ3} to general hyperbolic-parabolic systems of conservation laws in one space dimension.
Besides, an extension of \cite{Z2} to the case of multi-dimensions was made by Hoff-Zumbrun \cite{HZ1,HZ2}. Precisely, \cite{HZ1} showed that for the Navier-Stokes system in $\R^3$,
\begin{equation*}
    \|[n-n_{\rm b},u]\|_{L^p}\leq \left\{\begin{array}{ll}
      \dis C t^{-\frac{3}{2}(1-\frac{1}{p})}, &\dis \quad 2\leq p\leq \infty,\\[3mm]
       \dis C t^{-2(1-\frac{1}{p})}, &\dis\quad 1\leq p<2,
    \end{array}\right.
\end{equation*}
for all large $t>0$, provided that the small initial perturbation  belongs to $H^s\cap L^1$ with $s\geq 4$.
The main idea of \cite{HZ1} is to write the Green matrix $G$ in terms of the Fourier transform as
\begin{eqnarray*}
&\dis \hat{G}(t,k)=\hat{G}^{(1)}(t,k)+\hat{G}^{(2)}(t,k),\\
&\dis \hat{G}^{(1)}(t,k)= \hat{G}(t,k)\chi(k),\ \ \hat{G}^{(2)}(t,k)= \hat{G}(t,k)[1-\chi(k)]
\end{eqnarray*}
with $\chi(k)$ a smooth cutoff function supported on the finite domain, and thus shift the difficulty of the analysis to the high-frequency part $\hat{G}^{(2)}(t,k)$. We also mention the work in the case of multi-dimensions by Liu-Wang \cite{LW} for the pointwise estimate on the Green's function, and similar work by Li \cite{Li} and Wang-Yang \cite{WY}.
Finally, we mention a recent work Li-Zhang \cite{LZ} for the study of the optimal large time behavior of the Navier-Stokes system in the Besov space when initial perturbation does not belong to $L^1$ but some homogeneous Besov space which can induce some extra time rate.

In addition, when there is a given potential force, the density of the naturally existing steady state is not a constant but a function of spatial variable. In this case, the Fourier analysis fails due to the difficulty of variant coefficients.
The slow time decay property of solutions was studied by Deckelnick \cite{De1,De2} on the basis of the refined energy estimates together with an ordinary differential  inequality. Notice that even though the time rate obtained in \cite{De1,De2} is not optimal, the developed method has proven very useful with applications to the case when the spectral analysis which can lead to the optimal time decay rate is not clear; see Yang-Zhao \cite{YZ-CMP} for a study of the Vlasov-Poisson-Boltzmann system. An extension of \cite{De1,De2}  was made by Ukai-Yang-Zhao \cite{UYZ}  to obtain the almost optimal time decay rates. Recently, Duan-Ukai-Yang-Zhao \cite{DUYZ} developed a method of the combination of the spectral analysis  and energy estimates to deal with the optimal time decay of solutions, and an extension of \cite{DUYZ} was made in \cite{DLUY}; see also \cite{DUY} for the detailed presentation of the method. When the given external force depending only on spatial variable has non-zero rotation, the velocity of the steady flow thus can not vanish; for such case, the existence  and time asymptotic stability of the nontrivial steady state were studied by Shibata-Tanaka \cite{ST1} and \cite{ST2}, respectively.

At this moment, those mentioned works about the optimal time decay of solutions over the unbounded domain are concerned with the case of the whole space. Kagei-Kobayashi \cite{KK2,KK1} proved the large time behavior of solutions over the half space in $\R^3$. Kobayashi-Shibata \cite{KoSh} and later Kobayashi \cite{Ko} gave an analysis of the Green's function and large time behavior of solutions for the compressible Navier-Stokes equations in an exterior domain of $\R^3$.

\medskip

The second motivation is to compare the current study of the Navier-Stokes-Maxwell system to the one of some modelling systems such as the  Navier-Stokes-Poisson system in semiconductor theory and the MHD system in magnetohydrodynamics. Here, we also mention some previous work of these two systems only related to our study.
First of all, the Navier-Stokes-Poisson system can be used to simulate, for instance in semiconductor devices \cite{MRS}, the transport of charged particles under the influence of the electrostatic potential force field, and in the compressible case it takes the form of
\begin{equation}\label{NSP}
    \left\{\begin{array}{l}
      \dis \pa_t n + \na \cdot (nu)=0,\\[2mm]
      \dis \pa_t (nu) + \na \cdot (n u\otimes u) +\na P(n)=-n E+\nu \De u,\\[2mm]
      \dis E=\na\Phi,\ \ \na\cdot E=n_{\rm b}-n.
    \end{array}\right.
\end{equation}
Here, compared with the Navier-Stokes-Maxwell system, the magnetic field $B$ is omitted and the electric field $E$ is generated by a consistent potential function $\Phi$ therefore satisfying the Poisson equation
\begin{equation*}
    \De \Phi=n_{\rm b}-n.
\end{equation*}

It is easy to see that the compressible Navier-Stokes-Poisson system \eqref{NSP} is a coupled system of the hyperbolic-parabolic-elliptic type. This essential feature makes the Navier-Stokes-Poisson system have some special dissipative and time decay properties different from both the Navier-Stokes system and the Navier-Stokes-Maxwell system. Li-Matsumura-Zhang \cite{LMZ-NSP} made a delicate analysis of the linearized Navier-Stokes-Poisson system and investigated the spectrum of the linear semigroup in terms of the decomposition of wave modes
at lower frequency and higher frequency respectively, which can be used to deal with the influences of the rotating effect of the electric field. Precisely, for the rate of convergence, they showed that the solution over $\R^3$ satisfies
\begin{eqnarray*}
   \|n-n_{\rm b}\|&\leq & C (1+t)^{-\frac{3}{4}},\\
   \|u\|&\leq &C(1+t)^{-\frac{1}{4}},\\
   \|E\|&\leq &C(1+t)^{-\frac{1}{4}},
\end{eqnarray*}
for any $t\geq 0$, provided that initial perturbation belonging to $H^s\cap L^1$ for $s$ properly large is sufficiently small, and furthermore, these time rates are optimal due to the fact that they can also be obtained as the lower bounds if the Fourier transform of initial density perturbation $n_0-n_{\rm b}$ is strictly positive near some neighborhood of the origin, i.e.
\begin{equation}\label{n.inf}
    \inf_{|k|\ll 1}|\widehat{n_0-n_{\rm b}}(0,k)|>0.
\end{equation}
As pointed out in \cite{LMZ-NSP}, it is the rotating effect of electric
field which reduces the time decay rate of the momentum component compared with the Navier-Stokes system \eqref{NS}. We remark that if  $\widehat{n_0-n_{\rm b}}(0,k)$ behaves instead like $|k|$ as $|k|$ tends to zero, for instance
\begin{equation*}
    \int_{\R^3}(1+|x|)|n_0(x)-n_{\rm b}|dx<\infty
\end{equation*}
holds true, then those time decay rates obtained in \cite{LMZ-NSP} could be improved. Moreover, \eqref{n.inf} implies that $E_0=\na \Phi_0$ does not belong to $L^1(\R^3)$, otherwise it is possible to improve the time rate of $\|E\|$ as $(1+t)^{-\frac{3}{4}}$.

However, in our discussion of time decay rates for the linearized Navier-Stokes-Maxwell system, since the electric field $E$ satisfies a time evolution equation and initial data $E_0$ is given belonging to $L^1$, the higher time rate $(1+t)^{-\frac{3}{4}}$ of $\|E\|$ is natural. In addition, the time rate $(1+t)^{-\frac{5}{8}}$ of $\|u\|$ for the momentum component is still strictly slower than $(1+t)^{-\frac{3}{4}}$ in the case of the Navier-Stokes system.  This point essentially results from the magnetic field effect of $B$ which has the slowest time decay rate among all the components of the system. Finally, $\|n-n_{b}\|$ decays as $(1+t)^{-\frac{5}{4}}$ in the fastest way, because on one hand, the constrict condition $n-n_{\rm b}=-\na\cdot E$ makes $n-n_{b}$ gain one more spatial derivative, and on the other hand, $n$ at the linearized level turns out to be decoupled with and hence not effected by the magnetic field $B$ with the slower decay, where the decoupling property is due to the hidden assumption in our study that the background magnetic field is zero. Thus, it would be quite interesting to extend the current results to the case of non-vanishing constant  background magnetic field.

After \cite{LMZ-NSP}, Wang-Wu \cite{WW} obtained the pointwise estimates of solutions to the Navier-Stokes-Poisson system through the analysis of the Green's function. Zhang-Li-Zhu \cite{ZLZ} generalized the isentropic case to the non-isentropic case with additional efforts taking care of the temperature equation. Hao-Li \cite{HL} proved the global existence of solutions in the Besov space. We also mention the work Zhang-Tan \cite{ZT} for a study in two space dimensions and a survey Hsiao-Li \cite{HL-s}.

\medskip

Next, let us discuss a little about the magnetohydrodynamic system \cite{Bi,Da}. In the isentropic case, it reads
\begin{equation}\label{MHD}
    \left\{\begin{array}{l}
      \dis \pa_t n + \na \cdot (nu)=0,\\[2mm]
      \dis \pa_t (nu) + \na \cdot (n u\otimes u) +\na P(n)=(\na\times B)\times B+\nu \De u,\\[2mm]
      \dis \pa_t B +\na\times (\na\times B-u\times B)=0,\ \ \na\cdot B=0.
    \end{array}\right.
\end{equation}
Different from the Navier-Stokes-Poisson system, system \eqref{MHD} models the interaction between the viscous fluid and the variable magnetic field. In one and two dimensions, Kawashima-Okada \cite{KO} and Kawashima \cite{Ka1},  respectively, proved the global existence of smooth solutions near constant equilibrium states. Later, Umeda-Kawashima-Shizuta \cite{UKS} gave a detailed study of the linearized electro-magneto-fluid system, where more general situation of non-vanishing constant background magnetic field was considered. The authors obtained the same estimate on the solution as \eqref{U.bd}, which therefore implies the same time decay property as in the case of the Navier-Stokes system. Notice that the spectrum analysis was also made in \cite{UKS}. We finally mention two recent works Chen-Tan \cite{CT} and Zhang-Zhao \cite{ZZ} both about  the time decay rates for the nonlinear system, and Masmoudi \cite{Ma} for the study of the global existence and exponential convergence rate of regular solutions to the full magnetohydrodynamic system over two dimensional bounded domain. We remark that the current method could be applied to the model in \cite{Ma} to obtain some similar results in the framework of small perturbation.

\medskip

The third motivation is based on recent investigations on  the fluid and kinetic systems in the presence of the electromagnetic field, such as the Euler-Maxwell system with relaxation \cite{Du-EM} and the Vlasov-Maxwell-Boltzmann system \cite{DS-VMB,Du-1VMB}. For the linearized dynamics in all works \cite{Du-EM,DS-VMB,Du-1VMB} and the current work, the main idea is to construct the time-frequency Lyapunov functional to capture the dissipation of the electromagnetic field. Since the pure homogeneous Maxwell system preserves energy and hence has no dissipation, the recovered dissipation of the electromagnetic field for those coupled systems is truly nontrivial.

The optimal large time behavior of solutions near equilibriums for the Vlasov-Maxwell-Boltzmann system was proved by Duan-Strain \cite{DS-VMB} in the two-species case, but it is still unclear in the one-species case even though Duan \cite{Du-1VMB} made the full linearized analysis of the one-species Vlasov-Maxwell-Boltzmann system and proved the global existence of classical solutions under small perturbation. In fact, the one-species Vlasov-Maxwell-Boltzmann system at the linearized level shows the same time-frequency Lyapunov property as in \eqref{thm.ly.3}, that is,
\begin{equation*}
   \pa_t  \CE(\hat{U}(t,k))+c \frac{|k|^4}{(1+|k|^2)^3} \CE(\hat{U}(t,k))\leq 0.
\end{equation*}
Thus, the linearized solution operator also has the same time decay property as stated in Theorem \ref{thm.lr}. But, those rates are too slow to deduce the explicit time rate of solutions to the nonlinear system by using the usual bootstrap procedure.

Different from the case of the one-species Vlasov-Maxwell-Boltzmann system, our current study of the Navier-Stokes-Maxwell system has made full use of the pointwise estimate on the Green's function of the linearized system. Our analysis of the Green's function here strictly refines the time decay property of all the components except the magnetic field, which is hence able to be used to obtain the time decay of the nonlinear solution. For the computation of the Green's function, we use the technique of \cite{Du-EM} to write the linearized system as two decoupled system, which makes it possible to separately solve them and also shift the whole difficulty to the analysis of some ordinary differential equations. On the other hand, the estimate on the Green's function in this paper is much more complicated than that in the case of the relaxed Euler-Maxwell system \cite{Du-EM} because of the hyperbolic-parabolic character of the Navier-Stokes-Maxwell system, which results that multiple real eigenvalues could occur over the finite frequency domain. We shall make a careful analysis of the eigenvalues to overcome this difficulty. Therefore, the current method could also provide a clue to hopefully derive the optimal large time behavior of solutions to the    one-species Vlasov-Maxwell-Boltzmann system \cite{Du-1VMB}, particularly the corresponding Green's function as studied by Liu-Yu \cite{LY,LY-1D} for the pure Boltzmann equation.

Finally, it should be pointed out that all the coupled systems considered here and in \cite{Du-1VMB,Du-EM,DS-VMB} have a common feature that they are of the regularity-loss type. A typical estimate on the solution $U$ to the linearized system with this kind of regularity-loss property can be written as
\begin{equation*}
    |\hat{U}(t,k)|\leq C e^{-\phi(k) t}|\hat{U}_0(k)|
\end{equation*}
for all $t$ and $k$, where $\phi(k)$ is a strictly positive, continuous and real-valued function satisfying
$\phi(k)\to 0$ as $|k|\to 0$ and $\infty$ both with the explicit polynomial rate in $|k|$, see Lemma \ref{lem.glidecay}. The regularity-loss of the system corresponds to the property of $\phi(k)\to 0$ as $|k|\to \infty$, which is in general not true for the hyperbolic-parabolic system under the Shizuta-Kawashima condition. Therefore,  the solution over the high frequency domain does not decay in time exponentially any longer but decays polynomially with any time rate depending on the regularity of initial data.
Some recent studies of concrete models  also have shown the similar phenomenon. Among them we mention the relaxed Euler-Maxwell system in plasma physics \cite{UWK}, the Timoshenko system in biology \cite{IHK,IK} and the hyperbolic-elliptic system in radiative hydrodynamics \cite{HK,KuKa,DRZ}. Notice that  \cite{UWK} considered the case of the nonzero constant background magnetic field.

\subsection{Notations and arrangement of the paper}

Through this paper, $C$ denotes a generic positive (generally large)
constant, $c$ denotes a generic positive (generally small)
constant, and $O(1)$ denotes a generic strictly positive constant. For two quantities $A$ and $B$, $A\lesssim B$ denotes $A\leq CB$ and $A\sim B$ means $cB\leq A\leq CB$. For $u=u(x)$, the Fourier transform $\hat{u}=\hat{u}(k)$ of $u$ is defined as
\begin{equation*}
\hat{u}(k)=\CF u(k)=\int_{\R^3}e^{-ik\cdot x}u(x)dx.
\end{equation*}
For an integer $m\geq 0$, $H^m$ denotes the
Sobolev space $H^m(\R^3)$ with the usual norm $\|\cdot\|_{H^m}$.
$L^2=H^0$ with norm $\|\cdot\|$ is set when $m=0$, and
$\langle\cdot,\cdot\rangle$ denotes the inner product in $L^2$. For a multi-index
$\al=(\al_1,\al_2,\al_3)$, we denote
$\pa^{\al}=\pa_{x_1}^{\al_1}\pa_{x_2}^{\al_2}\pa_{x_3}^{\al_3}$,
and the length of $\al$ is $|\al|=\al_1+\al_2+\al_3$. Finally, $\lag
\na \rag$ is defined in terms of the Fourier transform as
$\CF \lag \na \rag=\lag \xi\rag:=(1+|\xi|^2)^{1/2}$.

The rest of the paper is arranged as follows. In Section \ref{sec2}, we study the dynamics of the linearized homogeneous Navier-Stokes-Maxwell system by proving some Lyapunov inequalities, and use the Lyapunov property of the system in terms of the Fourier transform to obtain some elementary time decay estimates on the linearized solution operator.

In Section \ref{sec3}, we explicitly solve the Green's function in the Fourier space. There, the key point is to analyze the equations of the electromagnetic field and shift the difficulty to solve a third-order ordinary equation. For the third-order ordinary equation,  we study in Section \ref{sec4} some basic properties and the explicit representation of its roots, and obtain the asymptotic behavior of roots both at the origin and infinity. In Section \ref{sec5}, we make a careful analysis of the pointwise bounds of the Green's function in the Fourier space, and use them to prove the refined time decay estimates on each component of the solution to the linearized system.

In Section \ref{sec6}, we devote ourselves to the proof of the global existence of solutions to the nonlinear Cauchy problem, and apply the obtained linearized estimates to the nonlinear case in order to prove the large time behavior of solutions and hence complete the proof of the main result Theorem \ref{thm.main}.

In Appendix, we provide a derivation of the original Navier-Stokes-Maxwell system from the kinetic Vlasov-Maxwell-Boltzmann system by using the Liu-Yang-Yu's macro-micro decomposition.

\section{The linear homogeneous dynamics}\label{sec2}

The linearized homogeneous system corresponding to \eqref{eq} and \eqref{eq.2} reads
\begin{equation}\label{leq}
    \left\{\begin{array}{l}
     \dis \pa_t n +\ga \na\cdot u =0,\\
      \dis \pa_t u+\ga \na n +\be {E}-\mu\De u=0\\
       \dis \pa_t {E} - \na \times {B} -\be u = 0,\\
      \dis \pa_t {B} + \na \times {E} = 0,\\
      \dis \na \cdot{E} = -\frac{\be}{\ga}n, \ \ \na\cdot {B} =0,\ \ t\geq 0,\ x\in \R^3.
    \end{array}\right.
\end{equation}
Here and in the sequel, for simplicity, we use the same notations $n,u,E,B$ as in \eqref{NSM} to denote the unknown functions. In terms of the Fourier transform in $x$ variable, the above system is equivalent with
\begin{equation}\label{leq.f}
    \left\{\begin{array}{l}
     \dis \pa_t \hat{n} +\ga i k\cdot \hat{u} =0,\\
      \dis \pa_t \hat{u}+\ga ik \hat{n} +\be \hat{E}+\mu |k|^2 \hat{u}=0\\
       \dis \pa_t \hat{E} - ik \times \hat{B} -\be \hat{u} = 0,\\
      \dis \pa_t \hat{B} + ik \times \hat{E} = 0,\\
      \dis ik \cdot \hat{E} = -\frac{\be}{\ga}\hat{n}, \ \ k\cdot \hat{B} =0,\ \ t\geq 0, \ k\in \R^3.
    \end{array}\right.
\end{equation}
In this section, we make some elementary Fourier analysis as in \cite{Du-EM,DS-VMB,Du-1VMB} to exploit the dissipative structure of the system \eqref{leq} or equivalently \eqref{leq.f} by showing that the system has the Lyapunov property both in the Fourier and physical spaces. The main idea is to construct the time-frequency functionals and the induced instant energy functionals that can capture the full dissipation of the system, particularly the dissipation of all the non-diffusion components in the solution. Moreover, basing on the obtained time-frequency Lyapunov inequality, we use the Hausdorff-Young inequality to study the time decay property of solutions in the standard way as in \cite{Ka2}. Those time decay estimates will be refined by further investigation of the Green's function of the linearized system in the next sections which are  the main parts of this paper.

\subsection{Lyapunov property}


First, this property is established with respect to the solution in the Fourier space in the following theorem. Throughout this subsection, we use $(a\mid b)$ to denote the complex dot product of any two complex numbers or vectors $a$ and $b$.

\begin{theorem}\label{thm.ly}
Let $U:=[n,u,E,B]$ satisfy the system \eqref{leq} for all $t\geq 0$, $x\in \R^3$. Then, there is a time-frequency functional $\CE(\cdot)$ such that
\begin{equation}\label{thm.ly.1}
    \CE(\hat{U}(t,k))\sim |\hat{U}(t,k)|^2:=|\hat{n}(t,k)|^2+|\hat{u}(t,k)|^2+|\hat{E}(t,k)|^2+|\hat{B}(t,k)|^2
\end{equation}
and
\begin{multline}\label{thm.ly.2}
\dis \pa_t  \CE(\hat{U}(t,k))+c|\hat{n}(t,k)|^2+c|k|^2 |\hat{u}(t,k)|^2 \\ +c\frac{|k|^2}{(1+|k|^2)^2}|\hat{E}(t,k)|^2+c\frac{|k|^4}{(1+|k|^2)^3}|\hat{B}(t,k)|^2 \leq 0
\end{multline}
for all $t\geq 0$, $k\in \R^3$. In particular,
\begin{equation}\label{thm.ly.3}
   \pa_t  \CE(\hat{U}(t,k))+c \frac{|k|^4}{(1+|k|^2)^3} \CE(\hat{U}(t,k))\leq 0
\end{equation}
for all $t\geq 0$, $k\in \R^3$.
\end{theorem}

\begin{proof}
The Fourier transform representation \eqref{leq.f} of system \eqref{leq} can be written as
\begin{multline*}
\pa_t \left(\begin{array}{c}
        \hat{n}\\
        \hat{u}\\
        \hat{E}\\
        \hat{B}
      \end{array}\right)
+\left(\begin{array}{cccc}
        0 & \ga ik^T & 0 & 0\\
        \ga i k & 0 & 0 & 0\\
        0 & 0 & 0 & -ik\times \\
        0 & 0 & ik\times & 0
      \end{array}\right)
\left(\begin{array}{c}
        \hat{n}\\
        \hat{u}\\
        \hat{E}\\
        \hat{B}
      \end{array}\right)\\
+\left(\begin{array}{cccc}
        0 & 0 & 0 & 0\\
        0 & 0 & \be & 0\\
        0 & -\be & 0 &0 \\
        0 & 0 & 0 & 0
      \end{array}\right)
\left(\begin{array}{c}
        \hat{n}\\
        \hat{u}\\
        \hat{E}\\
        \hat{B}
      \end{array}\right)
+\left(\begin{array}{cccc}
        0 & 0 & 0 & 0\\
        0 & \mu|k|^2 & 0 & 0\\
        0 & 0 & 0 &0 \\
        0 & 0 & 0 & 0
      \end{array}\right)
\left(\begin{array}{c}
        \hat{n}\\
        \hat{u}\\
        \hat{E}\\
        \hat{B}
      \end{array}\right)=0.
\end{multline*}
Since both the coefficient matrices of the first-order hyperbolic part and the zero-order relaxation part are skew-symmetric, the direct computation gives
\begin{equation}\label{thm.ly.p00}
    \frac{1}{2}\pa_t |[\hat{n},\hat{u},\hat{E},\hat{B}]|^2+\mu |k|^2 |\hat{u}|^2=0.
\end{equation}
From the identity
\begin{equation*}
    (i\ga k \hat{n}\mid \ga ik \hat{n} +\be \hat{E})=(\ga^2 |k|^2+\be^2)|\hat{n}|^2,
\end{equation*}
and also using the second and first equations of \eqref{leq.f} to deduce
\begin{eqnarray*}
  (i\ga k \hat{n}\mid \ga ik \hat{n} +\be \hat{E}) &=&  (i\ga k \hat{n}\mid -\pa_t \hat{u}-\mu |k|^2 \hat{u})\\
  &=&-\pa_t (i\ga k \hat{n}\mid  \hat{u})+(i\ga k \pa_t\hat{n}\mid  \hat{u})-\ga\mu (i k \hat{n}\mid|k|^2 \hat{u})\\
   &=&-\pa_t (i\ga k \hat{n}\mid  \hat{u})+\ga^2|k\cdot \hat{u}|^2-\ga\mu (i k \hat{n}\mid|k|^2 \hat{u}),
\end{eqnarray*}
one has
\begin{equation*}
  \pa_t (i\ga k \hat{n}\mid  \hat{u})+ (\ga^2 |k|^2+\be^2)|\hat{n}|^2= \ga^2|k\cdot \hat{u}|^2-\ga\mu (i k \hat{n}\mid|k|^2 \hat{u}),
\end{equation*}
which  after further taking the real part, using the Cauchy inequality for the last two terms and then dividing by $1+|k|^2$, implies
\begin{equation}\label{thm.ly.p01}
  \pa_t \frac{\MRk (i\ga k \hat{n}\mid  \hat{u})}{1+|k|^2}+ c|\hat{n}|^2\leq C|k|^2|\hat{u}|^2.
\end{equation}
In the similar way, from the identity
\begin{equation*}
    (\hat{E}\mid \ga ik \hat{n} +\be \hat{E})=\be|\hat{E}|^2+\frac{\ga^2}{\be}|k\cdot \hat{E}|^2,
\end{equation*}
and also using the second and third equations  of \eqref{leq.f} to derive
\begin{eqnarray*}
  (\hat{E}\mid \ga ik \hat{n} +\be \hat{E}) &=& (\hat{E}\mid  -\pa_t \hat{u}-\mu |k|^2 \hat{u})\\
 &=& -\pa_t(\hat{E}\mid \hat{u})+(\pa_t\hat{E}\mid \hat{u})-\mu|k|^2(\hat{E}\mid \hat{u})\\
  &=& -\pa_t(\hat{E}\mid \hat{u})+(ik \times \hat{B} +\be \hat{u}\mid \hat{u})-\mu|k|^2(\hat{E}\mid \hat{u}),
\end{eqnarray*}
one has
\begin{equation}\label{thm.ly.p1}
 \pa_t(\hat{E}\mid \hat{u})+ \be|\hat{E}|^2+\frac{\ga^2}{\be}|k\cdot \hat{E}|^2=\be|\hat{u}|^2
 +(ik \times \hat{B}\mid \hat{u})-\mu|k|^2(\hat{E}\mid \hat{u}).
\end{equation}
Moreover, from the fourth and third equations of \eqref{leq.f},
\begin{eqnarray*}
  |k|^2|\hat{B}|^2 &=& |k\times \hat{B}|^2\\
  &=&(ik\times \hat{B}\mid ik\times \hat{B}) \\
  &=&(ik\times \hat{B}\mid \pa_t \hat{E} -\be \hat{u} )\\
  &=&\pa_t (ik\times \hat{B}\mid \hat{E} )-(ik\times \pa_t\hat{B}\mid \hat{E} )-\be (ik\times \hat{B}\mid \hat{u} )\\
  &=&\pa_t (ik\times \hat{B}\mid \hat{E} )+|k\times \hat{E}|^2-\be (ik\times \hat{B}\mid \hat{u} ),
\end{eqnarray*}
that is,
\begin{equation}\label{thm.ly.p2}
  -\pa_t (ik\times \hat{B}\mid \hat{E} )  +|k|^2|\hat{B}|^2=|k\times \hat{E}|^2-\be (ik\times \hat{B}\mid \hat{u} ).
\end{equation}
We go on making estimates from the obtained identities \eqref{thm.ly.p1} and \eqref{thm.ly.p2}.  By taking the real part of \eqref{thm.ly.p1}, using the Cauchy inequality for its right-hand last term and then multiplying it by $|k|^2/(1+|k|^2)^2$, it follows
\begin{multline}\label{thm.ly.p3}
    \pa_t \frac{|k|^2\MRk (\hat{E}\mid \hat{u}) }{(1+|k|^2)^2}+c \frac{|k|^2}{(1+|k|^2)^2}|\hat{E}|^2\\
\leq \frac{\be |k|^2}{(1+|k|^2)^2}|\hat{u}|^2 +\frac{C |k|^6}{(1+|k|^2)^2}|\hat{u}|^2+ \frac{|k|^2\MRk (ik \times \hat{B}\mid \hat{u})}{(1+|k|^2)^2}\\
\leq C|k|^2|\hat{u}|^2+ \frac{|k|^2\MRk (ik \times \hat{B}\mid \hat{u})}{(1+|k|^2)^2}.
\end{multline}
On the other hand, from taking the real part of  \eqref{thm.ly.p2}, using the Cauchy inequality for its right-hand last term and then multiplying it by $|k|^2/(1+|k|^2)^3$, it follows
\begin{multline}\label{thm.ly.p4}
  \pa_t \frac{|k|^2\MRk (-ik\times \hat{B}\mid \hat{E} )}{(1+|k|^2)^3} +c \frac{|k|^4}{(1+|k|^2)^3}|\hat{B}|^2\\
  \leq  \frac{|k|^2}{(1+|k|^2)^3}|k\times \hat{E}|^2+C\frac{|k|^2}{(1+|k|^2)^3}|\hat{u}|^2
  \leq  \frac{|k|^2}{(1+|k|^2)^2}|\hat{E}|^2+C|k|^2|\hat{u}|^2.
\end{multline}
Define
\begin{multline}\label{def.CEf}
    \CE(\hat{U})=|[\hat{n},\hat{u},\hat{E},\hat{B}]|^2\\
    +\kappa_1  \frac{\MRk (i\ga k \hat{n}\mid  \hat{u})}{1+|k|^2} +\kappa_1\frac{|k|^2\MRk (\hat{E}\mid \hat{u}) }{(1+|k|^2)^2}+\kappa_1\kappa_2 \frac{|k|^2\Re (-ik\times \hat{B}\mid \hat{E} )}{(1+|k|^2)^3}
\end{multline}
with $0<\kappa_1, \kappa_2\ll 1$ to be determined so as to  guarantee that $\CE(\cdot)$ is the desired time-frequency functional satisfying  \eqref{thm.ly.1} and \eqref{thm.ly.2}. Notice that as long as $0<\kappa_1, \kappa_2\ll 1$,
\begin{equation*}
    \CE(\hat{U})\sim |[\hat{n},\hat{u},\hat{E},\hat{B}]|^2,
\end{equation*}
and hence \eqref{thm.ly.1} holds true.
One can choose $\kappa_2>0$ further small enough such that the sum of \eqref{thm.ly.p3} and \eqref{thm.ly.p4} multiplied by $\kappa_2$ gives
\begin{multline}\label{thm.ly.p5}
\pa_t\left[ \frac{|k|^2\MRk (\hat{E}\mid \hat{u}) }{(1+|k|^2)^2}+\kappa_2\frac{|k|^2\MRk (-ik\times \hat{B}\mid \hat{E} )}{(1+|k|^2)^3}\right]
+(c-\kappa_2) \frac{|k|^2}{(1+|k|^2)^2}|\hat{E}|^2\\
+c\kappa_2 \frac{|k|^4}{(1+|k|^2)^3}|\hat{B}|^2
\leq \frac{|k|^2\MRk (ik \times \hat{B}\mid \hat{u})}{(1+|k|^2)^2}+C|k|^2|\hat{u}|^2.
\end{multline}
Since
\begin{multline*}
\frac{|k|^2\MRk (ik \times \hat{B}\mid \hat{u})}{(1+|k|^2)^2}\leq \eps   \frac{|k|^4}{(1+|k|^2)^3}|\hat{B}|^2+\frac{1}{4\eps}\frac{|k|^2}{1+|k|^2}|\hat{u}|^2\\
\leq \eps   \frac{|k|^4}{(1+|k|^2)^3}|\hat{B}|^2+\frac{1}{4\eps}|k|^2|\hat{u}|^2
\end{multline*}
holds for any $\eps>0$ by the Cauchy inequality, it follows from \eqref{thm.ly.p5} that
\begin{multline}\label{thm.ly.p6}
\pa_t\left[ \frac{|k|^2\MRk (\hat{E}\mid \hat{u}) }{(1+|k|^2)^2}
+\kappa_2\frac{|k|^2\MRk (-ik\times \hat{B}\mid \hat{E} )}{(1+|k|^2)^3}\right]\\
+(c-\kappa_2) \frac{|k|^2}{(1+|k|^2)^2}|\hat{E}|^2
+c\kappa_2 \frac{|k|^4}{(1+|k|^2)^3}|\hat{B}|^2
\leq \frac{C}{\kappa_2}|k|^2|\hat{u}|^2.
\end{multline}
We now let $\kappa_2$ be fixed. By multiplying \eqref{thm.ly.p01}, \eqref{thm.ly.p6} by $\kappa_1>0$ further chosen small enough and then adding them to \eqref{thm.ly.p00}, \eqref{thm.ly.2} follows.
\eqref{thm.ly.3} is the immediate consequence of \eqref{thm.ly.2}. The proof of Theorem \ref{thm.ly} is complete.
\end{proof}

By taking integration over $k\in \R^3$ for those time-frequency inequalities in Theorem \ref{thm.ly}, it is straightforward to obtain the corresponding estimates on the solution in the physical space, which describe the evolution of instant energy functionals with the optimal energy dissipation rates. Precisely, we have

\begin{corollary}\label{cor.lya}
Let $U:=[n,u,E,B]$ satisfy the system \eqref{leq} for all $t\geq 0$, $x\in \R^3$. Let the time-frequency functional $\CE(\cdot)$ be obtained in Theorem \ref{thm.ly}. Then, for any integer $j\geq 0$,
\begin{multline*}
\frac{d}{dt}\int_{\R^3}\CE(|k|^j\hat{U}(t,k))dk
+ c \|\na^j n\|^2+c\|\na^{j+1}u\|^2\\
+c\|\na^{j+1}\lng \na \rng^{-2} E \|^2+c\|\na^{j+2}\lng \na\rng^{-3} B\|^2\leq 0
\end{multline*}
for any $t\geq 0$, where
\begin{equation*}
  \int_{\R^3}\CE(|k|^j\hat{U}(t,k))dk\sim \|\na^j U(t)\|^2.
\end{equation*}
Particularly, for any integer $N\geq 3$,
\begin{multline*}
\frac{d}{dt}\sum_{j=0}^N\int_{\R^3}\CE(|k|^j\hat{U}(t,k))dk
+ c \|n\|_{H^N}^2+c\|\na u\|_{H^N}^2\\
+c\|\na E \|_{H^{N-2}}^2+c\|\na^{2}B\|_{H^{N-3}}^2\leq 0
\end{multline*}
for any $t\geq 0$, where
\begin{equation*}
  \sum_{j=0}^N\int_{\R^3}\CE(|k|^j\hat{U}(t,k))dk\sim \|U(t)\|_{H^N}^2.
\end{equation*}
\end{corollary}

Theorem \ref{thm.ly} and Corollary \ref{cor.lya} exactly show the dissipative structure of the linearized Navier-Stokes-Maxwell system \eqref{leq}. The importance of this structure is that it exposes the dissipative property of the hyperbolic Maxwell equations in the full system, and also it plays a key role in the study of the nonlinear asymptotic stability of the constant steady state under small perturbations; see Theorem \ref{thm.ge} and its proof.


\subsection{Time-decay property}

As in \cite{Du-EM,DS-VMB,Du-1VMB}, it is now a standard procedure to derive from Theorem \ref{thm.ly} the $L^p$-$L^q$ time decay property of solutions to the linearized system \eqref{leq}. Here, we consider it by obtaining the following lemma under a more general situation as in \cite{Du-1VMB}, and also give the proof of the lemma both for completeness and for the later use in the proof of Theorem \ref{thm.tdp}.   To do that, for $\ell\in \R$ and $1\leq s\leq 2\leq q\leq \infty$, define
\begin{equation*}
    m(\ell,s,q)=
\left\{\begin{array}{ll}
  0 &\ \ \ \text{if}\ \ \ell+3(\frac{1}{s}-\frac{1}{q})<0,\\[2mm]
\ell &\ \ \ \text{if}\ \  \ell+3(\frac{1}{s}-\frac{1}{q})\geq 0, s=q=2,\\[2mm]
&\quad\quad\ \ \ \ \ \ \ \ \text{and}\ \ell\ \text{is integer},\\[2mm]
{[}\ell+3(\frac{1}{s}-\frac{1}{q})]+1&\ \ \ \text{otherwise},
\end{array}\right.
\end{equation*}
where $[\cdot]$ means the integer part of the nonnegative argument.

\begin{lemma}\label{lem.glidecay}
Assume that for any initial data $U_0$, the linear homogeneous
solution $U(t)=\semiG(t)U_0$ obeys the pointwise estimate
\begin{equation*}
    |\hat{U}(t,k)|\leq C e^{-\phi(k) t}|\hat{U}_0(k)|
\end{equation*}
for all $t\geq 0$, $k\in \R^3$, where $\phi(k)$ is a strictly positive, continuous and real-valued function over $k\in \R^3$ and satisfies
\begin{equation*}
    \phi(k)= \left\{
    \begin{array}{ll}
      O(1)|k|^{\si_+} & \ \ \text{as $|k|\to 0$},\\[3mm]
      O(1)|k|^{-\si_-} & \ \ \text{as $|k|\to \infty$},
    \end{array}\right.
\end{equation*}
for two constants $\si_->\si_+>0$. Let $j\geq 0$ be an integer, $1\leq p,s\leq 2\leq q\leq \infty$ and $\ell\geq 0$. Then, $U(t)=\semiG(t)U_0$ obeys the time-decay estimate
\begin{multline}\label{lem.glidecay.3}
 \|\na^j U(t)\|_{L^q}
\leq
C(1+t)^{-\frac{3}{\si_+}(\frac{1}{p}-\frac{1}{q})-\frac{j}{\si_+}}
\|U_0\|_{L^p}\\
 +C(1+t)^{-\frac{\ell}{\si_-}} \|\na^{m(j+\ell,s,q)}U_0\|_{L^s},
\end{multline}
for any $t\geq 0$.
\end{lemma}

\begin{proof}
Take a constant $L>0$. From the assumptions on $\phi(k)$, it is easy to see
\begin{equation*}
    \phi(k)\geq \left\{
    \begin{array}{ll}
      c |k|^{\si_+} & \ \ \text{if $|k|\leq L$},\\[3mm]
      c |k|^{-\si_-} & \ \ \text{if $|k|\geq L$}.
    \end{array}\right.
\end{equation*}
Take $2\leq q\leq \infty$ and an integer $j\geq 0$. By Hausdorff-Young inequality,
\begin{multline}\label{lem.glidecay.p1}
\|\na_x^j U(t)\|_{L^q}\leq C\left\||k|^j e^{-\phi(k)t}\hat{U}_0\right\|_{L^{q'}}\\
\leq C\left\||k|^j e^{-\la |k|^{\si_+}t}\hat{U}_0\right\|_{L^{q'}(|k|\leq L)}
+C \left\||k|^j e^{-\la |k|^{-\si_-}t}\hat{U}_0\right\|_{L^{q'}(|k|\geq L)}:=I_1+I_2,
\end{multline}
where $\frac{1}{q}+\frac{1}{q'}=1$. For $I_1$,
take $1\leq p\leq 2$ and by using the H\"{o}lder inequality for
$\frac{1}{q'}=\frac{p'-q'}{p'q'}+\frac{1}{p'}$ with $p'$ given by
$\frac{1}{p}+\frac{1}{p'}=1$,
\begin{equation*}
    I_1\leq C\left\||k|^j e^{-\la |k|^{\si_+}t}\right\|_{L^{\frac{p'q'}{p'-q'}}(|k|\leq L)}
    \|\hat{U}_0\|_{L^{p'}(|k|\leq
    L)}.
\end{equation*}
Here, the right-hand first factor can be estimated in a standard way \cite{Ka2}
as
\begin{equation*}
\left\||k|^j e^{-\la
|k|^{\si_+}t}\right\|_{L^{\frac{p'q'}{p'-q'}}(|k|\leq L)}\leq C
(1+t)^{-\frac{3}{\si_+}(\frac{1}{p}-\frac{1}{q})-\frac{j}{\si_+}}
\end{equation*}
by using change of variable $kt^{\frac{1}{\si_+}}\to k$, and the
right-hand second factor is estimated by Hausdorff-Young
inequalities as
\begin{equation*}
\|\hat{U}_0\|_{L^{p'}(|k|\leq L)}
\leq\|\hat{U}_0\|_{L^{p'}}
\leq C\|U_0\|_{L^p}.
\end{equation*}
Therefore, for $I_1$, one has
\begin{equation*}
I_1\leq
C(1+t)^{-\frac{3}{\si_+}(\frac{1}{p}-\frac{1}{q})-\frac{j}{\si_+}}\|U_0\|_{L^p}.
\end{equation*}
To estimate $I_2$, take a constant $\ell\geq 0$ so that
\begin{equation*}
    I_2=C \left\||k|^j e^{-\la |k|^{-\si_-}t}\hat{U}_0\right\|_{L^{q'}(|k|\geq L)}
    \leq C\sup_{|k|\geq L}|k|^{-\ell} e^{-c |k|^{-\si_-}t}
    \left\||k|^{j+\ell} \hat{U}_0\right\|_{L^{q'}(|k|\geq L)}.
\end{equation*}
Here, the right-hand first factor decays in time as
\begin{equation*}
\sup_{|k|\geq L}|k|^{-\ell} e^{-c |k|^{-\si_-}t}\leq
C(1+t)^{-\frac{\ell}{\si_-}}.
\end{equation*}
We estimate the right-hand second factor as follows. Take $1\leq
s\leq 2$ with $\frac{1}{s}+\frac{1}{s'}=1$ and take a constant
$\eps>0$ small enough. Then, similarly as before, from
H\"{o}lder inequality for
$\frac{1}{q'}=\frac{s'-q'}{s'q'}+\frac{1}{s'}$, one has
\begin{multline*}
 \left\||k|^{j+\ell} \hat{U}_0\right\|_{L^{q'}(|k|\geq L)}\\
  \leq \left\||k|^{-3(1+\eps)\frac{s'-q'}{s'q'}}\right\|_{L^{\frac{s'q'}{s'-q'}}(|k|\geq
  L)}\left\||k|^{j+\ell+3(1+\eps)\frac{s'-q'}{s'q'}}
  \hat{U}_0\right\|_{{L^{s'}(|k|\geq
  L)}}\\
  \leq C_\eps\left\||k|^{j+[\ell+3(\frac{1}{s}-\frac{1}{q})]_+}
 \hat{U}_0\right\|_{{L^{s'}(|k|\geq
  L)}}.
\end{multline*}
Here, due to $s'\geq 2$, further by Hausdorff-Young inequality,
\begin{multline*}
\left\||k|^{j+[\ell+3(\frac{1}{s}-\frac{1}{q})]_+}
 \hat{U}_0\right\|_{{L^{s'}(|k|\geq
  L)}}\leq \left\||k|^{j+[\ell+3(\frac{1}{s}-\frac{1}{q})]_+}
 \hat{U}_0\right\|_{{L^{s'}}}\\
\leq
  C\|\na_x^{j+[\ell+3(\frac{1}{s}-\frac{1}{q})]_+}U_0\|_{L^s}.
\end{multline*}
Then, it follows that
\begin{equation*}
\left\||k|^{j+\ell} \hat{U}_0\right\|_{L^{q'}(|k|\geq L)}\leq
C\|\na_x^{j+[\ell+3(\frac{1}{s}-\frac{1}{q})]_+}U_0\|_{L^s}.
\end{equation*}
Thus, $I_2$ is estimated by
\begin{equation*}
    I_2\leq
    C(1+t)^{-\frac{\ell}{\si_-}}\|\na_x^{j+[\ell+3(\frac{1}{s}-\frac{1}{q})]_+}U_0\|_{L^s}.
\end{equation*}
Now, \eqref{lem.glidecay.3} follows by plugging the estimates of
$I_1$ and $I_2$ into \eqref{lem.glidecay.p1}. This completes the
proof of Lemma \ref{lem.glidecay}.
\end{proof}

By applying Theorem \ref{thm.ly} together with  Lemma \ref{lem.glidecay} above  to the linearized system \eqref{leq}, we have

\begin{theorem}\label{thm.lr}
Let $1\leq p, s\leq 2\leq q\leq \infty$,  $\ell\geq 0$, and let $j\geq 0$ be an integer. Let $U=[n,u,E,B]$ satisfy the system \eqref{leq} for all $t>0$, $x\in \R^3$ with initial data $U|_{t=0}=U_0$.
Then, the solution of the
linearized homogeneous system satisfies
\begin{equation*}
 \|\na^j U(t)\|_{L^q}
\leq C(1+t)^{-\frac{3}{4}(\frac{1}{p}-\frac{1}{q})-\frac{j}{4}}
\|U_0\|_{L^p}
 +C(1+t)^{-\frac{\ell}{2}} \|\na^{m(j+\ell,s,q)}U_0\|_{L^s}
\end{equation*}
for any $t\geq 0$.
\end{theorem}

For the linearized system \eqref{leq}, Theorem \ref{thm.lr} describes the unified time decay property of the full solution. But, it can not be directly applied to the nonlinear system \eqref{eq}-\eqref{eq.2} to obtain the time decay rates of the solution, particularly those rates for each component of the solution as stated in Theorem \ref{thm.main}. Therefore, we turn to the study of the Green's function of the linearized system.

\section{Green's function}\label{sec3}

In this section we explicitly solve the Cauchy problem on the linearized Navier-Stokes-Maxwell system \eqref{leq} or equivalently its Fourier transform \eqref{leq.f} with initial data
\begin{equation*}
    U|_{t=0}=U_0=[n_0,u_0,E_0,B_0],\ \ x\in \R^3,
\end{equation*}
which satisfies the last equation of  \eqref{leq}. In fact,  as in \cite{Du-EM}, \eqref{leq} can be written as two decoupled subsystems which govern the time evolution of $n$, $\na\cdot u$, $\na\cdot E$ and $\na\times u$, $\na\times E$, $\na\times B$, respectively. For simplicity, we call two subsystems by two systems for the fluid part and the electromagnetic part, respectively, even though only $n$ and $B$ can be fully decoupled. By solving some ordinary differential equations, we shall deduce the representation of the Green's function both for the fluid part and the electromagnetic part.

\subsection{Green's function of the fluid part}

We first solve $\hat{n}$. From \eqref{leq.f}, by taking the time derivative of the first equation and using the second equation to replace $\pa_t\hat{u}$, one has
\begin{equation*}
  \pa_t^2 \hat{n} +\ga i k\cdot (-\ga ik \hat{n} -\be \hat{E}-\mu |k|^2 \hat{u}) =0.
\end{equation*}
Further by using $ ik \cdot \hat{E} = -\frac{\be}{\ga}\hat{n}$ and $\pa_t \hat{n} +\ga i k\cdot \hat{u} =0$, the above is equivalent with
\begin{equation*}
   \pa_t^2 \hat{n}+\mu  |k|^2 \pa_t\hat{n} +(\ga^2|k|^2+\be^2) \hat{n}=0.
\end{equation*}
Thus, one has to solve the second-order ODE
\begin{equation}\label{gf.n1}
    \left\{\begin{array}{l}
       \pa_t^2 \hat{n}+\mu  |k|^2 \pa_t\hat{n} +(\ga^2|k|^2+\be^2) \hat{n}=0,\\[3mm]
        \hat{n}(0,k)=\hat{n}_0(k),\\[3mm]
        \pa_t\hat{n}(0,k)=-i\ga k\cdot \hat{u}_0(k).
    \end{array}\right.
\end{equation}
Consider the indicate equation
\begin{equation}\label{gf.n.ind}
    z^2+\mu|k|^2 z+ (\ga^2|k|^2+\be^2)=0,
\end{equation}
and its roots are denoted by
\begin{equation}\label{gf.n.ind.r}
    \la_\pm=\la_\pm(k):=-\frac{1}{2}\mu |k|^2\pm \frac{1}{2} \sqrt{\mu^2 |k|^4-4(\ga^2|k|^2+\be^2)}.
\end{equation}
Set the representation of the solution $\hat{n}$ to \eqref{gf.n1} as
\begin{equation*}
   \hat{n}(t,k)=C_-(k)e^{\la_-(k)t}+C_+(k)e^{\la_+(k)t}
\end{equation*}
except a finite number of values of $|k|$ where $\la_+=\la_-$ become a two-multiple real root.
Then, the initial conditions of  \eqref{gf.n1} imply
\begin{equation*}
    \left\{\begin{array}{l}
    C_-(k)+C_+(k)=\hat{n}_0(k),\\[3mm]
    \la_-(k)C_-(k)+\la_+(k)C_+(k)=-i\ga k\cdot \hat{u}_0(k).
    \end{array}\right.
\end{equation*}
The solution $C_\pm(k)$ to the above system can be formally written as
\begin{equation*}
    C_-(k)=\frac{\la_+(k)\hat{n}_0(k)+i\ga k\cdot \hat{u}_0(k)}{\la_+(k)-\la_-(k)},\ \
     C_+(k)=\frac{-i\ga k\cdot \hat{u}_0(k)-\la_-(k)\hat{n}_0(k)}{\la_+(k)-\la_-(k)}.
\end{equation*}
Therefore, $\hat{n}$ is solved as
\begin{equation}\label{gf.n2}
   \hat{n}(t,k)=\frac{\la_+e^{\la_- t}-\la_-e^{\la_+ t}}{\la_+-\la_-} \hat{n}_0(k)
   -\frac{e^{\la+ t}-e^{\la_- t}}{\la_+-\la_-}i\ga k\cdot \hat{u}_0(k).
\end{equation}
Denote $\tilde{k}=k/|k|$ when $k\neq 0$. From $ik \cdot \hat{E} = -\frac{\be}{\ga}\hat{n}$, one has
$\tilde{k} \cdot \hat{E} = \frac{\be}{\ga}\frac{i}{|k|}\hat{n}$. Thus, multiplying \eqref{gf.n2} by $\frac{\be}{\ga}\frac{i}{|k|}$ and computing $\frac{\be}{\ga}\frac{i}{|k|}(-i\ga k\cdot \hat{u}_0(k))=\be \tilde{k}\cdot \hat{u}_0(k)$, gives
\begin{equation*}
   \tilde{k}\cdot \hat{E}=\frac{\la_+e^{\la_- t}-\la_-e^{\la_+ t}}{\la_+-\la_-} \tilde{k}\cdot \hat{E}_0
   +\frac{e^{\la+ t}-e^{\la_- t}}{\la_+-\la_-}\be \tilde{k}\cdot \hat{u}_0(k).
\end{equation*}
To solve $\tilde{k}\cdot \hat{u}$, from \eqref{leq.f}, by taking the time derivative of the second equation and then using the first and third equations to replace $\pa_t\hat{n}$, $\pa_t \hat{E}$, respectively, one has
\begin{equation*}
    \pa_t^2 \hat{u}+\ga^2 k k\cdot \hat{u} +\be ik \times \hat{B} +\be^2 \hat{u}+\mu |k|^2 \pa_t\hat{u}=0.
\end{equation*}
Further taking the dot product with $k$ gives
\begin{equation*}
   \dis \pa_t^2 (k\cdot \hat{u})+\mu |k|^2 \pa_t(k\cdot \hat{u})+(\ga^2|k|^2+\be^2) k\cdot \hat{u}=0.
\end{equation*}
In the same way as for $\hat{n}$, one can solve the above second order ODE as
\begin{equation*}
  k\cdot \hat{u}(t,k)=\frac{\la_+e^{\la_- t}-\la_-e^{\la_+ t}}{\la_+-\la_-}  k\cdot \hat{u}(0,k)
   +\frac{e^{\la+ t}-e^{\la_- t}}{\la_+-\la_-} k\cdot \pa_t \hat{u}(0,k).
\end{equation*}
Compute from initial data
\begin{equation*}
  k\cdot \pa_t \hat{u}(0,k)=k\cdot (-\ga ik \hat{n}_0 -\be \hat{E}_0-\mu |k|^2 \hat{u}_0)=-i\ga |k|^2 \hat{n}_0-\be k\cdot  \hat{E}_0-\mu |k|^2 k\cdot \hat{u}_0.
\end{equation*}
Therefore,
\begin{eqnarray}\label{gf.u1}
  k\cdot \hat{u}(t,k)&=& \frac{e^{\la+ t}-e^{\la_- t}}{\la_+-\la_-} (-i\ga |k|^2 \hat{n}_0)\\
  &&+\left[\frac{\la_+e^{\la_- t}-\la_-e^{\la_+ t}}{\la_+-\la_-} +\frac{e^{\la+ t}-e^{\la_- t}}{\la_+-\la_-}(-\mu |k|^2)\right] k\cdot \hat{u}_0\nonumber\\
  &&+\frac{e^{\la+ t}-e^{\la_- t}}{\la_+-\la_-}(-\be k\cdot  \hat{E}_0).\nonumber
\end{eqnarray}
Notice that from the indicate equation \eqref{gf.n.ind},
\begin{equation*}
    \la_++\la_-=-\mu |k|^2,\ \ \la_+\la_-=\ga^2|k|^2+\be^2,
\end{equation*}
which imply
\begin{equation*}
 \frac{\la_+e^{\la_- t}-\la_-e^{\la_+ t}}{\la_+-\la_-} +\frac{e^{\la+ t}-e^{\la_- t}}{\la_+-\la_-}(-\mu |k|^2)=\frac{\la_+e^{\la_+ t}-\la_-e^{\la_- t}}{\la_+-\la_-}.
\end{equation*}
Applying the above identity into \eqref{gf.u1} and dividing it by $|k|$, one has
\begin{multline*}
  \tilde{k}\cdot \hat{u}(t,k)= \frac{e^{\la+ t}-e^{\la_- t}}{\la_+-\la_-} (-i\ga |k| \hat{n}_0)
+\frac{\la_+e^{\la_+ t}-\la_-e^{\la_- t}}{\la_+-\la_-} \tilde{k}\cdot \hat{u}_0\\
  +\frac{e^{\la+ t}-e^{\la_- t}}{\la_+-\la_-}(-\be \tilde{k}\cdot  \hat{E}_0).
\end{multline*}
In a word, one has proved
\begin{proposition}\label{prop.n}
Let $U=[n,u,E,B]$ satisfy the system \eqref{leq} for all $t>0$, $x\in \R^3$ with initial data $U|_{t=0}=U_0=[n_0,u_0,E_0,B_0]$. Denote $\tilde{k}=k/|k|$ when $k \neq 0$. Set $ \hat{u}_{\parallel}$ by
$\hat{u}_{\parallel}=\tilde{k}\tilde{k}\cdot\hat{u}$, and likewise for $\hat{E}_{\parallel}$. Then,
\begin{equation*}
   \left(\begin{array}{c}
     \hat{n}\\
     \hat{u}_{\parallel}\\
      \hat{E}_{\parallel}
    \end{array}\right) =\CN
    \left(\begin{array}{c}
     \hat{n}_0\\
     \hat{u}_{0,\parallel}\\
      \hat{E}_{0,\parallel}
    \end{array}\right)
\end{equation*}
except a finite number of values of $|k|$, where the matrix $\CN=\CN(t,k)_{7\times 7}$ is given by
\begin{equation*}
     \CN=
     \left(\begin{array}{ccc}
     \dis \frac{\la_+e^{\la_- t}-\la_-e^{\la_+ t}}{\la_+-\la_-} &\ \ \ \dis  -i\ga{k}^T \frac{e^{\la+ t}-e^{\la_- t}}{\la_+-\la_-} &\ \ \ \dis  0 \\[4mm]
     \dis -i\ga k\frac{e^{\la+ t}-e^{\la_- t}}{\la_+-\la_-}   &\ \ \ \dis   \frac{\la_+e^{\la_+ t}-\la_-e^{\la_- t}}{\la_+-\la_-}\FI_3 &\ \ \ \dis  -\be\frac{e^{\la+ t}-e^{\la_- t}}{\la_+-\la_-}\FI_3\\[4mm]
     \dis 0 &\ \ \ \dis  \be\frac{e^{\la+ t}-e^{\la_- t}}{\la_+-\la_-}\FI_3 &\ \ \ \dis  \frac{\la_+e^{\la_- t}-\la_-e^{\la_+ t}}{\la_+-\la_-}\FI_3
    \end{array}\right),
\end{equation*}
and $\la_\pm=\la_\pm(k)$ given by \eqref{gf.n.ind.r} denote the roots of the indicate equation \eqref{gf.n.ind}.

\end{proposition}

\subsection{Green's function of the electromagnetic part}\label{sec3.2}

By applying the operator $\tilde{k}\times$ to equations of $\pa_t\hat{u},\pa_t\hat{E},\pa_t\hat{B}$ in \eqref{leq.f}, one has
\begin{equation*}
    \left\{\begin{array}{l}
     \pa_t \tilde{k}\times \hat{u} +\be \tilde{k}\times\hat{E}+\mu |k|^2 \tilde{k}\times\hat{u}=0,\\
     \pa_t \tilde{k}\times\hat{E} - ik \times (\tilde{k}\times\hat{B}) -\be \tilde{k}\times\hat{u} = 0,\\
     \pa_t  \tilde{k}\times\hat{B} + ik \times ( \tilde{k}\times\hat{E}) = 0.
    \end{array}\right.
\end{equation*}
Define
\begin{equation*}
    \left\{\begin{array}{l}
     M_1=M_1(t,k)=\hat{u}_\perp(t,k)=-\tilde{k}\times\tilde{k}\times \hat{u},\\
      M_2=M_2(t,k)=\hat{E}_\perp(t,k)=-\tilde{k}\times\tilde{k}\times \hat{E},\\
       M_3=M_3(t,k)=\hat{B}_\perp(t,k)=-\tilde{k}\times\tilde{k}\times \hat{B}.
    \end{array}\right.
\end{equation*}
Then, $M_1, M_2, M_3$ satisfy the first-order ODE system
\begin{equation*}
    \left\{\begin{array}{l}
    \pa_t M_1 +\mu |k|^2M_1 +\be M_2=0,\\
    \pa_t M_2 -ik\times M_3 -\be M_1=0,\\
    \pa_t M_3 +ik\times M_2=0,
    \end{array}\right.
\end{equation*}
namely,
\begin{equation*}
    \pa_t\left(\begin{array}{c}
     M_1\\
     M_2\\
     M_3
    \end{array}\right)+
    \left(\begin{array}{ccc}
    \mu |k|^2 & \be & 0 \\
    -\be & 0 & -ik\times\\
    0 & ik\times & 0
    \end{array}\right)
    \left(\begin{array}{c}
     M_1\\
     M_2\\
     M_3
    \end{array}\right)=0.
\end{equation*}
To simplify the computation of solving the above system, we take change of variables
\begin{equation*}
    \xi=\frac{1}{\be}k,\ \ \tau=\be t,
\end{equation*}
and thus  $M_1, M_2, M_3$ satisfy the rescaled ODE system
\begin{equation}\label{3ode}
    \pa_t\left(\begin{array}{c}
     M_1\\
     M_2\\
     M_3
    \end{array}\right)+
    \left(\begin{array}{ccc}
    a |\xi|^2 & 1 & 0 \\
    -1 & 0 & -i\xi\times\\
    0 & i\xi \times & 0
    \end{array}\right)
    \left(\begin{array}{c}
     M_1\\
     M_2\\
     M_3
    \end{array}\right)=0
\end{equation}
with  $a=\mu\be$.

Similar in \cite{Du-EM}, one can derive from \eqref{3ode} a third order ODE satisfied by $M_2$. In fact, compute $\pa_\tau^2M_2$ by replacing $\pa_\tau M_1$, $\pa_\tau M_3$ to get
\begin{equation}\label{3ode.p1}
    \pa_\tau^2 M_2+(|\xi|^2+1)M_2=-a|\xi|^2M_1.
\end{equation}
By taking one more $\tau$ derivative of the above equation and then replacing $\pa_\tau M_1$, it follows
\begin{equation}\label{3ode.p2}
    \pa_\tau^3 M_2+(|\xi|^2+1)\pa_\tau M_2-a |\xi|^2M_2=a^2|\xi|^4M_1.
\end{equation}
Multiplying \eqref{3ode.p1} by $a|\xi|^2$ and adding it to \eqref{3ode.p2} yields
\begin{equation*}
  \pa_\tau^3 M_2+ a |\xi|^2 \pa_\tau^2 M_2 +(|\xi|^2+1)\pa_\tau M_2+a |\xi|^4M_2=0.
\end{equation*}
Then, $M_2$ can be solved as
\begin{equation*}
    M_2(\tau,\xi)=C_1e^{\la_1\tau}+C_2 e^{\la_2\tau}+C_3 e^{\la_3 \tau},
\end{equation*}
where $\la_j=\la_j(\xi)$, $j=1,2,3$, denote the roots of
the indicate equation
\begin{equation}\label{ind}
    z^3+a |\xi|^2 z^2 +(|\xi|^2+1)z+a |\xi|^4=0,\ \ z\in \C.
\end{equation}
Here and hereafter we have excluded a finite number of values of $|\xi|$ such that equation \eqref{ind} has possible multiple real roots. Moreover, from equations of $M_1$ and $M_3$ in \eqref{3ode},  $M_1(\tau,\xi)$ and  $M_3(\tau,\xi)$ can be also solved as
\begin{eqnarray}\label{sm1}
  M_1(\tau,\xi) &=& M_1(0,\tau)e^{-a|\xi|^2\tau}-\int_0^\tau  e^{-a|\xi|^2(\tau-s)}M_2(s,\xi)ds\\
&=& M_1(0,\tau)e^{-a|\xi|^2\tau}-\int_0^\tau  e^{-a|\xi|^2(\tau-s)}\sum_{j=1}^3 C_je^{\la_j s}ds\nonumber\\
&=& M_1(0,\tau)e^{-a|\xi|^2\tau}-e^{-a|\xi|^2\tau}\sum_{j=1}^3C_j\int_0^\tau  e^{(\la_j+a|\xi|^2)s}ds\nonumber\\
&=& \left( M_1(0,\tau)+\sum_{j=1}^3\frac{C_j}{\la_j+a|\xi|^2}\right)e^{-a|\xi|^2\tau}
-\sum_{j=1}^3\frac{C_j}{\la_j+a|\xi|^2}e^{\la_j\tau},\nonumber
\end{eqnarray}
and
\begin{eqnarray}\label{sm3}
  M_3(\tau,\xi) &=& M_3(0,\xi)+\int_0^\tau -i\xi\times M_2(s,\xi)ds\\
&=& M_3(0,\xi) -i\xi \times\int_0^\tau\sum_{j=1}^3 C_je^{\la_j s}ds\nonumber\\
&=& M_3(0,\xi)+i\xi\times \sum_{j=1}^3 \frac{C_j}{\la_j}-i\xi\times \sum_{j=1}^3 \frac{C_j}{\la_j}e^{\la_j\tau}.\nonumber
\end{eqnarray}

One can determine $C_j=C_j(\xi)$ in terms of $M_j(0,\xi)$ as follows. The initial condition of \eqref{3ode} implies
\begin{equation*}
    \left(\begin{array}{ccc}
      1 & 1 & 1 \\
      \la_1 & \la_2 & \la_3 \\
      \la_1^2 & \la_2^2 & \la_3^3
    \end{array}\right)
 \left(\begin{array}{c}
     C_1 \\
     C_2 \\
     C_3
    \end{array}\right)=
 \left(\begin{array}{c}
     M_2(0,\xi) \\
     \pa_\tau M_2(0,\xi) \\
      \pa_\tau^2 M_2(0,\xi)
    \end{array}\right).
\end{equation*}
On the other hand, from equations of $\pa_\tau M_2$ and $\pa_\tau^2M_2$ in \eqref{3ode} and \eqref{3ode.p1},
\begin{equation*}
 \left(\begin{array}{c}
     M_2(0,\xi) \\
     \pa_\tau M_2(0,\xi) \\
      \pa_\tau^2 M_2(0,\xi)
    \end{array}\right)
= \left(\begin{array}{ccc}
      0& 1 & 0 \\
      1 & 0 & i\xi\times \\
     -a|\xi|^2 & -(|\xi|^2+1) & 0
    \end{array}\right)
\left(\begin{array}{c}
     M_1(0,\xi) \\
    M_2(0,\xi) \\
     M_3(0,\xi)
    \end{array}\right).
\end{equation*}
Combining the above two equations gives
\begin{equation*}
 \left(\begin{array}{c}
     C_1 \\
     C_2 \\
     C_3
    \end{array}\right)=
 \left(\begin{array}{ccc}
      1 & 1 & 1 \\
      \la_1 & \la_2 & \la_3 \\
      \la_1^2 & \la_2^2 & \la_3^3
    \end{array}\right)^{-1}
\left(\begin{array}{ccc}
      0& 1 & 0 \\
      1 & 0 & i\xi\times \\
     -a|\xi|^2 & -(|\xi|^2+1) & 0
    \end{array}\right)
\left(\begin{array}{c}
     M_1(0,\xi) \\
    M_2(0,\xi) \\
     M_3(0,\xi)
    \end{array}\right).
\end{equation*}
One can compute
\begin{multline*}
  \left(\begin{array}{ccc}
      1 & 1 & 1 \\
      \la_1 & \la_2 & \la_3 \\
      \la_1^2 & \la_2^2 & \la_3^3
    \end{array}\right)^{-1}=\frac{1}{(\la_1-\la_2)(\la_2-\la_3)(\la_3-\la_1)}\\
   \left(\begin{array}{ccc}
      -\la_2\la_3(\la_2-\la_3) &\ \ \ (\la_2+\la_3)(\la_2-\la_3) &\ \ \ -(\la_2-\la_3) \\
   -\la_3\la_1(\la_3-\la_1) &\ \ \ (\la_3+\la_1)(\la_3-\la_1) &\ \ \ -(\la_3-\la_1) \\
     -\la_1\la_2(\la_1-\la_2) &\ \ \ (\la_1+\la_2)(\la_1-\la_2) &\ \ \ -(\la_1-\la_2)
    \end{array}\right).
\end{multline*}
Thus,
\begin{equation}\label{form.C}
 \left(\begin{array}{c}
     C_1 \\
     C_2 \\
     C_3
    \end{array}\right)=\CC  \left(\begin{array}{c}
     M_1(0,\xi) \\
    M_2(0,\xi) \\
     M_3(0,\xi)
    \end{array}\right)
\end{equation}
with $\CC$ given by
\begin{multline*}
  \CC =\frac{1}{(\la_1-\la_2)(\la_2-\la_3)(\la_3-\la_1)}\\
    \left(\begin{array}{ccc}
 (\la_2+\la_3+a|\xi|^2)(\la_2-\la_3) &\ \ \ (|\xi|^2+1-\la_2\la_3)(\la_2-\la_3) &\ \ \ (\la_2^2-\la_3^2)i\xi\times \\
 (\la_3+\la_1+a|\xi|^2)(\la_3-\la_1) &\ \ \ (|\xi|^2+1-\la_3\la_1)(\la_3-\la_1) &\ \ \ (\la_3^2-\la_1^2)i\xi\times \\
 (\la_1+\la_2+a|\xi|^2)(\la_1-\la_2) &\ \ \ (|\xi|^2+1-\la_1\la_2)(\la_1-\la_2) &\ \ \ (\la_1^2-\la_2^2)i\xi\times
    \end{array}\right).
\end{multline*}
Notice that
\begin{equation*}
    \prod_{j=1}^3(z-\la_j)=z^3-(\la_1+\la_2+\la_3)z^2+(\la_1\la_2+\la_2\la_3+\la_3\la_1)z-\la_1\la_2\la_3.
\end{equation*}
Comparing the above with the indicate equation \eqref{ind} shows
\begin{eqnarray*}
&&\la_1+\la_2+\la_3=-a|\xi|^2,\\
&&\la_1\la_2+\la_2\la_3+\la_3\la_1=|\xi|^2+1,\\
&&\la_1\la_2\la_3=-a|\xi|^4.
\end{eqnarray*}
Then, one can rewrite $\CC$ as
\begin{multline}\label{def.C}
  \CC =\frac{1}{(\la_1-\la_2)(\la_2-\la_3)(\la_3-\la_1)}\\
   \left(\begin{array}{ccc}
 -\la_1(\la_2-\la_3) &\ \ \ \la_1(\la_2^2-\la_3^2) &\ \ \ (\la_2^2-\la_3^2)i\xi\times \\[3mm]
  -\la_2(\la_3-\la_1) &\ \ \ \la_2(\la_3^2-\la_1^2) &\ \ \ (\la_3^2-\la_1^2)i\xi\times \\[3mm]
 -\la_3(\la_1-\la_2) &\ \ \ \la_3(\la_1^2-\la_2^2) &\ \ \ (\la_1^2-\la_2^2)i\xi\times
    \end{array}\right).
\end{multline}
Basing on the formula \eqref{form.C} of the coefficients $C_i(\xi)$ in terms of initial data $M_i(0,\xi)$, it is a direct and long computation to verify the following lemma, and its proof is omitted for simplicity.

\begin{lemma}
\begin{eqnarray*}
&& M_3(0,\xi)+i\xi\times \sum_{j=1}^3 \frac{C_j}{\la_j}=0,\\
&& M_1(0,\tau)+\sum_{j=1}^3\frac{C_j}{\la_j+a|\xi|^2}=0.
\end{eqnarray*}
\end{lemma}

From the above lemma, $M_1(\tau,\xi)$ and $M_3(\tau,\xi)$ respectively  given in \eqref{sm1} and \eqref{sm3} are reduced to
\begin{eqnarray*}
 M_1(\tau,\xi) &=&
-\sum_{j=1}^3\frac{C_j}{\la_j+a|\xi|^2}e^{\la_j\tau}=\sum_{j=1}^3\frac{C_j}{\sum_{i\neq j}\la_i}e^{\la_j\tau},\\
  M_3(\tau,\xi) &=&-i\xi\times \sum_{j=1}^3 \frac{C_j}{\la_j}e^{\la_j\tau}.
\end{eqnarray*}
Thus, we have solved $M_j(\tau,\xi)$, $j=1,2,3,$ as
\begin{equation*}
\left(\begin{array}{c}
     M_1(\tau,\xi) \\
      M_2(\tau,\xi) \\
     M_3(\tau,\xi)
    \end{array}\right)=\CB
 \left(\begin{array}{c}
     C_1 \\
     C_2 \\
     C_3
    \end{array}\right)
\end{equation*}
with $\CB$ defined by
\begin{equation}\label{def.B}
    \CB=\left(\begin{array}{ccc}
     \frac{e^{\la_1\tau}}{\la_2+\la_3} &\ \ \ \frac{e^{\la_2\tau}}{\la_3+\la_1} &\ \ \frac{e^{\la_3\tau}}{\la_1+\la_2} \\[3mm]
  e^{\la_1\tau} &\ \  e^{\la_2\tau} &\ \  e^{\la_3\tau} \\[3mm]
      -i\xi\times \frac{e^{\la_1\tau}}{\la_1} &\ \ -i\xi\times \frac{e^{\la_2\tau}}{\la_2}&\ \ -i\xi\times \frac{e^{\la_3\tau}}{\la_3}
    \end{array}\right).
\end{equation}
Set
\begin{equation*}
    \CM=\CB\CC.
\end{equation*}
Then,
\begin{equation*}
\left(\begin{array}{c}
     M_1(\tau,\xi) \\
      M_2(\tau,\xi) \\
     M_3(\tau,\xi)
    \end{array}\right)=\CM\left(\begin{array}{c}
     M_1(0,\xi) \\
    M_2(0,\xi) \\
     M_3(0,\xi)
    \end{array}\right) .
\end{equation*}
In a summary, one has proved

\begin{proposition}\label{prop.m}
Let $U=[n,u,E,B]$ satisfy the system \eqref{leq} for all $t>0$, $x\in \R^3$ with initial data $U|_{t=0}=U_0=[n_0,u_0,E_0,B_0]$. Denote $\tilde{k}=k/|k|$ when $k \neq 0$. Set
\begin{equation*}
    \hat{u}_\perp=-\tilde{k}\times (\tilde{k}\times \hat{u})=(\FI_3-\tilde{k}\otimes \tilde{k})\hat{u},
\end{equation*}
and likewise for $\hat{E}_{\perp}$ and $\hat{B}_{\perp}$. Then,
\begin{equation*}
   \left(\begin{array}{c}
      \hat{u}_\perp\\
      \hat{E}_\perp\\
      \hat{B}_\perp
    \end{array}\right) =\CM(\be t,\frac{k}{\be})
    \left(\begin{array}{c}
     \hat{u}_{0,\perp}\\
      \hat{E}_{0,\perp}\\
      \hat{B}_{0,\perp}
    \end{array}\right)
\end{equation*}
except a finite number of values of $|k|$.
Here, in terms of variables $\tau=\be t$ and $\xi=\frac{1}{\be} k$, the matrix $\CM=\CM(\tau,\xi)_{9\times 9}$ is equal to $\CB\CC$ with $\CB$, $\CC$ defined in \eqref{def.B} and \eqref{def.C}, respectively.
Moreover, the elements of $\CM(\tau,\xi)$ take the form of
\begin{eqnarray*}
  \CM_{11} &=& \frac{e^{\la_1\tau}(-\la_1)\frac{\la_2-\la_3}{\la_2+\la_3}
+e^{\la_2\tau}(-\la_2)\frac{\la_3-\la_1}{\la_3+\la_1}
+e^{\la_3\tau}(-\la_3)\frac{\la_1-\la_2}{\la_1+\la_2}}{(\la_1-\la_2)(\la_2-\la_3)(\la_3-\la_1)}\FI_3,\\
  \CM_{22} &=& \frac{e^{\la_1\tau}\la_1(\la_2^2-\la_3^2)
+e^{\la_2\tau}\la_2(\la_3^2-\la_1^2)
+e^{\la_3\tau}\la_3(\la_1^2-\la_2^2)}{(\la_1-\la_2)(\la_2-\la_3)(\la_3-\la_1)}\FI_3,\\
  \CM_{33} &=& \frac{e^{\la_1\tau}(-|\xi|^2)\frac{\la_2^2-\la_3^2}{\la_1}
+e^{\la_2\tau}(-|\xi|^2)\frac{\la_3^2-\la_1^2}{\la_2}
+e^{\la_3\tau}(-|\xi|^2)\frac{\la_1^2-\la_2^2}{\la_3}}{(\la_1-\la_2)(\la_2-\la_3)(\la_3-\la_1)}\FI_3,
\end{eqnarray*}
and
\begin{eqnarray*}
  \CM_{12} &=& \frac{e^{\la_1\tau}\la_1(\la_2-\la_3)
+e^{\la_2\tau}\la_2(\la_3-\la_1)
+e^{\la_3\tau}\la_3(\la_1-\la_2)}{(\la_1-\la_2)(\la_2-\la_3)(\la_3-\la_1)}\FI_3,\\
  \CM_{13} &=& \frac{e^{\la_1\tau}(\la_2-\la_3)
+e^{\la_2\tau}(\la_3-\la_1)
+e^{\la_3\tau}(\la_1-\la_2)}{(\la_1-\la_2)(\la_2-\la_3)(\la_3-\la_1)}i\xi\times,\\
  \CM_{23} &=& \frac{e^{\la_1\tau}(\la_2^2-\la_3^2)
+e^{\la_2\tau}(\la_3^2-\la_1^2)
+e^{\la_3\tau}(\la_1^2-\la_2^2)}{(\la_1-\la_2)(\la_2-\la_3)(\la_3-\la_1)}i\xi\times,
\end{eqnarray*}
and
\begin{equation*}
    \CM_{21}=-\CM_{12}, \ \CM_{31}=\CM_{13}, \ \CM_{32}=-\CM_{23},
\end{equation*}
where $\la_j=\la_j(\xi)$, $j=1,2,3$, denote the roots of
the indicate equation \eqref{ind}.
\end{proposition}

Now, by combining two representations of the Green's functions to the fluid part and the electromagnetic part obtained in Proposition \ref{prop.n} and Proposition \ref{prop.m} respectively, we have

\begin{theorem}\label{thm.gr}
Let $U=[n,u,E,B]$ satisfy the system \eqref{leq} for all $t>0$, $x\in \R^3$ with initial data $U|_{t=0}=U_0=[n_0,u_0,E_0,B_0]$. Assume that the solution $U$ is written as
\begin{equation*}
    U=G\ast U_0
\end{equation*}
for the Green's function $G=G(t,x)$. Then, under the decomposition
\begin{equation*}
    U=U^{(1)}+U^{(2)},\ \  U_0=U^{(1)}_0+U^{(2)}_0
\end{equation*}
with
\begin{equation*}
 U^{(1)}=\left(\begin{array}{c}
      \hat{n}\\
      \hat{u}_\parallel\\
      \hat{E}_\parallel\\
      0
    \end{array}\right),\
 U^{(1)}_0=\left(\begin{array}{c}
      \hat{n}_0\\
      \hat{u}_{0,\parallel}\\
      \hat{E}_{0,\parallel}\\
      0
    \end{array}\right),\
 U^{(2)}=\left(\begin{array}{c}
       0\\
      \hat{u}_\perp\\
      \hat{E}_\perp\\
      \hat{B}_\perp
    \end{array}\right),\
 U^{(2)}_0=\left(\begin{array}{c}
       0\\
      \hat{u}_{0,\perp}\\
      \hat{E}_{0,\perp}\\
      \hat{B}_{0,\perp}
    \end{array}\right),
\end{equation*}
there are two Green's functions $G^{(1)}$, $G^{(2)}$ defined by
\begin{equation*}
    \hat{G}^{(1)}(t,k)=\left(\begin{array}{c|c}
       \CN(t,k) & \\
       \hline
        & \ 0
    \end{array}\right),\ \
    \hat{G}^{(2)}(t,k)=\left(\begin{array}{c|c}
      0 & \\
       \hline
        & \ \CM(\be t,\frac{k}{\be})
    \end{array}\right)
\end{equation*}
except a finite number of values of $|k|$, such that
\begin{equation*}
    U^{(j)}=G^{(j)}\ast  U^{(j)}_0,\ \ j=1,2.
\end{equation*}

\end{theorem}

\section{Solve a third order ODE}\label{sec4}

In this section we study the property of roots of the indicate equation
\begin{equation*}
    z^3+a |\xi|^2 z^2 +(|\xi|^2+1)z+a |\xi|^4=0,\ \ z\in \C.
\end{equation*}
Here, $a>0$ is a constant. In what follows,  set
\begin{equation*}
  r:=|\xi|^2>0,
\end{equation*}
and
\begin{equation*}
    g(z):=z^3+a r z^2 +(r+1)z+a r^2.
\end{equation*}

\subsection{Elementary properties}\label{sec4.1}

We list some elementary properties for roots of $g(z)=0$ as follows.

\medskip

\noindent{\it Property 1.}
Consider the existence of a real root for $g(z)=0$ with $z\in \R$.
Check that
\begin{itemize}
  \item $g(0)=a r^2> 0$;
  \item $g(-ar)=(-ar)^3+a^3r^3-ar(r+1)+a r^2=-ar< 0$;
  \item $g'(z)=3z^2+2a rz +(r+1)>r+1$ for $z\geq 0$;
  \item $g'(z)=z^2+2z(z+a r) +(r+1)>r+1$ for $z\leq -ar$;
  \item $g(z)$ is strictly increasing over $\{z\leq -ar\}\cup \{z\geq 0\}$.
\end{itemize}
The above observation implies that equation $g(z)=0$ has at least  one real root denoted by
$\si=\si(r)\in \R$, which satisfies $-ar<\si(r)<0$ for all $r> 0$.
Moreover, suppose that $\si=\si(r)\in (-ar,0)$ is a real root of $g(z)=0$. Then
\begin{equation*}
    g(\si)=\si^2(\si+ar)+\si (r+1)+a r^2=0.
\end{equation*}
Since $\si^2(\si+ar)>0$, then $\si (r+1)+a r^2<0$ holds, i.e.
\begin{equation*}
    \si<-\frac{a r^2}{r+1}.
\end{equation*}
Therefore, any real root $\si=\si(r)$ of $g(z)=0$ actually satisfies
\begin{equation*}
    -ar<\si(r)<-\frac{a r^2}{r+1},\ \ \forall\,r> 0.
\end{equation*}

\medskip
\noindent{\it Property 2.} Let $r>\frac{1}{a}$. Then
$-ar+\frac{1}{ar}\in (-ar,0)$.
Compute
\begin{equation*}
    g(-ar+\frac{1}{ar})=(-ar+\frac{1}{ar})^3+ar(-ar+\frac{1}{ar})^2+(r+1)(-ar+\frac{1}{ar})+ar^2
\end{equation*}
to obtain
\begin{equation*}
   g(-ar+\frac{1}{ar})=\frac{1}{a}\left[1-\frac{1}{r}(1+\frac{1}{ar})(1-\frac{1}{ar})\right].
\end{equation*}
Notice
\begin{equation*}
    \sup\limits_{0<\frac{1}{ar}<1}\frac{1}{ar}(1+\frac{1}{ar})(1-\frac{1}{ar})
=\frac{1}{\sqrt{3}}(1+\frac{1}{\sqrt{3}})(1-\frac{1}{\sqrt{3}})=\frac{2}{3\sqrt{3}}.
\end{equation*}
Then, whenever
\begin{equation*}
    1-\frac{2}{3\sqrt{3}}a>0,\ \ \text{i.e.}\ \ a<\frac{3\sqrt{3}}{2},
\end{equation*}
one has
\begin{equation*}
  g(-ar+\frac{1}{ar})>0.
\end{equation*}
Thus, it follows  that if $a<\frac{3\sqrt{3}}{2}$, then whenever $r>\frac{1}{a}$,  equation $g(z)=0$ has at least  one real root, denoted still by $\si=\si(r)$, satisfying
\begin{equation*}
    -ar<\si(r)<-ar+\frac{1}{ar},\ \ \forall\,r> \frac{1}{a}.
\end{equation*}

\medskip
\noindent{\it Property 3.} Let $\si\in (-ar,0)$ be a real root of $g(z)=0$. Consider the sign of
\begin{equation*}
    g'(\si)=3\si^2+2ar\si+(1+r)
\end{equation*}
by the following two cases.

\medskip
\noindent{\it Case i.} Write
\begin{equation*}
    g'(\si)=3\si^2+r(2a\si+1)+1.
\end{equation*}
If $2a\si+1\geq 0,\ \ \text{i.e.}\ \ \si\geq -\frac{1}{2a}$,
then $g'(\si)\geq 1$.

\medskip
\noindent{\it Case ii.} Using $g(\si)=0$, compute
\begin{eqnarray*}
  g'(\si) &=& \frac{[3\si^2+r(2a\si+1)+1]\si}{\si}\\
  &=&\frac{3\si^3+2ar\si^2+(1+r)\si}{\si} \\
  &=&\frac{-3[ar\si^2+(1+r)\si+ar^2]+2ar\si^2+(1+r)\si}{\si}\\
  &=&\frac{-ar \si^2-2(1+r)\si-3ar^2}{\si}.
\end{eqnarray*}
Write
\begin{equation*}
   -ar \si^2-2(1+r)\si-3ar^2=-ar(\si^2-2)-2(1+r)(\si+ar)-ar^2.
\end{equation*}
If
$\si^2-2\geq 0,\ \ \text{i.e.}\ \ \si\leq -\sqrt{2}$,
then $g'(\si)>0$.

\medskip
Let
$-\frac{1}{2a}\leq  -\sqrt{2}, \ \ \text{i.e.}\ \ a\leq \frac{\sqrt{2}}{4}$.
Then, from case i and case ii, one has
$g'(\si)>0$.
This implies that if $a\leq \frac{\sqrt{2}}{4}$ then equation $g(z)=0$ has one and only one real root $\si=\si(r)$. Moreover, the real root $\si(r)$ is strictly decreasing in $r$ if  $a\leq \frac{\sqrt{2}}{4}$. In fact, since
\begin{equation*}
    g(\si)=\si^3+ar\si^2+(r+1)\si+ar^2=0,
\end{equation*}
then
\begin{equation*}
    g'(\si)\frac{d\si}{dr}+(a\si^2+\si+2ar)=0,
\end{equation*}
that is,
\begin{equation*}
  \frac{d\si}{dr}=\frac{-a\si^2-\si-2ar}{g'(\si)}.
\end{equation*}
From
\begin{equation*}
  -a\si^2-\si-2ar=-a\si^2-(\si+ar)-ar ,\ \ -ar<\si<0,
\end{equation*}
one can see
\begin{equation*}
      -a\si^2-\si-2ar<0.
\end{equation*}
Therefore,
\begin{equation*}
   \frac{d\si}{dr}<0,\ \ \forall\,r>0.
\end{equation*}

\medskip
\noindent{\it Property 4.} Let $\si=\si(r)\in (-ar,0)$ be one real root of $g(z)=0$. Then
\begin{equation*}
    g(z)=(z-\si)[z^2+(\si+ar)z+(\si^2+ar\si+r+1)]=0.
\end{equation*}
The rest two roots, denoted by $\chi_\pm$, satisfy the equation
\begin{equation*}
  z^2+(\si+ar)z+(\si^2+ar\si+r+1)=0.
\end{equation*}
Thus $\chi_\pm$ is solved as
\begin{equation*}
    \chi_\pm=-\frac{1}{2}(\si+ar)\pm\frac{1}{2}\sqrt{(\si+ar)^2-4(\si^2+ar\si+r+1)}.
\end{equation*}
To the end, we set $\varphi$ to be a function of $\si$ given by
\begin{equation*}
    \varphi=3\si^2+2ar\si-a^2r^2+4(r+1).
\end{equation*}
It is straightforward to check
\begin{equation*}
    \chi_\pm=-\frac{1}{2}(\si+ar)\pm\frac{1}{2}\sqrt{-\varphi}.
\end{equation*}

\noindent{\bf Claim:} {\it If $a<\frac{\sqrt{2}}{4}$ then $\varphi>0$ for all $r>0$, and moreover
\begin{equation*}
    \frac{3}{a^2r^2}+4r\leq \varphi\leq 4(1+r),\ \ \forall\, r>\frac{\sqrt{6}}{2a}.
\end{equation*}}

\noindent{\it Proof of Claim.} Let $a<\frac{\sqrt{2}}{4}$. From Properties 1-3, $\si\in (-ar,0)$ is one and only one real root of $g(z)=0$ and satisfies
\begin{equation*}
    -ar<\si(r)<-ar+\frac{1}{ar},\ \ \forall\, r>\frac{1}{a}.
\end{equation*}
Write
\begin{equation*}
    \varphi=3\si^2+2ar(\si+ar)+r(4-3a^2r)+4.
\end{equation*}
Then, whenever $4-3a^2r\geq 0$, i.e. $r\leq \frac{4}{3a^2}$,
$\varphi\geq 4$ holds.
On the other hand, whenever $r>\frac{1}{a}$ and $\frac{1}{ar}\leq \frac{2ar}{3}$, i.e. $r\geq \frac{\sqrt{6}}{2a}$, one has
\begin{eqnarray*}
 0\geq  3\si^2+2ar\si-a^2r^2 &\geq & 3(-ar+\frac{1}{ar})^2+2ar(-ar+\frac{1}{ar})-a^2r^2\\
 &=& 3(a^2r^2-2+\frac{1}{a^2r^2})-2a^2r^3+2-a^2r^2\\
 &=&-4+\frac{3}{a^2r^2},
\end{eqnarray*}
which implies
\begin{equation*}
   4(1+r)\geq 3\si^2+2ar\si-a^2r^2+  4(1+r)=\varphi\geq  \frac{3}{a^2r^2}+4r.
\end{equation*}
Since
$\frac{\sqrt{6}}{2a}\leq  \frac{4}{3a^2}$
is equivalent with $a\leq \frac{4\sqrt{6}}{9}$ and the inequality $ \frac{4\sqrt{6}}{9}>\frac{\sqrt{2}}{4}$ holds, $\varphi>0$
holds true for any $r>0$. Claim follows.

\medskip

Therefore, when $a\leq \frac{\sqrt{2}}{4}$, equation $g(z)=0$ has one real root $\si$ and two non-real conjugate roots $\chi_{\pm}$ with
\begin{equation*}
    \chi_\pm=-\frac{1}{2}(\si+ar)\pm\frac{1}{2}\sqrt{\varphi}i.
\end{equation*}

\subsection{Representation of roots}\label{sec4.2}

We shall consider a representation of all roots $z\in \C$ of $g(z)=0$. For this moment, although we have known that there is at least one real root, it is not clear whether there exists an exact condition to assure its uniqueness. Recall the equation
\begin{equation*}
    g(z)=z^3+c_2z^2 +c_1z+c_0=0,\ \ c_2=a r,\ c_1=r+1,\ c_0=a r^2.
\end{equation*}
Define
\begin{eqnarray*}
  S &=& 2c_2^3-9c_2c_1 +27 c_0=2a^3r^3-9ar(r+1)+27 ar^2,\\
  R &=& S^2-4(c_2^2-3c_1)^3\\
  &=&[2a^3r^3-9ar(r+1)+27 ar^2]^2-4[a^2r^2-3(r+1)]^3.
\end{eqnarray*}
Then, there are two distinct cases:
\begin{itemize}
  \item $ R>0$ in which case there are one real root and two non-real conjugate roots;
  \item $ R\leq 0$ in which case there are three real roots.
\end{itemize}
The roots of $g(z)=0$ can be written as
\begin{eqnarray*}
  \la_1 &=& -\frac{1}{3}\left(c_2+\sqrt[3]{\frac{S+\sqrt{R}}{2}}+\sqrt[3]{\frac{S-\sqrt{R}}{2}}\right),\\
  \la_2 &=& -\frac{1}{3}\left(c_2+\om\sqrt[3]{\frac{S+\sqrt{R}}{2}}+\bar{\om}\sqrt[3]{\frac{S-\sqrt{R}}{2}}\right),\\
  \la_3 &=& -\frac{1}{3}\left(c_2+\bar{\om}\sqrt[3]{\frac{S+\sqrt{R}}{2}}+\om\sqrt[3]{\frac{S-\sqrt{R}}{2}}\right),
\end{eqnarray*}
where
\begin{equation*}
    \om=-\frac{1}{2}+\frac{1}{2}\sqrt{3}i,\ \ \bar{\om}=-\frac{1}{2}-\frac{1}{2}\sqrt{3}i.
\end{equation*}
When $R<0$, we set
\begin{eqnarray*}
  \sqrt[3]{\frac{S+\sqrt{R}}{2}}&=&\sqrt[3]{\frac{S+i\sqrt{-R}}{2}}=(\frac{S^2-R}{4})^{1/6}e^{\frac{i\theta}{3}}
=(c_2^2-3c_1)^{1/2}e^{\frac{i\theta}{3}},\\
  \sqrt[3]{\frac{S-\sqrt{R}}{2}}&=&\sqrt[3]{\frac{S-i\sqrt{-R}}{2}}=(\frac{S^2-R}{4})^{1/6}e^{\frac{-i\theta}{3}}
=(c_2^2-3c_1)^{1/2}e^{-\frac{i\theta}{3}},
\end{eqnarray*}
where
\begin{equation*}
    \theta=\arccos\frac{S}{\sqrt{S^2-R}}.
\end{equation*}
Notice $c_2^2-3c_1>0$ if $R<0$, that is,
\begin{equation*}
    R= S^2-4(c_2^2-3c_1)^3<0.
\end{equation*}
In this way, it is easy to see that for $R<0$,
\begin{equation*}
  \sqrt[3]{\frac{S+\sqrt{R}}{2}},\ \  \sqrt[3]{\frac{S-\sqrt{R}}{2}}
\end{equation*}
are a pair of conjugate complex numbers, and  all roots $\la_i$, $i=1,2,3$, are real .

Let's compute $R$. In fact,
\begin{equation*}
    R=27[4(1+r)^3+a^2r^2(8r^2-20r-1)+4a^4r^5].
\end{equation*}
It is straightforward to observe that whenever $r>0$ is sufficiently small or large, $R>0$ holds true. In addition, since $R$ can be also written as
\begin{equation*}
    R=27[4+12r+(12-a^2)r^2+(4-20a^2)r^3+8a^2r^4+4a^4r^5],
\end{equation*}
then, as long as $a\leq \frac{1}{\sqrt{5}}$, one also has $R>0$ for all $r>0$.

\subsection{Asymptotic of roots at $0$ and $\infty$}

Suppose that there are $\eps=\eps(a)>0$, $L=L(a)>0$ such that $R>0$ over $r\leq \eps(a)$ and $r\geq L(a)$. Then, over $\{r\leq \eps(a)\}\cup\{r\geq L(a)\}$, equation $g(z)=0$ has only one real root $\si$ and two non-real conjugate complex roots $\chi_\pm$, and thus we let
\begin{equation*}
    \la_1=\si,\ \ \la_{2,3}=\chi_\pm.
\end{equation*}

\medskip
\noindent{\it Case when $r\leq \eps(a)$:} $\la_1=\si=\si(r)$ satisfies
\begin{equation*}
    g(\si)=\si^3+ar\si+(1+r)\si+ar^2=0.
\end{equation*}
Set
\begin{equation*}
    \si(r)=\sum_{m=0}^\infty a_mr^m
\end{equation*}
as $r\to 0$. One can compute
\begin{equation*}
    a_0=a_1=0,\ a_2=-a,\ a_3=a,\ a_4=-a,\ a_5=a(1-a^2).
\end{equation*}
Thus, the real root $\si$ can be written  as
\begin{equation*}
    \si(r)=-ar^2+ar^3-ar^4+a(1-a^2)r^5+\cdots=-ar^2(1-r+O(1)r^2)
\end{equation*}
as $r\to 0$. Recall that two non-real roots $\chi_\pm$ take the form of
\begin{equation*}
    \chi_\pm=-\frac{\si+ar}{2}\pm\frac{1}{2}\sqrt{\varphi}i.
\end{equation*}
Then, it follows
\begin{eqnarray*}
  \Re\, \chi_+ &=& -\frac{-ar^2+ar^3-O(1)r^4+ar}{2}=-\frac{ar(1-r+O(1)r^2)}{2},\\
  \Im\,\chi_+ &=&\frac{1}{2}\sqrt{3\si^2+2ar\si-a^2r^2+4(1+r)}\\
  &=& \frac{1}{2}\sqrt{3a^2r^4(1-r+O(1)r^2)^2-2a^2r^3(1-r+O(1)r^2)-a^2r^2+4(1+r)}\\
  &=& \frac{1}{2}\sqrt{-a^2r^2-O(1)r^3+4(1+r)}\\
  &=&\sqrt{1+r-O(1)r^2},
\end{eqnarray*}
which implies
\begin{equation*}
    \chi_\pm=-\frac{ar(1-r+O(1)r^2)}{2}\pm i\sqrt{1+r-O(1)r^2}
\end{equation*}
as $r\to 0$.

\medskip
\noindent{\it Case when $r\geq L(a)$:} Notice that
\begin{equation*}
    g(\si)=\si^3+ar\si+(1+r)\si+ar^2=0
\end{equation*}
is equivalent with
\begin{equation*}
    (\frac{\si}{r})^3+a (\frac{\si}{r})^2+\frac{1+r}{r^2}\frac{\si}{r} +\frac{a}{r}=0.
\end{equation*}
Set
\begin{equation*}
    \frac{\si}{r}=\sum_{m=0}^\infty a_mr^{-m}
\end{equation*}
as $r\to \infty$. One can check
\begin{equation*}
    a_0=-a,\ a_1=0,\ a_2=\frac{1}{a},\ a_3=-\frac{1}{a^3},\ a_4=\frac{1}{a^3}(1+\frac{1}{a}),\ \ a_5=-\frac{1}{a^5}(4+\frac{1}{a}).
\end{equation*}
Then, the real root $\si$ can be written  as
\begin{eqnarray*}
  \frac{\si}{r} = -a+\frac{1}{a}r^{-2}-\frac{1}{a^3}r^{-3}+ \frac{1}{a^3}(1+\frac{1}{a})r^{-4}
  -\frac{1}{a^3}(1+\frac{1}{a})r^{-5}+\cdots,
\end{eqnarray*}
that is,
\begin{equation*}
    \si=-ar+\frac{1}{ar} (1-\frac{1}{a^2r}+\frac{O(1)}{r^2})
\end{equation*}
as $r\to\infty$. For non-real roots, let's  further compute
\begin{eqnarray*}
  \Re \chi_\pm &=& -\frac{\si+ar}{2} =-\frac{1}{2ar}(1-\frac{1}{a^2r}+\frac{O(1)}{r^2}),
\end{eqnarray*}
and
\begin{eqnarray*}
  \varphi &=&  3\si^2+2ar\si-a^2r^2+4(1+r)\\
  &=&3\si (\si+ar)-ar(\si+ar)+4(1+r)\\
  &=&r(\frac{3\si}{r}-a)(\si+ar)+4(1+r)\\
  &=& r\left[-3a+\frac{3}{ar^2}(1-\frac{1}{a^2r}+\frac{O(1)}{r^2})-a\right]\frac{1}{ar} (1-\frac{1}{a^2r}+\frac{O(1)}{r^2})+4(1+r)\\
  &=&\left[-4+\frac{3}{a^2r^2}(1-\frac{1}{a^2r}+\frac{O(1)}{r^2})\right] (1-\frac{1}{a^2r}+\frac{O(1)}{r^2})+4(1+r)\\
  &=&-4(1-\frac{1}{a^2r}+\frac{O(1)}{r^2})+\frac{3}{a^2r^2}(1-\frac{1}{a^2r}+\frac{O(1)}{r^2})^2+4(1+r)\\
  &=&4r+O(1)r^{-1}.
\end{eqnarray*}
Therefore, one has
\begin{equation*}
    \chi_\pm=-\frac{1}{2ar}(1-\frac{1}{a^2r}+\frac{O(1)}{r^2})\pm i \sqrt{r+O(1)r^{-1}}
\end{equation*}
as $r\to\infty$.

\medskip

We conclude this section with a remark. The asymptotic behavior of roots $\la_i(r)$ at origin and infinity exactly determines the same thing for the Fourier transform $\CM$ of the Green's function of the electromagnetic part. However, $\CM$ is not defined at a finite number of values of $r=|\xi|^2=|k|^2/\be$ where $R=0$, that is, possible multiple real roots occur. Therefore, to estimate $\CM$ over the whole frequency domain, one has to characterize the behavior of all roots  $\la_i(r)$ in the domain $\{\eps(a)\leq r\leq L(a)\}$, particularly over some small neighborhood of a finite number of points such that $R=0$; we shall clarify this later on.  Notice that this issue is different from that in \cite{Du-EM} where the situation is simpler due to the fact that all roots are far from the imaginary axis over  $\{\eps(a)\leq r\leq L(a)\}$ and multiple roots do not happen.

\section{Pointwise estimate on the Green's function}\label{sec5}

In this section, we consider the pointwise estimates on $\CN$ and $\CM$ which are the Fourier transforms of the Green's function to the fluid part and electromagnetic part, respectively. The most difficult part is the analysis of $\CN$ and $\CM$ in the finite domain far from origin and infinity, which turn out to decay in time exponentially.  We then use these estimates to derive the corresponding pointwise estimates on the Fourier transform of solutions to the linearized Navier-Stokes-Maxwell system \eqref{leq}, and further obtain the time decay property of each component in the solution.

\subsection{Pointwise bound on the fluid part}

Recall that $\la_\pm$ satisfies the indicate equation
\begin{equation*}
    z^2+\mu |k|^2 z + (\ga^2|k|^2+\be^2)=0,
\end{equation*}
and
\begin{equation*}
    \la_\pm=-\frac{1}{2}\mu |k|^2\pm \frac{1}{2} \sqrt{\psi},\ \ \psi:=\mu^2 |k|^4-4(\ga^2|k|^2+\be^2).
\end{equation*}
Divide the whole space $\R^3$ by three subdomains as
\begin{eqnarray*}
&\dis   \R^3=\D_0\cup\D_1\cup\D_\infty,\\
&\dis   \D_0 =  \{|k|\leq \eps\},\ \
  \D_1 =\{\eps\leq |k|\leq L\},\ \
  \D_\infty = \{|k|\geq L\},
\end{eqnarray*}
and further set
\begin{eqnarray*}
&\dis \D_1=\D_1^-\cup\D_1^+,\\
&\dis \D_1^-=\{\eps\leq |k|\leq L,\psi\leq 0\},\ \ \D_1^+=\{\eps\leq |k|\leq L, \psi\geq 0\}.
\end{eqnarray*}

\medskip
\noindent{\it Case when $|k|$ is near $0$:} In this case,
as $|k|\to 0$,
\begin{equation*}
    \la_\pm=-\frac{1}{2}\mu |k|^2\pm (\be + O(1)|k|^2)i
\end{equation*}
and
\begin{equation*}
    \la_+-\la_-=\sqrt{-\psi}i,\ \ \la_++\la_-=-\mu |k|^2,\ \ \la_+\la_-=\ga^2|k|^2+\be^2.
\end{equation*}
Notice that in $\CN_{11},\CN_{33}$,
\begin{equation*}
    \frac{\la_+ e^{\la_- t }-\la_- e^{\la_+ t}}{\la_+ - \la_-}=-\frac{\la_+\la_-}{\la_+-\la_-}(\frac{e^{\la_+ t}}{\la_+}-\frac{e^{\la_- t}}{\la_-})=-\frac{\ga^2|k|^2+\be^2}{\sqrt{-\psi}} 2\Im \frac{e^{\la_+ t}}{\la_+}.
\end{equation*}
Then, $\CN_{11},\CN_{33}$ can be estimated as
\begin{equation*}
    |\CN_{11}|+|\CN_{33}|\lesssim e^{-\frac{1}{2}\mu |k|^2 t}.
\end{equation*}
In the similar way, it holds that
\begin{equation*}
   |\CN_{22}|+ |\CN_{23}|\lesssim e^{-\frac{1}{2}\mu |k|^2 t},\ \   |\CN_{12}|\lesssim |k|e^{-\frac{1}{2}\mu |k|^2 t}.
\end{equation*}

\medskip
\noindent{\it Case when $|k|$ is near $\infty$:} In this case, one can compute that
as $|k|\to \infty$,
\begin{eqnarray*}
  \la_\pm &=& -\frac{1}{2}\mu |k|^2 \pm \frac{1}{2}\mu |k|^2 (1-\frac{4(\ga^2|k|^2+\be^2)}{\mu^2|k|^4})^{1/2} \\
 &=& -\frac{1}{2}\mu |k|^2 \pm \frac{1}{2}\mu |k|^2 (1-\frac{2(\ga^2|k|^2+\be^2)}{\mu^2|k|^4}-O(1)[\frac{\ga^2|k|^2+\be^2}{\mu^2|k|^4}]^2)\\
&=& -\frac{1}{2}\mu |k|^2 \pm \frac{1}{2}\mu |k|^2 (1-\frac{2\ga^2}{\be^2}|k|^{-2}-O(1)|k|^{-4}),
\end{eqnarray*}
which implies
\begin{eqnarray*}
  \la_- &=&  -\frac{1}{2}\mu |k|^2-\frac{1}{2}\mu |k|^2+\frac{\mu\ga^2}{\be^2}+O(1)|k|^{-2}=-O(1)|k|^2,\\
\la_+ &=& -\frac{1}{2}\mu |k|^2+\frac{1}{2}\mu |k|^2-\frac{\mu\ga^2}{\be^2}-O(1)|k|^{-2}=-\frac{\mu\ga^2}{\be^2}-O(1)|k|^{-2}.
\end{eqnarray*}
Therefore, one has
\begin{eqnarray*}
    |\CN_{11}|+|\CN_{33}|&\leq &\frac{|\la_+\la_-|}{|\la_+-\la_-|}(\frac{e^{\la_+ t}}{|\la_+|}+\frac{e^{\la_- t}}{|\la_-|})\\
&\leq & \frac{\ga^2 |k|^2+\be^2}{\sqrt{\psi}} (\frac{e^{-(\frac{\mu\ga^2}{\be^2}+O(1)|k|^{-2})t}}{\frac{\mu\ga^2}{\be^2}+O(1)|k|^{-2}}
+\frac{e^{-O(1)|k|^2 t}}{O(1)|k|^2})\\
&\lesssim & e^{-O(1)t} +\frac{1}{|k|^2}e^{-O(1)|k|^2 t},
\end{eqnarray*}
and
\begin{eqnarray*}
    |\CN_{22}|&\leq &\frac{1}{|\la_+-\la_-|}(|\la_+|e^{\la_+ t}+|\la_-|e^{\la_- t})\\
&\leq & \frac{1}{\sqrt{\psi}} ((\frac{\mu\ga^2}{\be^2}+O(1)|k|^{-2})e^{-(\frac{\ga^2}{\be^2}+O(1)|k|^{-2})t}
+O(1)|k|^2e^{-O(1)|k|^2 t})\\
&\lesssim& \frac{1}{|k|^2}e^{-O(1)t} +e^{-O(1)|k|^2 t}.
\end{eqnarray*}
In the similar way, one can also obtain
\begin{eqnarray*}
    |\CN_{12}|&\leq &\frac{\ga |k|}{|\la_+-\la_-|}(e^{\la_+ t}+e^{\la_- t})\\
&\leq & \frac{\ga |k|}{\sqrt{\psi}} (e^{-(\frac{\mu\ga^2}{\be^2}+O(1)|k|^{-2})t}
+e^{-O(1)|k|^2 t})\\
&\lesssim & \frac{1}{|k|}e^{-O(1)t} +\frac{1}{|k|}e^{-O(1)|k|^2 t},
\end{eqnarray*}
and
\begin{eqnarray*}
    |\CN_{23}|
&\lesssim & \frac{1}{|k|^2}e^{-O(1)t} +\frac{1}{|k|^2}e^{-O(1)|k|^2 t}.
\end{eqnarray*}

\medskip
\noindent{\it Case when $|k|$ is Far from $0$ and $\infty$:} In this case, $\la_+$ tends to $\la_-$ as $|k|$ is properly chosen such that $\psi\to 0$. Still consider first the estimate on $\CN_{11},\CN_{33}$. Notice
\begin{equation*}
    \frac{\la_+ e^{\la_- t }-\la_- e^{\la_+ t}}{\la_+ - \la_-}=-\la_+\la_-\int_0^1\frac{d}{d\la} (\frac{e^{\la t}}{\la})|_{\la=(1-s)\la_-+s\la_+} ds,
\end{equation*}
where
\begin{equation*}
 \frac{d}{d\la} (\frac{e^{\la t}}{\la})=\frac{(t\la -1) e^{\la t}}{\la^2}.
\end{equation*}
Then, it follows
\begin{equation*}
    |\CN_{11}|+|\CN_{33}|\leq |\la_+\la_-|\int_0^1
\frac{t|(1-s)\la_-+s\la_+|+1}{|(1-s)\la_-+s\la_+|^2}e^{t \Re\{(1-s)\la_-+s\la_+\}}ds.
\end{equation*}
We make further estimates over two subdomains $\D_1^+$ and $\D_1^-$ as follows.

\medskip
\noindent Over $\D_{1}^+=\{\eps\leq |k|\leq L, \psi\geq 0\}$: $\la_\pm<0$ are real, which shows
\begin{equation*}
   (1-s)\la_-+s\la_+<0,\ 0\leq s\leq 1.
\end{equation*}
Thus, one has
\begin{multline*}
   |\CN_{11}|+|\CN_{33}|\lesssim  \int_0^1\frac{t\max|(1-s)\la_-+s\la_+|+1}{\min|(1-s)\la_-+s\la_+|^2}e^{t \min\Re\{(1-s)\la_-+s\la_+\}}ds\\
\lesssim  (1+t)e^{-O(1)t}
\lesssim  e^{-O(1)t}.
\end{multline*}
where $\max$, $\min$ are taken over $\{k\in\D_1^+,0\leq s\leq 1\}$.

\medskip
\noindent
Over $\D_{1}^-=\{\eps\leq |k|\leq L, \psi\leq 0\}$: $\la_\pm$ are complex conjugate, which shows
\begin{equation*}
    \Re \{(1-s)\la_-+s\la_+\}=\Re \la_\pm=-\frac{1}{2}\mu|k|^2<0.
\end{equation*}
Then, in the similar way, one has
\begin{multline*}
   |\CN_{11}|+|\CN_{33}|\lesssim  \int_0^1\frac{t\max|(1-s)\la_-+s\la_+|+1}{\min| \Re \{(1-s)\la_-+s\la_+\}|^2}e^{t\Re \{(1-s)\la_-+s\la_+\}}ds\\
\lesssim (1+t)e^{-O(1)t}
\lesssim  e^{-O(1)t}.
\end{multline*}
where $\max$, $\min$ are taken over $\{k\in\D_1^-,0\leq s\leq 1\}$.

\medskip
Therefore, from both cases above, over $\D_1=\D_1^+\cup\D_1^-$, it holds
\begin{equation*}
     |\CN_{11}|+|\CN_{33}|\lesssim e^{-O(1)t}.
\end{equation*}
In the completely same way, over  $\CD_1$,
\begin{equation*}
     |\CN_{22}|+|\CN_{12}|+|\CN_{23}|\lesssim e^{-O(1)t}.
\end{equation*}

\medskip

In sum, we have proved

\begin{proposition}\label{prop.n.bdd}
Let the matrix $\CN=\CN(t,k)_{7\times 7}$ be defined in Proposition \ref{prop.n}. The elements of $\CN$ have the following upper bounds for pointwise $t\geq 0$ and $k\in \R^3$.
Over $\D_0$,
\begin{equation*}
   |\CN| \lesssim  \left(\begin{array}{ccc}
                      1 & |k|d_3^T & 0 \\
                      |k|d_3 & \FI_3 &\FI_3 \\
                      0 & \FI_3 & \FI_3
                    \end{array}\right)
e^{-\frac{1}{2}\mu |k|^2 t} .
\end{equation*}
Over $\D_\infty$,
\begin{multline*}
   |\CN| \lesssim  \left(\begin{array}{ccc}
                      1 & |k|^{-1}d_3^T & 0 \\
                      |k|^{-1}d_3 & |k|^{-2}\FI_3 & |k|^{-2}\FI_3 \\
                      0 & |k|^{-2}\FI_3 & \FI_3
                    \end{array}\right)
e^{-O(1) t}\\
+ \left(\begin{array}{ccc}
                      |k|^{-2} & |k|^{-1}d_3^T & 0 \\
                      |k|^{-1}d_3 & \FI_3 & |k|^{-2}\FI_3 \\
                      0 & |k|^{-2}\FI_3 & |k|^{-2}\FI_3
                    \end{array}\right)
e^{-O(1) |k|^2 t}.
\end{multline*}
Over $\D_1$,
\begin{equation*}
   |\CN| \lesssim  \left(\begin{array}{ccc}
                      1 & d_3^T & 0 \\
                      d_3 & \FI_3 & \FI_3 \\
                      0 & \FI_3 & \FI_3
                    \end{array}\right)
e^{-O(1) t}.
\end{equation*}
Here, for simplicity, $|\CN|$ means the  matrix corresponding to $\CN$ with each element taken the absolute value. $d_3$ denotes the column vector $(1,1,1)^T$.
\end{proposition}

\subsection{Pointwise bound on the electromagnetic part}

In order to estimate the pointwise bound on $\CM$, let's rewrite its elements as
\begin{eqnarray*}
  \CM_{11} &=&\frac{\la_1}{(\la_1-\la_2)(\la_1-\la_3)(\la_2+\la_3)}e^{\la_1 \tau}\\
&&-\frac{1}
{(\la_1+\la_2)(\la_1+\la_3)(\la_2-\la_3)}
\left[\frac{\la_2(\la_1+\la_2)}{\la_1-\la_2}e^{\la_2\tau}-\frac{\la_3(\la_1+\la_3)}{\la_1-\la_3}
e^{\la_3\tau}\right],\\
  \CM_{22} &=&-\frac{\la_1(\la_2+\la_3)}{(\la_1-\la_2)(\la_1-\la_3)} e^{\la_1\tau}\\
&&+\frac{(\la_1+\la_2)(\la_1+\la_3)}{\la_2-\la_3}\left[
\frac{\la_2}{\la_1^2-\la_2^2}e^{\la_2\tau}-\frac{\la_3}{\la_1^2-\la_3^2}e^{\la_3\tau}
\right],\\
  \CM_{33} &=&\frac{r(\la_2+\la_3)}{\la_1(\la_1-\la_2)(\la_1-\la_3)} e^{\la_1\tau}\\
&&-\frac{r(\la_1+\la_2)(\la_1+\la_3)}{\la_2-\la_3}\left[
\frac{1}{\la_2(\la_1^2-\la_2^2)}e^{\la_2\tau}-\frac{1}{\la_3(\la_1^2-\la_3^2)}e^{\la_3\tau}
\right],
\end{eqnarray*}
and
\begin{eqnarray*}
  \CM_{12} &=&-\frac{\la_1}{(\la_1-\la_2)(\la_1-\la_3)} e^{\la_1\tau}\\
&&+\frac{1}{\la_2-\la_3}\left[
\frac{\la_2}{\la_1-\la_2}e^{\la_2\tau}-\frac{\la_3}{\la_1-\la_3}e^{\la_3\tau}
\right],\\
  \CM_{13} &=& \left\{-\frac{1}{(\la_1-\la_2)(\la_1-\la_3)} e^{\la_1\tau}\right.\\
&&\left.+\frac{1}{\la_2-\la_3}\left[
\frac{1}{\la_1-\la_2}e^{\la_2\tau}-\frac{1}{\la_1-\la_3}e^{\la_3\tau}
\right]\right\} i\xi\times,\\
  \CM_{23} &=& \left\{-\frac{\la_2+\la_3}{(\la_1-\la_2)(\la_1-\la_3)} e^{\la_1\tau}\right.\\
&&\left.+\frac{(\la_1+\la_2)(\la_1+\la_3)}{\la_2-\la_3}\left[
\frac{1}{\la_1^2-\la_2^2}e^{\la_2\tau}-\frac{1}{\la_1^2-\la_3^2}e^{\la_3\tau}
\right]\right\} i\xi\times,
\end{eqnarray*}
and
\begin{equation*}
    \CM_{21}=-\CM_{12}, \ \CM_{31}=\CM_{13}, \ \CM_{32}=-\CM_{23}.
\end{equation*}
Similarly as for considering $\CN$, let us divide the whole space $\R^3$ by three subdomains as
\begin{eqnarray*}
&\dis   \R^3=\D_0\cup\D_1\cup\D_\infty,\\
&\dis   \D_0 =  \{r\leq \eps\},\ \
  \D_1 =\{\eps\leq r\leq L\},\ \
  \D_\infty = \{r\geq L\}.
\end{eqnarray*}
Here, we still used $\D_0,\D_1,\D_\infty$ for simplicity of notations without any confusion.

\medskip
\noindent{\it Case when $r$ is near $0$, i.e. $r\in \D_0$}. In this case, $\la_i$, $i=1,2,3$, have the following asymptotic behavior
\begin{eqnarray*}
  \la_1&=&\si(r) = -ar^2 (1-r+O(1)r^2),\\
\la_2 &=&\chi_+ =-\frac{ar-ar^2+O(1)r^3}{2}+ i \sqrt{1+r-O(1)r^2},\\
\la_3 &=&\chi_-=\bar{\chi}_+,
\end{eqnarray*}
as $r$ tends to 0. This also implies
\begin{eqnarray*}
 \la_1-\la_{2,3}&=&\frac{1}{2}ar-O(1)r^2\mp i \sqrt{1+r-O(1)r^2},\ |\la_1-\la_{2,3}|=O(1),\\
 \la_1+\la_{2,3}&=&-\frac{1}{2}ar-O(1)r^2\pm i \sqrt{1+r-O(1)r^2}, \ |\la_1+\la_{2,3}|=O(1),\\
\la_2+\la_3 &=& -(ar-ar^2+O(1)r^3),\ \ |\la_2+\la_3|=O(1)r,\\
\la_2-\la_3 &=&2i \sqrt{1+r-O(1)r^2},\ \ |\la_2-\la_3|=O(1)r^{1/2}.
\end{eqnarray*}
By using the above computations, it is straightforward to make estimates on $\CM$ as   $r$ tends to 0:
\begin{eqnarray*}
  |\CM_{11}| &\lesssim & r e^{- O(1) r^2 \tau} +e^{- O(1) r \tau},\\
  |\CM_{22}| &\lesssim & r^3 e^{-O(1)r^2\tau} +  e^{-O(1)r\tau},\\
 |\CM_{33}| &\lesssim &  e^{-O(1)r^2\tau} + r e^{-O(1)r\tau},
\end{eqnarray*}
and
\begin{eqnarray*}
  |\CM_{12}| &\lesssim & r^2 e^{- O(1) r^2 \tau} +e^{- O(1) r \tau},\\
  |\CM_{13}| &\lesssim & r^{1/2} e^{-O(1)r^2\tau} +  r^{1/2}e^{-O(1)r\tau},\\
 |\CM_{23}| &\lesssim & r^{3/2} e^{-O(1)r^2\tau} + r^{1/2} e^{-O(1)r\tau}.
\end{eqnarray*}

\medskip
\noindent{\it Case when $r$ is near $\infty$, i.e. $r\in \D_\infty$}. In this case,
one has
\begin{eqnarray*}
  \la_1 &=& \si=-a r +\frac{1}{ar} (1-\frac{1}{a^2 r}+\frac{O(1)}{r^2}),\\
\la_2 &=&\chi_+=-\frac{1}{2ar} (1-\frac{1}{a^2 r}+\frac{O(1)}{r^2}) + i \sqrt{r+O(1)r^{-1}},\\
\la_3 &=& \chi_-=\bar{\chi}_+,
\end{eqnarray*}
as $r$ tends to $\infty$.
This further implies
\begin{eqnarray*}
  \la_1 - \la_2 &=& -ar +\frac{O(1)}{r} -  i \sqrt{r+O(1)r^{-1}},\ \ |\la_1-\la_2|= O(1)r\\
 \la_1 + \la_2 &=& -ar +\frac{O(1)}{r} +  i \sqrt{r+O(1)r^{-1}},\ \ |\la_1+\la_2| = O(1)r,\\
\la_2-\la_3 &=& 2 i \sqrt{r+O(1)r^{-1}},\ \ |\la_2-\la_3|=O(1)r^{1/2},\\
\la_2+\la_3 &=&-\frac{1}{ar} (1-\frac{1}{a^2 r}+\frac{O(1)}{r^2}),\ \ |\la_2+\la_3 |= O(1)r^{-1}.
\end{eqnarray*}
Then, as in the case when $r$ is near 0, it is also straightforward to obtain the estimates on $\CM$
\begin{eqnarray*}
|\CM_{11}| &\lesssim& e^{-O(1) r\tau} +\frac{1}{r^2}e^{-\frac{O(1)}{r}\tau},\\
|\CM_{22}| &\lesssim& \frac{1}{r^2}e^{-O(1)r\tau}+  e^{-\frac{O(1)}{r}\tau},\\
|\CM_{33}| &\lesssim & \frac{1}{r^3}e^{-O(1)r\tau}+  e^{-\frac{O(1)}{r}\tau},
\end{eqnarray*}
and
\begin{eqnarray*}
|\CM_{12}| &\lesssim& \frac{1}{r}e^{-O(1) r\tau} +\frac{1}{r}e^{-\frac{O(1)}{r}\tau},\\
|\CM_{13}| &\lesssim& \frac{1}{r^{3/2}}e^{-O(1)r\tau}+ \frac{1}{r} e^{-\frac{O(1)}{r}\tau},\\
|\CM_{23}| &\lesssim& \frac{1}{r^{5/2}}e^{-O(1)r\tau}+  e^{-\frac{O(1)}{r}\tau},
\end{eqnarray*}
as $r$ tends to $\infty$.

\medskip
\noindent{\it Case when $r>0$ is far from $0$ and $\infty$, i.e. $r\in \D_1$}. In this case, we shall consider the estimates on $\CM$ by dividing $\D_1=\{\eps\leq r\leq L\}$ into three parts. For that, denote $A_{\eps,L}$ as a set of all zeros  of $R$ over $[\eps,L]$, that is,
\begin{equation*}
   A_{\eps,L}=\{r_i,\ 1\leq i\leq i_0\}:=\{\eps\leq r\leq L, R=0\}.
\end{equation*}
Notice that from Property 1 in Section \ref{sec4.1},  $g(z)=0$ can not have one real root with three multiples.
Therefore, for any $r_i\in A_{\eps,L}$, by the representation of roots in Section \ref{sec4.2}, $g(z)=0$ at $r=r_i$ has one 1-multiple real root and the other different 2-multiple real root. Set $\si$ to be the  1-multiple real root, and $\chi_+=\chi_-$ the 2-multiple real roots, of $g(z)=0$ at $r_i$. Then,
it follows
\begin{equation*}
    g'(\si)|_{r=r_i}>0,\ \ \varphi|_{r=r_i}=0,
\end{equation*}
where $\varphi$ depending on $\si$ and $r$ is given by
\begin{equation*}
    \varphi=3\si^2+2ar \si -a^2r^2 +4(r+1).
\end{equation*}
Define
\begin{equation*}
    \de_0=\frac{1}{2}\min_{1\leq i\leq i_0} g'(\si)|_{r=r_i}>0.
\end{equation*}
By continuity, for any $\eta>0$ small enough, there are  $\eps_i$-neighborhoods $B(r_i,\eps_i)\subseteq [\eps,L]$ which do not intersect, such that
\begin{equation*}
    \min_{1\leq i\leq i_0}\inf_{B(r_i,\eps_i)}g'(\si)\geq \de_0,\ \ \max_{1\leq i\leq i_0}\sup_{B(r_i,\eps_i)}|\varphi|\leq \eta.
\end{equation*}
We now divide two cases to decide how to choose roots $\la_1,\la_2,\la_3$ and further make some essential estimates.

\medskip
\noindent{\it Case 1.} For each $1\leq i\leq i_0$, over $B(r_i,\eps_i)$, we choose $\si$ depending on $r_i$ as before, and set
\begin{eqnarray*}
  \la_1 &=& \si,\\
 \la_{2,3} &=&\chi_\pm= -\frac{\si+ar}{2}\pm \frac{1}{2}\sqrt{-\varphi}.
\end{eqnarray*}
One can check
\begin{equation*}
    \la_1-\la_{2,3}=\frac{3\si +ar}{2}\mp \frac{1}{2}\sqrt{-\varphi},
\end{equation*}
which implies
\begin{equation*}
    (\la_1-\la_2)(\la_1-\la_3)=\frac{1}{4} [(3\si+ar)^2-(-\varphi)]=g'(\si)\geq \de_0>0.
\end{equation*}
Moreover, one can choose $\eta>0$ small enough such that
\begin{equation*}
    \sqrt{\de_0-\frac{\eta}{4}}-\sqrt{\frac{\eta}{4}}>0.
\end{equation*}
Then, it holds that
\begin{equation*}
    \frac{|3\si +ar|}{2}\geq \sqrt{\de_0-\frac{|\varphi|}{4}}\geq \sqrt{\de_0-\frac{\eta}{4}}>\sqrt{\frac{\eta}{4}}>0,
\end{equation*}
and
\begin{equation*}
  \frac{|3\si +ar|}{2}-\sqrt{\frac{|\varphi|}{4}}\geq \sqrt{\de_0-\frac{\eta}{4}}-  \sqrt{\frac{\eta}{4}}>0.
\end{equation*}

\medskip
\noindent{\it Case 2.} Over $[\eps,L]\backslash \cup_{i=1}^{i_0} B(r_i,\eps_i)$, we choose $\si$ as the minimum (or maximum) real root of $g(z)=0$. Then, by this choice, one has
\begin{equation*}
    g'(\si)>0,\ \ |\varphi|>0,\ \ \forall\, r\in [\eps,R]\backslash \cup_{i=1}^{i_0} B(r_i,\eps_i).
\end{equation*}
This implies
\begin{equation*}
    (\la_1-\la_2)(\la_1-\la_3)=\frac{1}{4} [(3\si+ar)^2-(-\varphi)]=g'(\si)\geq \min_{[\eps,R]\backslash \cup_{i=1}^{i_0} B(r_i,\eps_i)}g'(\si)>0,
\end{equation*}
and
\begin{equation*}
    \min_{[\eps,R]\backslash \cup_{i=1}^{i_0} B(r_i,\eps_i)}|\varphi|>0.
\end{equation*}

We now decompose
\begin{eqnarray*}
  \D_1 &=& \D_{1,-}\cup \D_{1,0}\cup\D_{1,+}, \\
\D_{1,-} &=& \{\eps\leq r\leq L, r\notin \cup_{i=1}^{i_0} B(r_i,\eps_i), R\leq 0\},\\
\D_{1,+} &=& \{\eps\leq r\leq L, r\notin\cup_{i=1}^{i_0} B(r_i,\eps_i), R\geq 0\},\\
\D_{1,0} &=& \{\eps\leq r\leq L, r\in \cup_{i=1}^{i_0} B(r_i,\eps_i)\}.
\end{eqnarray*}
Notice that all domains $\D_{1,\pm}$ and $\D_{1,0}$ are bounded and closed, and
\begin{equation*}
    R|_{\D_{1,+}}>0,\ \ R|_{\D_{1,-}}<0.
\end{equation*}
In what follows we shall estimate $\CM_{11}$ only; other elements of $\CM$ can be estimated in the completely same way.

\medskip
\noindent{\it Over $\D_{1,0}$:} Rewrite $\CM_{11}$ as
\begin{multline*}
\CM_{11} =\frac{\la_1}{(\la_1-\la_2)(\la_1-\la_3)(\la_2+\la_3)}e^{\la_1 \tau}\FI_3\\
-\frac{1}
{(\la_1+\la_2)(\la_1+\la_3)}\int_0^1 \frac{d}{d\la}
\left[\frac{\la(\la_1+\la)}{\la_1-\la}
e^{\la\tau}\right]_{\la=(1-\theta)\la_2+\theta\la_3} d\theta\,\FI_3,
\end{multline*}
where
\begin{equation*}
  \frac{d}{d\la}
\left[\frac{\la(\la_1+\la)}{\la_1-\la}
e^{\la\tau}\right]=\frac{\tau \la (\la_1+\la)+\la_1+(\la_1+2)\la+\la^2}{(\la_1-\la)^2} e^{\la \tau}.
\end{equation*}
Then, one has
\begin{eqnarray}\label{est.m11}
&&|\CM_{11}| \leq \frac{|\la_1|}{|(\la_1-\la_2)(\la_1-\la_3)|\cdot |(\la_2+\la_3)|}e^{\tau \Re \la_1}\\
&&\hspace{1.5cm}+\frac{1}
{|(\la_1+\la_2)(\la_1+\la_3)|}\nonumber\\
&&\int_0^1
\left[\frac{\tau |\la (\la_1+\la)|+|\la_1+(\la_1+2)\la+\la^2|}{|\la_1-\la|^2} e^{\tau \Re \la }\right]_{\la=(1-\theta)\la_2+\theta\la_3}
d\theta.\nonumber
\end{eqnarray}
Notice
\begin{eqnarray*}
    &&\la_2+\la_3=-(\si+ar)<0,\\
&&\Re \la_1=\si<0,\\
&&\Re \{(1-\theta)\la_2+\theta\la_3\}<0.
\end{eqnarray*}
Also, notice
\begin{equation*}
    \la_1+\la_{2,3}=\frac{\si-ar}{2}\pm \frac{1}{2}\sqrt{-\varphi},
\end{equation*}
which implies
\begin{eqnarray*}
(\la_1+\la_2)(\la_1+\la_3)&=&\frac{(\si-ar)^2+\varphi}{4}\\
&=&\frac{(\si-ar)^2+3\si^2+2ar \si -a^2r^2 +4(r+1)}{4}\\
&=&\si^2+r+1>0.
\end{eqnarray*}
Consider the low bound of $|\la_1-\la|$ at $\la=(1-\theta)\la_2+\theta\la_3$ for all $0\leq \theta\leq 1$.

\medskip
\noindent{\it Case when $\la_{2,3}$ are real:} In this case,
\begin{equation*}
    \la_1-[(1-\theta)\la_2+\theta\la_3]=\frac{3\si+ar}{2}-(2\theta-1)\frac{\sqrt{-\varphi}}{2}
\end{equation*}
By the definition of $\D_{1,0}$,
\begin{equation*}
   |\la_1-[(1-\theta)\la_2+\theta\la_3]|\geq \frac{|3\si+ar|}{2}- \sqrt{\frac{|\varphi|}{4}}\geq \sqrt{\de_0-\frac{\eta}{4}}-  \sqrt{\frac{\eta}{4}}>0.
\end{equation*}

\medskip
\noindent{\it Case when $\la_{2,3}$ are non-real and thus complex conjugate:} In this case,
\begin{equation*}
   \la_1-[(1-\theta)\la_2+\theta\la_3]=\la_1-\Re \la_2 +(2\theta-1) i \Im \la_2,
\end{equation*}
which also implies
\begin{equation*}
    |\la_1-[(1-\theta)\la_2+\theta\la_3]|\geq |\la_1-\Re \la_2|= \frac{|3\si +ar|}{2}\geq \sqrt{\frac{\eta}{4}}>0.
\end{equation*}

\medskip
\noindent Therefore, whether $\la_{2,3}$ are real or not,
\begin{equation*}
    \min_{r\in \D_{1,0}, 0\leq \theta\leq 1} |\la_1-[(1-\theta)\la_2+\theta\la_3]|>0.
\end{equation*}
Then, putting all estimates above into \eqref{est.m11} gives
\begin{equation*}
    |\CM_{11}|\lesssim (1+\tau) e^{-O(1)\tau} \lesssim e^{-O(1)\tau}, \ \ r\in \D_{1,0}.
\end{equation*}

\medskip
\noindent {\it Over $\D_{1,+}$:} By noticing
\begin{equation*}
    \la_2+\la_3=2\Re \la_2=-(\si+ar)<0, \la_2-\la_3= 2 i \Im \la_2=i\sqrt{\varphi},\ \varphi>0,
\end{equation*}
it is straightforward to make the estimate
\begin{eqnarray*}
  |\CM_{11}| &\leq &\frac{|\la_1|}{|(\la_1-\la_2)(\la_1-\la_3)|\cdot |\si+ar|}e^{\tau \Re\la_1 }\\
&&+\frac{2}{|(\la_1+\la_2)(\la_1+\la_3)|\cdot \sqrt{\varphi}}
\frac{|\la_2(\la_1+\la_2)|}{|\la_1-\la_2|}e^{\tau \Re\la_2}\\\
&\leq &\frac{\max|\la_1|}{\min|(\la_1-\la_2)(\la_1-\la_3)|\cdot \min|\si+ar|}e^{\tau \min\si }\\
&&+\frac{2}{\min|(\la_1+\la_2)(\la_1+\la_3)|\cdot \sqrt{\min\varphi}}\\
&&\ \ \ \frac{\max|\la_2(\la_1+\la_2)(\la_1-\la_3)|}{\min|(\la_1-\la_2)(\la_1-\la_3)|}e^{\tau \min \frac{-(\si+ar)}{2}}\\
&\lesssim & e^{-O(1)\tau},
\end{eqnarray*}
where $\max,\min$ are taken over $\D_{1,+}$.

\medskip
\noindent {\it Over $\D_{1,-}$:} Since
\begin{equation*}
    \la_2+\la_3=2\Re \la_2=-(\si+ar)<0, \la_2-\la_3= \sqrt{-\varphi},\ \varphi<0,
\end{equation*}
it follows in the same way that
\begin{equation*}
     |\CM_{11}|\lesssim e^{-O(1)\tau}.
\end{equation*}

\medskip
Thus, for $\CM_{11}$, one has
\begin{equation*}
     |\CM_{11}|\lesssim e^{-O(1)\tau},\ \ r\in \D_1=\D_{1,0}\cup\CD_{1,+}\cup\CD_{1,-}.
\end{equation*}
The above estimate also holds true for all other elements of $\CM$ and details are omitted for simplicity.

Then, in a summary, we have proved

\begin{proposition}\label{prop.m.bdd}
Let the matrix $\CM=\CM(t,k)_{9\times 9}$ be defined in Proposition \ref{prop.m}. The elements of $\CM$ have the following upper bounds for pointwise $t\geq 0$ and $k\in \R^3$.
Over $\D_0$,
\begin{eqnarray*}
  |\CM| &\lesssim & \left(\begin{array}{ccc}
                      r & r^2 & r^{1/2} \\
                      r^2 & r^3 & r^{3/2} \\
                      r^{1/2} & r^{3/2} & 1
                    \end{array}\right)
e^{-O(1)r^2\tau}
+\left(\begin{array}{ccc}
                      1 & 1 & r^{1/2} \\
                      1 & 1 & r^{1/2} \\
                      r^{1/2} & r^{1/2} & r
                    \end{array}\right)
e^{-O(1)r\tau}.
\end{eqnarray*}
Over $\D_\infty$,
\begin{eqnarray*}
  |\CM| &\lesssim& \left(\begin{array}{ccc}
                      1 & r^{-1} & r^{-3/2} \\
                      r^{-1} & r^{-2} & r^{-5/2} \\
                      r^{-3/2} & r^{-5/2} & r^{-3}
                    \end{array}\right)
e^{-O(1)r\tau}
+\left(\begin{array}{ccc}
                      r^{-2} & r^{-1} & r^{-1} \\
                      r^{-1} & 1 & 1 \\
                      r^{-1} & 1 & 1
                    \end{array}\right)
e^{-\frac{O(1)\tau}{r}}.
\end{eqnarray*}
Over $\D_1$,
\begin{equation*}
     |\CM|\lesssim \left(\begin{array}{ccc}
                      1 &1 & 1 \\
                      1 & 1 & 1 \\
                     1 & 1 & 1
                    \end{array}\right)e^{-O(1)\tau}.
\end{equation*}
\end{proposition}

By combining Proposition \ref{prop.n.bdd} and Proposition \ref{prop.m.bdd} together with Theorem \ref{thm.gr}, we have the following

\begin{theorem}\label{thm.up}
Let $U=[n,u,E,B]$ satisfy the system \eqref{leq} for all $t>0$, $x\in \R^3$ with initial data $U|_{t=0}=U_0=[n_0,u_0,E_0,B_0]$.
Then, it holds that
\begin{equation*}
 \left(\begin{array}{c}
                      |\hat{n}| \\
                     |\hat{u}| \\
                     |\hat{E}|\\
                   |\hat{B}|
                    \end{array}\right)\lesssim \hat{G}^{\rm upp} \left(\begin{array}{c}
                      |\hat{n}_0| \\
                     |\hat{u}_0| \\
                     |\hat{E}_0|\\
                   |\hat{B}_0|
                    \end{array}\right)
\end{equation*}
for all $k\in \R^3$, where over $\D_0$,
\begin{equation*}
    \hat{G}^{\rm upp}= \left(\begin{array}{cccc}
                      1 &|k| & 0  & 0   \\
                    |k| & 1  & 1  & |k| \\
                     0  & 1  & 1  & |k| \\
                     0  &|k| &|k| &|k|^2
                    \end{array}\right)e^{-O(1)|k|^2 t}
+ \left(\begin{array}{cccc}
                      0 &0    &0   &0      \\
                      0 &|k|^2&|k|^4&|k|   \\
                     0  &|k|^4&|k|^6&|k|^3 \\
                     0  &|k|  &|k|^3&1
                    \end{array}\right)e^{-O(1)|k|^4 t};
\end{equation*}
over $\D_\infty$
\begin{multline*}
\dis  \hat{G}^{\rm upp}= \left(\begin{array}{cccc}
                      1       &|k|^{-1}&0       &0      \\
                      |k|^{-1}&|k|^{-2}&|k|^{-2}&0      \\
                      0       &|k|^{-2}&1       &0      \\
                      0       &0       &0       &0
                    \end{array}\right)e^{-O(1)t}\\
\dis
+\left(\begin{array}{cccc}
                      |k|^{-2}&|k|^{-1}&0       &0       \\
                      |k|^{-1}&1       &|k|^{-2}&|k|^{-3}\\
                      0       &|k|^{-2}&|k|^{-2}&|k|^{-5}\\
                      0       &|k|^{-3}&|k|^{-5}&|k|^{-6}
                    \end{array}\right)e^{-O(1)|k|^2t}
                    +\left(\begin{array}{cccc}
                      0       &0       &0       &0       \\
                      0       &|k|^{-4}&|k|^{-2}&|k|^{-2}\\
                      0       &|k|^{-2}&1       &1       \\
                      0       &|k|^{-2}&1       &1
                    \end{array}\right)e^{-\frac{O(1) t}{|k|^2}};
\end{multline*}
over $\D_1$,
\begin{equation*}
     \hat{G}^{\rm upp}= \left(\begin{array}{cccc}
                      1       &1       &0       &0      \\
                      1       &1       &1       &1      \\
                      0       &1       &1       &1      \\
                      0       &1       &1       &1
                    \end{array}\right)e^{-O(1)t}.
\end{equation*}
\end{theorem}

\subsection{Refined time-decay property} From Theorem \ref{thm.up}, for any integer $j\geq 0$, one has the following estimates for $t\geq 0$, $k\in \R^3$. For $n$,
\begin{eqnarray*}
  |\CF (\na^j n)| &\lesssim& (|k|^j|\hat{n}_0|+|k|^{j+1}|\hat{u}_0|)e^{-c|k|^2 t}1_{\D_0}\\
&&+|k|^j(|\hat{n}_0|+|\hat{u}_0|)e^{-ct}1_{\D_1} \\
&&+(|k|^j|\hat{n}_0|+|k|^{j-1}|\hat{u}_0|)e^{-c t}1_{\D_\infty}.
\end{eqnarray*}
Here, for later use, we use $\hat{n}_0=-\frac{\ga}{\be}ik\cdot \hat{E}_0$ to replace $\hat{n}_0$ in the first term on the right-hand side so that
\begin{eqnarray*}
  |\CF (\na^j n)| &\lesssim& (|k|^{j+1}|\hat{E}_0|+|k|^{j+1}|\hat{u}_0|)e^{-c|k|^2 t}1_{\D_0}\\
&&+|k|^j(|\hat{n}_0|+|\hat{u}_0|)e^{-ct}1_{\D_1} \\
&&+(|k|^j|\hat{n}_0|+|k|^{j-1}|\hat{u}_0|)e^{-c t}1_{\D_\infty}.
\end{eqnarray*}
For $u$,
\begin{eqnarray*}
  |\CF (\na^j u)| &\lesssim& (|k|^{j+1}|\hat{n}_0|+|k|^{j}|\hat{u}_0|+|k|^j|\hat{E}_0|+|k|^{j+1}|\hat{B}_0|)e^{-c|k|^2 t}1_{\D_0}\\
&&+(|k|^{j+2}|\hat{u}_0|+|k|^{j+4}|\hat{E}_0|+|k|^{j+1}|\hat{B}_0|)e^{-c|k|^4 t}1_{\D_0}\\
&&+(|k|^{j-1}|\hat{n}_0|+|k|^{j}|\hat{u}_0|+|k|^{j-2}|\hat{E}_0|+|k|^{j-3}|\hat{B}_0|)e^{-ct}1_{\D_1\cup \D_\infty}\\
&&+(|k|^{j-4}|\hat{u}_0|+|k|^{j-2}|\hat{E}_0|+|k|^{j-2}|\hat{B}_0|)e^{-\frac{ct}{|k|^2}}1_{\D_\infty}.
\end{eqnarray*}
For $E$ and $B$,
\begin{eqnarray*}
  |\CF (\na^j E)| &\lesssim& (|k|^{j}|\hat{u}_0|+|k|^j|\hat{E}_0|+|k|^{j+1}|\hat{B}_0|)e^{-c|k|^2 t}1_{\D_0}\\
&&+(|k|^{j+4}|\hat{u}_0|+|k|^{j+6}|\hat{E}_0|+|k|^{j+3}|\hat{B}_0|)e^{-c|k|^4 t}1_{\D_0}\\
&&+(|k|^{j-2}|\hat{u}_0|+|k|^{j}|\hat{E}_0|+|k|^{j-5}|\hat{B}_0|)e^{-ct}1_{\D_1\cup \D_\infty}\\
&&+(|k|^{j-2}|\hat{u}_0|+|k|^{j}|\hat{E}_0|+|k|^{j}|\hat{B}_0|)e^{-\frac{ct}{|k|^2}}1_{\D_\infty},
\end{eqnarray*}
and
\begin{eqnarray*}
  |\CF (\na^j B)| &\lesssim& (|k|^{j+1}|\hat{u}_0|+|k|^{j+1}|\hat{E}_0|+|k|^{j+2}|\hat{B}_0|)e^{-c|k|^2 t}1_{\D_0}\\
&&+(|k|^{j+1}|\hat{u}_0|+|k|^{j+3}|\hat{E}_0|+|k|^{j}|\hat{B}_0|)e^{-c|k|^4 t}1_{\D_0}\\
&&+(|k|^{j-3}|\hat{u}_0|+|k|^{j-5}|\hat{E}_0|+|k|^{j-6}|\hat{B}_0|)e^{-ct}1_{\D_1\cup \D_\infty}\\
&&+(|k|^{j-2}|\hat{u}_0|+|k|^{j}|\hat{E}_0|+|k|^{j}|\hat{B}_0|)e^{-\frac{ct}{|k|^2}}1_{\D_\infty}.
\end{eqnarray*}

Now, for brevity of presentation, we introduce a notation related to the rate index over the low-frequency domain as follows. For integer $j\geq 0$ and $1\leq p\leq q\leq \infty$, define
\begin{equation*}
    \de(j,p,q)=\frac{j}{2}+\frac{3}{2}(\frac{1}{p}-\frac{1}{q})
\end{equation*}
which corresponds to the usual value in the case of the heat kernel in $\R^3$.
Then, by applying Lemma \ref{lem.glidecay} to those estimates above,  it is straightforward to obtain

\begin{theorem}\label{thm.tdp}
Let $U=[n,u,E,B]$ satisfy the system \eqref{leq} for all $t>0$, $x\in \R^3$ with initial data $U|_{t=0}=U_0=[n_0,u_0,E_0,B_0]$. Let $1\leq p,s\leq 2\leq q\leq \infty$ and let $\ell\geq 0$. Then, for any integer $j\geq 0$, the following time decay property of $U=[n,u,E,B]$  holds true:
\begin{eqnarray*}
  \|\na^j n\|_{L^q} &\lesssim& (1+t)^{-\de(j+1,p,q)}\|E_0\|_{L^p}+(1+t)^{-\de(j+1,p,q)}\|u_0\|_{L^p}\\
&&+ e^{-ct}\|\na^{m(j,s,q)}n_0\|_{L^s}+e^{-ct}\|\na^{m(j-1,s,q)}u_0\|_{L^s},
\end{eqnarray*}
\begin{eqnarray*}
  \|\na^j u\|_{L^q} &\lesssim& (1+t)^{-\de(j+1,p,q)}\|[n_0,B_0]\|_{L^p}+(1+t)^{-\de(j,p,q)}\|[u_0,E_0]\|_{L^p}\\
&&+(1+t)^{-\frac{1}{2}\de(j+2,p,q)}\|u_0\|_{L^p}+(1+t)^{-\frac{1}{2}\de(j+4,p,q)}\|E_0\|_{L^p}\\
&&+(1+t)^{-\frac{1}{2}\de(j+1,p,q)}\|B_0\|_{L^p}\\
&&+e^{-ct}(\|\na^{m(j-1,s,q)}n_0\|_{L^s}
+\|\na^{m(j,s,q)}u_0\|_{L^s}\\
&&+\|\na^{m(j-2,s,q)}E_0\|_{L^s}
+\|\na^{m(j-3,s,q)}B_0\|_{L^s})\\
&&+ (1+t)^{-\frac{\ell}{2}}(\|\na^{m(j-4+\ell,s,q)}u_0\|_{L^s}+\|\na^{m(j-2+\ell,s,q)}[E_0,B_0]\|_{L^s}),
\end{eqnarray*}
\begin{eqnarray*}
  \|\na^j E\|_{L^q} &\lesssim& (1+t)^{-\de(j,p,q)}\|[u_0,E_0]\|_{L^p}+(1+t)^{-\de(j+1,p,q)}\|B_0\|_{L^p}\\
&&+(1+t)^{-\frac{1}{2}\de(j+4,p,q)}\|u_0\|_{L^p}+(1+t)^{-\frac{1}{2}\de(j+6,p,q)}\|E_0\|_{L^p}\\
&&+(1+t)^{-\frac{1}{2}\de(j+3,p,q)}\|B_0\|_{L^p}\\
&&+e^{-ct}(\|\na^{m(j-2,s,q)}u_0\|_{L^s}+\|\na^{m(j,s,q)}E_0\|_{L^s}
+\|\na^{m(j-5,s,q)}B_0\|_{L^s})\\
&&+ (1+t)^{-\frac{\ell}{2}}(\|\na^{m(j-2+\ell,s,q)}u_0\|_{L^s}+\|\na^{m(j+\ell,s,q)}[E_0,B_0]\|_{L^s})
\end{eqnarray*}
and
\begin{eqnarray*}
  \|\na^j B\|_{L^q} &\lesssim& (1+t)^{-\de(j+1,p,q)}\|[u_0,E_0]\|_{L^p}+(1+t)^{-\de(j+2,p,q)}\|B_0\|_{L^p}\\
&&+(1+t)^{-\frac{1}{2}\de(j+1,p,q)}\|u_0\|_{L^p}+(1+t)^{-\frac{1}{2}\de(j+3,p,q)}\|E_0\|_{L^p}\\
&&+(1+t)^{-\frac{1}{2}\de(j,p,q)}\|B_0\|_{L^p}\\
&&+e^{-ct}(\|\na^{m(j-3,s,q)}u_0\|_{L^s}+\|\na^{m(j-5,s,q)}E_0\|_{L^s}
+\|\na^{m(j-6,s,q)}B_0\|_{L^s})\\
&&+ (1+t)^{-\frac{\ell}{2}}(\|\na^{m(j-2+\ell,s,q)}u_0\|_{L^s}+\|\na^{m(j+\ell,s,q)}[E_0,B_0]\|_{L^s})
\end{eqnarray*}
for any $t\geq 0$.
\end{theorem}

For later use, we need the following result which is an immediate corollary from Theorem \ref{thm.tdp}.

\begin{corollary}\label{cor.ld}
Under the assumption of Theorem \ref{thm.tdp}, one has
\begin{equation*}
  \|n\| \lesssim (1+t)^{-\frac{5}{4}} \|[n_0,u_0]\|_{L^1\cap L^2},
\end{equation*}
\begin{equation*}
  \|u\| \lesssim (1+t)^{-\frac{5}{4}}\|n_0\|_{L^1\cap L^2}+ (1+t)^{-\frac{3}{4}}\|[u_0,E_0]\|_{L^1\cap L^2}
  +(1+t)^{-\frac{5}{8}}\|B_0\|_{L^1\cap L^2},
\end{equation*}
\begin{multline*}
  \|E\| \lesssim (1+t)^{-\frac{3}{4}}\|u_0\|_{L^1\cap L^2}+ (1+t)^{-\frac{3}{4}}(\|E_0\|_{L^1\cap L^2}+\|\na^2 E_0\|)\\
+(1+t)^{-\frac{9}{8}}(\|B_0\|_{L^1\cap L^2}+\|\na^3 B_0\|),
\end{multline*}
and
\begin{multline*}
  \|B\| \lesssim (1+t)^{-\frac{5}{8}}\|u_0\|_{L^1\cap L^2}+ (1+t)^{-\frac{9}{8}}(\|E_0\|_{L^1\cap L^2}+\|\na^3 E_0\|)\\
+(1+t)^{-\frac{3}{8}}(\|B_0\|_{L^1\cap L^2}+\|\na B_0\|),
\end{multline*}
for any $t\geq 0$. Moreover, one also has
\begin{equation*}
  \|\na B\| \lesssim (1+t)^{-\frac{7}{8}}\|u_0\|_{L^1\cap L^2}+ (1+t)^{-\frac{5}{8}}(\|[E_0,B_0]\|_{L^1\cap L^2}+\|\na^3 [E_0,B_0]\|)
\end{equation*}
for any $t\geq 0$.
\end{corollary}

\section{Nonlinear asymptotic stability}\label{sec6}

In this section we are prepared to prove Theorem \ref{thm.main} concerning the global existence, uniqueness and time decay rates of solutions to the reformulated Cauchy problem \eqref{eq}, \eqref{eq.2}, \eqref{eq.id} for the nonlinear Navier-Stokes-Maxwell system. We first prove the global existence and uniqueness by only deriving the key uniform-in-time a priori energy estimates, and then apply the time decay property of the linearized system together with the further high-order energy estimates as well as the bootstrap argument to obtain the desired time decay rates of solutions.

\subsection{Global existence}
First of all, still using the notation $[n,u,E,B]$ for simplicity, let us write the reformulated nonlinear system \eqref{eq}, \eqref{eq.2} as
\begin{equation}\label{ge.e}
    \left\{\begin{array}{l}
     \dis \pa_t n +\ga \na\cdot u =h_1,\\
      \dis \pa_t u+\ga \na n +\be {E}-\mu\De u=h_2\\
       \dis \pa_t {E} - \na \times {B} -\be u = h_3,\\
      \dis \pa_t {B} + \na \times {E} = 0,\\
      \dis \na \cdot{E} = -\frac{\be}{\ga}n, \ \ \na\cdot {B} =0,
    \end{array}\right.
\end{equation}
where $\na\cdot h_3=-\frac{\be}{\ga}h_1$ and the nonlinear source terms $h_1,h_2,h_3$ behave as
\begin{eqnarray*}
h_1&\sim& \na \cdot (nu),\\
h_2&\sim& u\cdot \na u + n \na n +u\times B + n \De u,\\
h_3&\sim& nu.
\end{eqnarray*}
Initial data of the system is given by
\begin{equation}\label{ge.e.id}
    n|_{t=0}=n_0,\ u|_{t=0}=u_0,\ E|_{t=0}=E_0,\ B|_{t=0}=B_0
\end{equation}
satisfying
\begin{equation}\label{ge.e.idc}
  \na \cdot{E}_0 = -\frac{\be}{\ga}n_0, \ \ \na\cdot {B}_0 =0.
\end{equation}
Then, we have

\begin{theorem}\label{thm.ge}
Let $N\geq 4$. Assume \eqref{ge.e.idc} for initial data $U_0:=[n_0,u_0,E_0,B_0]$. If $\|U_0\|_{H^N}$ is sufficiently small, then the Cauchy problem \eqref{ge.e}, \eqref{ge.e.id} admits a unique global solution $U=[n,u,E,B]$ with
\begin{eqnarray*}
&&\dis U\in C([0,\infty);H^N(\R^3)),\\
&& n\in L^2((0,\infty);H^N(\R^3)),\
\na u\in L^2((0,\infty);H^N(\R^3)),\\
&&\na E\in L^2((0,\infty);H^{N-2}(\R^3)),\
\na^2 B\in L^2((0,\infty);H^{N-3}(\R^3)),
\end{eqnarray*}
and
\begin{multline}\label{thm.ge.1}
    \|U(t)\|_{H^N}^2+\int_0^t (\|n(s)\|_{H^N}^2+\|\na u(s)\|_{H^{N}}^2\\
    +\|\na E(s)\|_{H^{N-2}}^2+\|\na^2 B(s)\|_{H^{N-3}}^2)ds
    \leq C\|U_0\|_{H^N}^2
\end{multline}
for any $t\geq 0$.
\end{theorem}

\begin{proof}
We only consider the proof of the following uniform-in-time a priori estimate.
Define
\begin{eqnarray*}
\CE_N(U(t))&=&\|[n,u,E,B]\|_{H^N}^2
+\kappa_1 \sum_{|\al|\leq N-1} \lag \pa^\al (\ga \na n)\cdot \pa^\al u \rag  \\
&&+\kappa_1 \sum_{|\al|\leq N-2} \lag \pa^\al \na \times E \cdot\pa^\al \na\times u \rag   \\
&&+\kappa_1\kappa_2\sum_{1\leq |\al|\leq N-2} \lag  \pa^\al (-\na\times B)\cdot \pa^\al E \rag  ,
\end{eqnarray*}
and
\begin{equation*}
    \CD_N(U(t))=\|n\|_{H^N}^2+\|\na u\|_{H^N}^2+\|\na E\|_{H^{N-2}}^2+\|\na^2 B\|_{H^{N-3}}^2.
\end{equation*}
Here, constants $0<\kappa_1,\kappa_2\ll 1$ are properly chosen as in \eqref{def.CEf}.
Then, for any smooth solution $U$ to the Cauchy problem \eqref{ge.e}, \eqref{ge.e.id} over $0\leq t\leq T$ with $T>0$, one has
\begin{equation}\label{thm.ge.p1}
    \frac{d}{dt}\CE_N(U(t))+c\CD_N (U(t)) \leq C(\|n\|_{L^\infty})\CE_N(U(t))^{1/2}\CD_N(U(t))
\end{equation}
for any $0\leq t\leq T$. As long as the above estimate is proved, Theorem \ref{thm.ge} follows in the standard way by combining it with the local-in-time existence and uniqueness as well as the continuity argument. Therefore, in what follows we prove \eqref{thm.ge.p1} only, and other details are omitted for simplicity.

In fact, zero-order energy estimate implies
\begin{multline*}
\dis \frac{1}{2}\frac{d}{dt}\|[n,u,E,B]\|_{H^N}^2 +\mu \|\na u\|_{H^N}^2\\
\dis =\sum_{|\al|\leq N}
(\lag\pa^\al h_1,\pa^\al n \rag +\lag  \pa^\al h_2\cdot \pa^\al u \rag  +\lag  \pa^\al h_3 \cdot \pa^\al E \rag  ).
\end{multline*}
In the same way as in the proof of Theorem \ref{thm.ly}, from the first two equations of \eqref{ge.e} and $\na\cdot E=-\frac{\be}{\ga}n$, one has
\begin{multline*}
\dis \frac{d}{dt} \sum_{|\al|\leq N-1} \lag \pa^\al (\ga \na n),\pa^\al u \rag
+ c\sum_{|\al|\leq N-1} \int_{\R^3} (\ga^2|\pa^\al \na n|^2+\be^2|\pa^\al n|^2)dx \\
\dis \leq C \|\na u\|_{H^N}^2 + \sum_{|\al|\leq N-1}\lag \pa^\al (\ga \na h_1), \pa^\al u \rag
+ \sum_{|\al|\leq N-1}\lag \pa^\al (\ga \na n), \pa^\al h_2 \rag.
\end{multline*}
By also using the curl of the second equation of \eqref{ge.e}, i.e.
\begin{equation*}
    \pa_t \na\times u +\be \na\times E -\mu \De \na \times u=\na\times h_2
\end{equation*}
together with the third equation of \eqref{ge.e}, it follows
\begin{multline*}
\dis \frac{d}{dt} \sum_{|\al|\leq N-2} \lag \pa^\al \na \times E,\pa^\al \na\times u \rag
+c\sum_{|\al|\leq N-2}\int_{\R^3}|\pa^\al \na\times E|^2dx  \\
\dis \leq \eps \sum_{1\leq |\al|\leq N-2}\int_{\R^3}|\pa^\al \na\times B|^2dx   +\frac{C}{\eps}\|\na u\|_{H^N}^2
\dis \\
+ \sum_{|\al|\leq N-2}\lag \pa^\al\na\times h_3, \pa^\al \na\times u \rag
+ \sum_{|\al|\leq N-2}\lag \pa^\al \na\times E, \pa^\al\na\times h_2 \rag
\end{multline*}
for an arbitrary  constant $0<\eps\leq 1$. Moreover, the combination of the third and fourth equations in \eqref{ge.e} gives
\begin{multline*}
\dis \frac{d}{dt} \sum_{1\leq |\al|\leq N-2} \lag  \pa^\al (-\na\times B), \pa^\al E \rag
+c\sum_{1\leq |\al|\leq N-2}\int_{\R^3} |\pa^\al \na\times B|^2dx  \\
\dis \leq \sum_{1\leq |\al|\leq N-2}\int_{\R^3} |\pa^\al \na\times E|^2dx   +C\sum_{1\leq |\al|\leq N-2}\int_{\R^3} |\pa^\al u|^2dx \\
\dis -\sum_{1\leq |\al|\leq N-2}\lag \pa^\al \na\times B , \pa^\al h_3 \rag  .
\end{multline*}
Then, by putting the above estimates together for properly chosen constants $0<\kappa_1,\kappa_2,\eps\ll 1$, one has
\begin{eqnarray}\label{thm.ge.p2}
&&\ \ \   \frac{d}{dt}\CE_N(U(t))+c\CD_N (U(t)) \\
&&\leq  \sum_{|\al|\leq N}(\lag  \pa^\al h_1, \pa^\al n \rag
+\lag  \pa^\al h_2, \pa^\al u \rag
+\lag  \pa^\al h_3, \pa^\al E \rag  )\nonumber\\
&&\ \ \ + \kappa_1\sum_{|\al|\leq N-1}\lag \pa^\al (\ga \na h_1), \pa^\al u \rag
+ \kappa_1\sum_{|\al|\leq N-1}\lag \pa^\al (\ga \na n), \pa^\al h_2 \rag  \nonumber\\
&&\ \ \ +\kappa_1\sum_{|\al|\leq N-2}\lag \pa^\al \na\times E, \pa^\al\na\times h_2 \rag
+ \kappa_1\sum_{|\al|\leq N-2}\lag \pa^\al\na\times h_3, \pa^\al \na\times u \rag\nonumber\\
&&\ \ \ -\kappa_1\kappa_2\sum_{1\leq |\al|\leq N-2}\lag \pa^\al \na\times B , \pa^\al h_3 \rag.\nonumber
\end{eqnarray}
The rest is a long and standard procedure to verify by using the Leibniz formula and the Sobolev inequalities \cite{Adams,Ta} that the whole right-hand term of the above inequality is bounded by $C(\|n\|_{L^\infty})\CE_N(U(t))^{1/2}\CD_N(U(t))$, where $C$ depends only on $\|n\|_{L^\infty}$; for simplicity, we also omit the proof's details of the rest part, cf.~\cite{DLUY} and \cite{Du-EM}. The proof of Theorem \ref{thm.ge} is complete.
\end{proof}

\subsection{Large-time behavior} From now, we suppose that all conditions in Theorem \ref{thm.ge} hold and $U=[n,u,E,B]$ is the obtained solution to the Cauchy problem \eqref{ge.e}, \eqref{ge.e.id} with the condition \eqref{ge.e.idc}. In this subsection we devote ourselves to  proving the time decay rate of the full energy $\|U(t)\|_{H^N}^2$ or equivalently $\CE_N(U(t))$. For that purpose, define
\begin{equation*}
    X(t)=\sup_{0\leq s\leq t} (1+s)^{\frac{3}{4}}\CE_N(U(s)),\ \ t\geq 0.
\end{equation*}
In fact, we have

\begin{lemma}\label{lem.X}
If $\|U_0\|_{L^1\cap H^{N+1}}$ is sufficiently small, then
\begin{equation}\label{lem.X.1}
    \sup_{t\geq 0} X(t)\leq C\|U_0\|_{L^1\cap H^{N+1}}^2.
\end{equation}

\end{lemma}

\begin{proof}
Under smallness assumption of $\|U_0\|_{H^N}$, \eqref{thm.ge.p1} implies
\begin{equation}\label{lem.X.p1}
    \frac{d}{dt}\CE_N(U(t))+c\CD_N (U(t)) \leq 0
\end{equation}
for any $t\geq 0$. This is the starting point to deduce \eqref{lem.X.1}. In fact, fix a constant $\eps>0$ small enough. Then, the further time weighted estimate on \eqref{lem.X.p1} gives
\begin{multline*}
\dis (1+t)^{\frac{3}{4}+\eps}\CE_N(U(t))+c\int_0^t(1+s)^{\frac{3}{4}+\eps}\CD_N(U(s)) ds\\
\dis \leq \CE_N(U_0)+ (\frac{3}{4}+\eps) \int_0^t(1+s)^{-\frac{1}{4}+\eps}\CE_N(U(s))ds.
\end{multline*}
Noticing
\begin{multline*}
  \CE_{N}(U) \sim  \|U\|_{H^N}^2 \leq   \CD_{N+1}(U)+\|[u,E,B]\|^2+\|\na B\|^2\\
\leq C\CD_{N+1}(U)+\|[u,E,B]\|^2
\end{multline*}
and
\begin{equation*}
    \int_0^t\CD_{N+1}(U(s))ds\leq C \CE_{N+1}(U_0),
\end{equation*}
it follows that
\begin{multline}\label{lem.X.p2}
\dis (1+t)^{\frac{3}{4}+\eps}\CE_N(U(t))+\int_0^t(1+s)^{\frac{3}{4}+\eps}\CD_N(U(s)) ds\\
\dis \leq C\CE_{N+1}(U_0)+ C \int_0^t(1+s)^{-\frac{1}{4}+\eps}\|[u,E,B]\|^2 ds.
\end{multline}
Due to Corollary \ref{cor.ld}, one has
\begin{eqnarray*}
\|B\| &\leq & C(1+t)^{-\frac{3}{8}}\|[u_0,E_0,B_0\|_{L^1\cap H^3}\\
&&+C\int_0^t(1+t-s)^{-\frac{5}{8}} \|h_2(s)\|_{L^1\cap L^2}ds\\
&&+C\int_0^t(1+t-s)^{-\frac{9}{8}} (\|h_3(s)\|_{L^1\cap L^2}+\|\na^3 h_3(s)\|)ds,
\end{eqnarray*}
\begin{eqnarray*}
\|E\| &\leq & C(1+t)^{-\frac{3}{4}}\|[u_0,E_0,B_0\|_{L^1\cap H^3}\\
&&+C\int_0^t(1+t-s)^{-\frac{3}{4}} \|h_2(s)\|_{L^1\cap L^2}ds\\
&&+C\int_0^t(1+t-s)^{-\frac{3}{4}} (\|h_3(s)\|_{L^1\cap L^2}+\|\na^2 h_3(s)\|)ds,
\end{eqnarray*}
and
\begin{eqnarray*}
\|u\| &\leq & C(1+t)^{-\frac{5}{8}}\|U_0\|_{L^1\cap L^2}\\
&&+C\int_0^t(1+t-s)^{-\frac{5}{4}} \|h_1(s)\|_{L^1\cap L^2}ds\\
&&+C\int_0^t(1+t-s)^{-\frac{3}{4}} \|[h_2(s),h_3(s)]\|_{L^1\cap L^2}ds.
\end{eqnarray*}
It is straightforward to verify
\begin{equation*}
   \|[h_1,h_2,h_3\|_{L^1\cap L^2}+\|h_3\|_{H^3}\leq C \CE_N(U).
\end{equation*}
Then, $\|B\|$ is estimated by
\begin{eqnarray*}
\|B\| &\leq & C(1+t)^{-\frac{3}{8}}\|[u_0,E_0,B_0\|_{L^1\cap H^3}\\
&&+C\int_0^t(1+t-s)^{-\frac{5}{8}}(1+s)^{-\frac{3}{4}}ds X(t)\\
&\leq & C(1+t)^{-\frac{3}{8}}(\|[u_0,E_0,B_0\|_{L^1\cap H^3}+X(t)),
\end{eqnarray*}
and in the same way, it holds that
\begin{equation*}
  \|[u,E,B]\| \leq  C(1+t)^{-\frac{3}{8}}(\|U_0\|_{L^1\cap H^3}+X(t)).
\end{equation*}
Therefore, one has
\begin{eqnarray*}
 && \int_0^t(1+s)^{-\frac{1}{4}+\eps}\|[u,E,B]\|^2 ds \\
&&\leq  C\int_0^t(1+s)^{-\frac{1}{4}+\eps}(1+s)^{-\frac{3}{4}} ds (\|U_0\|_{L^1\cap H^3}^2+X(t)^2)\\
&&\leq C(1+t)^\eps (\|U_0\|_{L^1\cap H^3}^2+X(t)^2).
\end{eqnarray*}
Substituting it into \eqref{lem.X.p2} gives
\begin{multline}\label{lem.X.p3}
\dis (1+t)^{\frac{3}{4}+\eps}\CE_N(U(t))+\int_0^t(1+s)^{\frac{3}{4}+\eps}\CD_N(U(s)) ds\\
\dis \leq C(1+t)^\eps (\|U_0\|_{L^1\cap H^{N+1}}^2+X(t)^2)
\end{multline}
which implies
\begin{equation*}
    X(t)\leq C(\|U_0\|_{L^1\cap H^{N+1}}^2+X(t)^2).
\end{equation*}
Since $\|U_0\|_{L^1\cap H^{N+1}}$ is sufficiently small, $X(t)$ is bounded uniformly in time and also \eqref{lem.X.1} holds true. This completes the proof Lemma \ref{lem.X}.
\end{proof}

\subsection{Optimal large-time behavior}

According to Corollary \ref{cor.ld} in the linear case, the time rate obtained in Lemma \ref{lem.X} is not optimal with respect to $n$, $u$ and $E$ in the solution $U=[n,u,E,B]$. In fact, the further estimates imply the following improved result.

\begin{theorem}\label{thm.odecay}
Let $N\geq 4$. Assume that initial data $U_0:=[n_0,u_0,E_0,B_0]$ satisfies \eqref{ge.e.idc} and $\|U_0\|_{L^1\cap H^{N+2}}$ is sufficiently small. Let $U=[n,u,E,B]$  be the solution to the Cauchy problem \eqref{ge.e}, \eqref{ge.e.id},  \eqref{ge.e.idc} obtained in Theorem \ref{thm.ge}. Then,
\begin{eqnarray*}
  \|n(t)\| &\lesssim& (1+t)^{-1}\|U_0\|_{L^1\cap H^{N+2}},\\
    \|u(t)\| &\lesssim& (1+t)^{-\frac{5}{8}}\|U_0\|_{L^1\cap H^{N+2}},\\
      \|E(t)\| &\lesssim& (1+t)^{-\frac{3}{4}}\ln (3+t)\|U_0\|_{L^1\cap H^{N+2}},\\
        \|B(t)\| &\lesssim& (1+t)^{-\frac{3}{8}}\|U_0\|_{L^1\cap H^{N+1}},
\end{eqnarray*}
for any $t\geq 0$.
\end{theorem}

We shall prove the above theorem as follows. The estimate of $B$ in fact has been obtained in  Lemma \ref{lem.X}, the estimate of $u$ will be  given in Lemma \ref{lem.Y} and the estimates of $n$ and $E$ will be given in Lemma \ref{lem.nE}.

Thus, we first consider the time decay estimate of  $u$. Before that, we need a high-order energy inequality. Define
\begin{eqnarray*}
\CE_N^{\rm h}(U(t))&=&\|\na [n,u,E,B]\|_{H^{N-1}}^2
+\kappa_1 \sum_{1\leq |\al|\leq N-1} \lag \pa^\al (\ga \na n), \pa^\al u \rag  \\
&&+\kappa_2 \sum_{1\leq |\al|\leq N-2} \lag \pa^\al \na \times E ,\pa^\al \na\times u \rag   \\
&&+\kappa_1\kappa_2\sum_{2\leq |\al|\leq N-2} \lag  \pa^\al (-\na\times B), \pa^\al E \rag  ,
\end{eqnarray*}
and
\begin{equation*}
    \CD_N^{\rm h}(U(t))=\|\na n\|_{H^{N-1}}^2+\|\na^2 u\|_{H^{N-1}}^2+\|\na^2 E\|_{H^{N-3}}^2+\|\na^3 B\|_{H^{N-4}}^2.
\end{equation*}
Here, constants $0<\kappa_1,\kappa_2\ll 1$ are properly chosen as in \eqref{def.CEf}. Then, we have

\begin{lemma}\label{lem.hee}
If $\|U_0\|_{H^N}$ is sufficiently small, then
\begin{equation}\label{lem.hee.1}
    \frac{d}{dt}\CE_N^{\rm h}(U(t))+c  \CD_N^{\rm h}(U(t))\leq C(\|u\|^2+\CE_N^{\rm h}(U(t)))\CE_N^{\rm h}(U(t))
\end{equation}
for any $t\geq 0$.
\end{lemma}

\begin{proof}
Similar to obtain \eqref{thm.ge.p2}, from the system \eqref{ge.e}, one has
\begin{eqnarray*}
&&\ \ \   \frac{d}{dt}\CE_N^{\rm h}(U(t))+c\CD_N^{\rm h} (U(t)) \\
&&\ \ \ \leq  \sum_{1\leq |\al|\leq N}(\lag  \pa^\al h_1, \pa^\al n \rag
+\lag  \pa^\al h_2, \pa^\al u \rag
+\lag  \pa^\al h_3, \pa^\al E \rag  )\\
&&\ \ \ + \kappa_1\sum_{1\leq |\al|\leq N-1}\lag \pa^\al (\ga \na h_1), \pa^\al u \rag
+ \kappa_1\sum_{1\leq |\al|\leq N-1}\lag \pa^\al (\ga \na n), \pa^\al h_2 \rag  \\
&&\ \ \ +\kappa_1\sum_{1\leq |\al|\leq N-2}\lag \pa^\al \na\times E, \pa^\al\na\times h_2 \rag
+\kappa_1\sum_{1\leq |\al|\leq N-2}\lag \pa^\al \na\times h_3, \pa^\al\na\times u\rag\\
&&\ \ \ -\kappa_1\kappa_2\sum_{2\leq |\al|\leq N-2}\lag \pa^\al \na\times B , \pa^\al h_3 \rag,
\end{eqnarray*}
which together with $h_1=-\frac{\ga}{\be} \na \cdot h_3$ imply
\begin{multline}\label{lem.hee.p1}
\dis\frac{d}{dt}\CE_N^{\rm h}(U(t))+c\CD_N^{\rm h} (U(t)) \\
\dis \lesssim  \sum_{1\leq |\al|\leq N}(\lag  \pa^\al h_1, \pa^\al n \rag
+\lag  \pa^\al h_2, \pa^\al u \rag
+\lag  \pa^\al h_3, \pa^\al E \rag  )
+\|\na [h_2,h_3]\|_{H^{N-2}}^2.
\end{multline}
It is easy to see
\begin{equation*}
    \|\na [h_2,h_3]\|_{H^{N-2}}^2\lesssim \CE_N^{\rm h}(U(t))\left[\CE_N^{\rm h}(U(t))+\CD_N^{\rm h}(U(t))\right].
\end{equation*}
In what follows we estimate the right-hand first term of \eqref{lem.hee.p1}. Let $1\leq |\al|\leq N$. First, it holds that
\begin{equation*}
  \lag  \pa^\al h_3, \pa^\al E \rag= \lag  \pa^\al (nu), \pa^\al E \rag \lesssim \|\na n\|_{H^{n-1}}\|\na u\|_{H^{N-1}}\|\na E\|_{H^{N-1}}.
\end{equation*}
Next, for the term containing $h_2$, one has
\begin{eqnarray*}
&& \lag  \pa^\al h_2, \pa^\al u \rag\\
&&\sim \lag \pa^\al (u\cdot \na u), \pa^\al u\rag +\lag \pa^\al (n \na n),\pa^\al u \rag \\
 &&\ \ \ +\lag \pa^\al (u\times B), \pa^\al u \rag +\lag \pa^\al (n \De u), \pa^\al u \rag\\
 &&\lesssim\|\na u\|_{H^{N-1}}\|\na^2 u\|_{H^{N-1}}\|\na u\|_{H^{N-1}}
 +\|\na n\|_{H^{N-1}}^2\|\na^2 u\|_{H^{N-1}}\\
 &&\ \ \ +(\|u\|\cdot\|\na B\|\cdot \|\na^2 u\|_{H^1}+\|\na u\|_{H^{N-1}}\|\na B\|_{H^{N-1}}\|\na^2 u\|_{H^{N-2}})\\
 &&\ \ \ +\|\na n\|_{H^{N-1}}\|\na^2 u\|_{H^{N-1}}^2.
\end{eqnarray*}
Similarly,
\begin{equation*}
   \lag  \pa^\al h_1, \pa^\al n \rag\sim \lag  \pa^\al\na \cdot (nu), \pa^\al n \rag \lesssim \|\na n\|_{H^{N-1}}^2  \|\na^2 u\|_{H^{N-1}}.
\end{equation*}
By collecting the above estimates, it follows from \eqref{lem.hee.p1} that
\begin{multline}\label{lem.hee.p2}
\dis\frac{d}{dt}\CE_N^{\rm h}(U(t))+c\CD_N^{\rm h} (U(t)) \\
\dis \lesssim  (\|u\|+\CE_N^{\rm h}(U(t))^{1/2})\CE_N^{\rm h}(U(t))^{1/2}\CD_N^{\rm h}(U(t))^{1/2}\\
+\CE_N^{\rm h}(U(t))^{1/2}\CD_N^{\rm h}(U(t))
+\CE_N^{\rm h}(U(t))\left[\CE_N^{\rm h}(U(t))+\CD_N^{\rm h}(U(t))\right].
\end{multline}
Since $\|U_0\|_{H^N}$ is sufficiently small and so is $\CE_N(U(t))$ uniformly for all $t\geq 0$ by \eqref{thm.ge.1}, \eqref{lem.hee.1} holds by applying the Cauchy inequality to \eqref{lem.hee.p2}. This completes the proof of Lemma \ref{lem.hee}.
\end{proof}

Furthermore, we define
\begin{equation*}
    Y(t)=\sup_{0\leq s\leq t}\left\{(1+s)^{\frac{5}{4}}\left[\|u(s)\|^2+\CE_N^{\rm h}(U(s))\right]\right\},\ \ t\geq 0.
\end{equation*}
Similar to obtain the uniform-in-time bound of $X(t)$ in Lemma \ref{lem.X}, we have the following result to show the boundedness of $Y(t)$ for all $t\geq 0$ and thus the time decay rate of $\|u\|$.

\begin{lemma}\label{lem.Y}
If $\|U_0\|_{L^1\cap H^{N+2}}$ is sufficiently small, then
\begin{equation}\label{lem.Y.1}
    \sup_{t\geq 0} Y(t)\leq C \|U_0\|_{L^1\cap H^{N+2}}^2.
\end{equation}
\end{lemma}

\begin{proof}
The starting point is based on the high-order energy inequality \eqref{lem.hee.1}. Fix $\eps>0$ small enough. The time-weighted estimate on \eqref{lem.hee.1} gives
\begin{multline}\label{lem.Y.p1}
(1+t)^{\frac{5}{4}+\eps}\CE_N^{\rm h}(U(t))+\int_0^t (1+s)^{\frac{5}{4}+\eps}\CD_N^{\rm h}(U(s))ds \\
\lesssim \CE_N(U_0)+ \int_0^t (1+s)^{\frac{5}{4}+\eps}[\|u(s)\|^2+\CE_N^{\rm h}(U(s))]\CE_N^{\rm h}(U(s))ds\\
+\int_0^t(1+s)^{\frac{1}{4}+\eps}\CE_N^{\rm h}(U(s))ds.
\end{multline}
It is easy to see that the right-hand second term is bounded by $Y(t)^2$ since
\begin{equation*}
 \int_0^t (1+s)^{\frac{5}{4}+\eps}(1+s)^{-\frac{5}{4}-\frac{5}{4}}ds\leq C
\end{equation*}
for $\eps>0$ small enough. For the last term of \eqref{lem.Y.p1}, one has
\begin{multline*}
\int_0^t(1+s)^{\frac{1}{4}+\eps}\CE_N^{\rm h}(U(s))ds\leq \int_0^t(1+s)^{\frac{1}{4}+\eps}[\CD_{N+1}(U(s))+\|\na B(s)\|^2]ds\\
\lesssim (1+t)^\eps\|U_0\|_{L^1\cap H^{N+2}}^2+ \int_0^t(1+s)^{\frac{1}{4}+\eps}\|\na B(s)\|^2ds.
\end{multline*}
where \eqref{lem.X.p3} was used. Here, from Corollary \ref{cor.ld}, one can estimate $\|\na B\|$ as
\begin{multline*}
\|\na B\|\lesssim (1+t)^{-\frac{5}{8}}\|U_0\|_{L^1\cap H^4}\\
+\int_0^t (1+t-s )^{-\frac{7}{8}}(\|h_2(s)\|_{L^1\cap L^2}+\|\na h_2(s)\|)ds\\
+\int_0^t (1+t-s )^{-\frac{11}{8}}(\|h_3(s)\|_{L^1\cap L^2}+\|\na^4 h_3(s)\|)ds.
\end{multline*}
Since
\begin{multline*}
\|h_2(t)\|_{L^1\cap L^2}+\|\na h_2(t)\|+\|h_3(t)\|_{L^1\cap L^2}+\|\na^4 h_3(t)\|\\
\lesssim \|U_0\|_{L^1\cap H^{N+1}} (1+t)^{-\frac{3}{8}} (1+t)^{-\frac{5}{8}} Y(t)^{1/2}\\
\lesssim \|U_0\|_{L^1\cap H^{N+1}} (1+t)^{-1} Y(t)^{1/2},
\end{multline*}
it follows
\begin{equation*}
\|\na B\|\lesssim (1+t)^{-\frac{5}{8}}\|U_0\|_{L^1\cap H^{N+1}}(1+Y(t)^{1/2}).
\end{equation*}
Putting the above estimates into \eqref{lem.Y.p1} yields
\begin{eqnarray*}
&&(1+t)^{\frac{5}{4}+\eps}\CE_N^{\rm h}(U(t))+\int_0^t (1+s)^{\frac{5}{4}+\eps}\CD_N^{\rm h}(U(s))ds \\
&&\lesssim (1+t)^\eps\|U_0\|_{L^1\cap H^{N+2}}^2+ Y(t)^2\\
&&\ \ \ +\int_0^t(1+s)^{\frac{1}{4}+\eps}(1+s)^{-\frac{5}{4}}\|U_0\|_{L^1\cap H^{N+1}}^2(1+Y(s))ds\\
&&\lesssim (1+t)^\eps \|U_0\|_{L^1\cap H^{N+2}}^2(1+Y(t))+Y(t)^2,
\end{eqnarray*}
which implies
\begin{equation}\label{lem.Y.p2}
    \sup_{0\leq s\leq t}(1+s)^{\frac{5}{4}} \CE_N^{\rm h}(U(s))\lesssim \|U_0\|_{L^1\cap H^{N+2}}^2(1+Y(t))+Y(t)^2.
\end{equation}
The rest is to estimate the optimal time-decay of $\|u\|$ in terms of $Y(t)$. In fact,
for $\|u\|$, from Corollary \ref{cor.ld}, one has
\begin{multline*}
\|u(t)\|\lesssim (1+t)^{-\frac{5}{8}}\|U_0\|_{L^1\cap L^2}
+\int_0^t (1+t-s )^{-\frac{5}{4}}\|h_1(s)\|_{L^1\cap L^2}ds\\
+\int_0^t (1+t-s )^{-\frac{3}{4}}\|[h_2(s),h_3(s)\|_{L^1\cap L^2}ds\\
\lesssim  (1+t)^{-\frac{5}{8}}\|U_0\|_{L^1\cap H^{N+1}} Y(t)^{1/2},
\end{multline*}
that is,
\begin{equation}\label{lem.Y.p3}
    \sup_{0\leq s\leq t}(1+s)^{\frac{5}{4}} \|u(s)\|^2\lesssim \|U_0\|_{L^1\cap H^{N+1}}^2Y(t).
\end{equation}
Combining \eqref{lem.Y.p2} and \eqref{lem.Y.p3} gives
\begin{equation*}
    Y(t)\leq C\|U_0\|_{L^1\cap H^{N+2}}^2 (1+Y(t))+Y(t)^2.
\end{equation*}
Therefore \eqref{lem.Y.1} holds true since $\|U_0\|_{L^1\cap H^{N+2}}$ is sufficiently small. This completes the proof of Lemma \ref{lem.Y}.
\end{proof}

To complete the proof of Theorem \ref{thm.odecay}, the rest is to consider the time decay estimates of $n$ and $E$. In fact, we have

\begin{lemma}\label{lem.nE}
If $\|U_0\|_{L^1\cap H^{N+2}}$ is sufficiently small, then
\begin{eqnarray}
   \|n(t)\|&\leq& C\|U_0\|_{L^1\cap H^{N+2}}(1+t)^{-1},\label{lem.nE.1}\\
   \|E(t)\|&\leq& C\|U_0\|_{L^1\cap H^{N+2}}(1+t)^{-\frac{3}{4}}\ln (3+t), \label{lem.nE.2}
\end{eqnarray}
for any $t\geq 0$.
\end{lemma}

\begin{proof}
For $\|n\|$, from Corollary \ref{cor.ld}, one has
\begin{equation*}
\|n(t)\|\lesssim (1+t)^{-\frac{5}{4}}\|U_0\|_{L^1\cap L^2}
+\int_0^t (1+t-s )^{-\frac{5}{4}}\|[h_2(s),h_3(s)]\|_{L^1\cap L^2}ds.
\end{equation*}
Here, it holds that
\begin{equation*}
  \|[h_2(t),h_3(t)]\|_{L^1\cap L^2} \lesssim  \|U_0\|_{L^1\cap H^{N+2}}^2 (1+t)^{-\frac{5}{8}-\frac{3}{8}}\lesssim  \|U_0\|_{L^1\cap H^{N+2}}^2 (1+t)^{-1}.
\end{equation*}
Then, \eqref{lem.nE.1} follows. For $\|E(t)\|$, similarly, it suffices to notice from Corollary \ref{cor.ld} that
\begin{multline*}
\|E(t)\|\lesssim (1+t)^{-\frac{3}{4}}\|U_0\|_{L^1\cap H^3}
+\int_0^t (1+t-s )^{-\frac{3}{4}}\|h_2(s)\|_{L^1\cap L^2}ds\\
+\int_0^t (1+t-s )^{-\frac{3}{4}}(\|h_3(s)\|_{L^1\cap L^2}+\|\na^2 h_3(s)\|)ds\\
\lesssim   (1+t)^{-\frac{3}{4}}\|U_0\|_{L^1\cap H^3}+\int_0^t(1+t-s )^{-\frac{3}{4}}(1+s)^{-1}ds\|U_0\|_{L^1\cap H^{N+2}}^2\\
\lesssim\|U_0\|_{L^1\cap H^{N+2}}(1+t)^{-\frac{3}{4}}\ln (3+t).
\end{multline*}
Then \eqref{lem.nE.2} holds. This completes the proof of Lemma \ref{lem.nE}.
\end{proof}

\medskip

\noindent{\bf Proof of Theorem \ref{thm.main}:} The global existence is proved in Theorem \ref{thm.ge}, and the time decay rates are obtained in Theorem \ref{thm.odecay}. \qed

\section{Appendix}

In this section we provide a derivation of the Navier-Stokes-Maxwell system \eqref{NSM} from the one-species Vlasov-Maxwell-Boltzmann system. The main tool that we shall use is the Liu-Yang-Yu's macro-micro decomposition, which was initiated in Liu-Yu \cite{LY-BE} and developed in Liu-Yang-Yu \cite{LYY}, and also has extensive applications in the study of the stability of the Boltzmann equation, cf.~\cite{YZ-CMP} and references therein.

The Vlasov-Maxwell-Boltzmann system is a kinetic model in plasma physics to describe the time evolution of dilute charged particles (e.g., electrons and ions) with the self-consistent Lorentz forces \cite{MRS}. Under the same physical assumptions as for the  Navier-Stokes-Maxwell system \eqref{NSM}, the Vlasov-Maxwell-Boltzmann system takes the form of
\begin{equation}\label{VMB.o}
    \left\{\begin{array}{l}
      \dis \pa_t f+\xi\cdot \na_x f-(E+\xi \times B)\cdot \na_\xi f =Q
      (f,f),\\[3mm]
      \dis\pa_t E-\na_x\times B=\int_{\R^3}\xi f\, d\xi,\\[3mm]
      \dis\pa_t B +\na_x \times E =0,\\[3mm]
      \dis\na\cdot E=n_{\rm b}-\int_{\R^3}f \,d\xi,\ \ \na_x\cdot B=0.
    \end{array}\right.
\end{equation}
Here, the unknowns are $f=f(t,x,\xi): (0,\infty)\times \R^3\times \R^3\to [0,\infty)$, $E=E(t,x): (0,\infty)\times \R^3\to \R^3$ and $B=B(t,x): (0,\infty)\times \R^3\to \R^3$, with $f(t,x,\xi)$ standing for the number
distribution function of one-species of particles (e.g., electrons) which have position
$x$ and velocity
$\xi=(\xi_1,\xi_2,\xi_3)$ at time $t$, and $E(t,x)$ and
$B(t,x)$ denoting the electromagnetic field  as in the fluid case. Notice that throughout this appendix, we still use $\xi$, which is the same notation as in Section \ref{sec3.2}, to denote the micro velocity variable for simplicity, and this does not make any confusion.

Let $f(t,x,\xi)$ be the
solution to the system \eqref{VMB.o}. As in \cite{LYY}, in terms of $f(t,x,\xi)$,  we introduce
five fluid
quantities, that is,  the mass density $n(t,x)$, momentum density
$n(t,x) u(t,x)$ and energy density $e(t,x)+|u(t,x)|^2/2$  by
\begin{eqnarray*}
\left\{
\begin{array}{rl}
n(t,x) \equiv&{\displaystyle\int_{\R^3}}\psi_0(\xi) f(t,x,\xi) d
\xi,\\[3mm]
n(t,x) u_i(t,x) \equiv& {\displaystyle\int_{\R^3}} \psi_i(\xi)
f(t,x,\xi) d \xi, \quad  i=1,2,3,\\[3mm]
\left[\rho\left(e+\frac{1}{2} |u|^2\right)\right](t,x)
\equiv&\dis \int_{\R^3} \psi_{4}(\xi) f(t,x,\xi) d \xi,
\end{array}
\right.
\end{eqnarray*}
and further define the local Maxwellian as
\begin{eqnarray*}
{\bf M}\equiv {\bf M}_{[n,u,\theta]}
 (\xi) \equiv \frac{ n(t,x)  }{ (2 \pi
 \theta(t,x))^{3/2}} \exp\left( -\frac{|\xi -u(t,x)|^2  }{ 2 \theta(t,x)   }
 \right).
\end{eqnarray*}
Here $\theta(t,x)$ is the temperature related to the internal
energy $e(t,x)$ by $e=\frac{3}{2} \theta$,
$u(t,x)=\left(u_1(t,x), u_2(t,x), u_3(t,x)\right)$ is the
fluid velocity. As usual, $\psi_\alpha(\xi)$, $\alpha =0,1,
2,3, 4$, are the five collision invariants:
\begin{eqnarray*}
\left\{
\begin{array}{rl}
 \psi_0(\xi) \equiv& 1, \\[3mm]
 \psi_i(\xi) \equiv& \xi_i, \  \  i=1,2,3,\\[3mm]
 \psi_{4}(\xi) \equiv& \frac{1}{2} |\xi|^2.
 \end{array}
\right.
\end{eqnarray*}
With the above definitions, the Maxwell equations in system \eqref{VMB.o} can be rewritten as
\begin{equation}\label{A.max}
    \left\{\begin{array}{l}
      \dis\pa_t E-\na_x\times B=nu,\\[3mm]
      \dis\pa_t B +\na_x \times E =0,\\[3mm]
      \dis\na\cdot E=n_{\rm b}-n,\ \ \na_x\cdot B=0.
    \end{array}\right.
\end{equation}
This is exactly the same as the part of  the Maxwell equations in system \eqref{NSM}. The rest is to derive the same as the part of the Navier-Stokes equations in system \eqref{NSM}.

As in \cite{LYY}, let us
write $f(t,x,\xi)$ as the sum of the local Maxwellian ${\bf M}$ and
 microscopic component ${\bf G}={\bf
 G}(t,x,\xi)$, that is,
\begin{eqnarray}\label{1.7}
f(t,x,\xi)={\bf M}(t,x,\xi)+{\bf G}(t,x,\xi).
\end{eqnarray}
By doing this, the system \eqref{VMB.o}  can be reformulated
into a coupled system containing  the conservation laws for the
macroscopic fluid components and the equation for the microscopic
component. We shall see that the Navier-Stokes-Maxwell system \eqref{NSM} turns out to be the macroscopic fluid part.

For later use, we need to define the projection with respect to the Maxwellian $\FM$. To do that, define the inner product $\langle \cdot,\cdot\rangle_{\FM}$ over $L^2_\xi$ as
$$
 \langle h,g\rangle_{\FM}\equiv \int_{\R^3}
 h(\xi)g(\xi)\FM^{-1}d \xi.
$$
As in \cite{LYY}, the macroscopic linear subspace $\M$  spanned by
$
\{\psi_0(\xi){\FM}, \psi_i(\xi){\FM}, i=1,2,3,
 \psi_{4}(\xi){\FM}\}
$
has an orthogonal basis consisting of
\begin{eqnarray*}
\left\{
\begin{array}{rl}
\chi_0=&
\frac{1}{\sqrt{
n}}\FM,\\[3mm]
\chi_i=&
\frac{\xi_i-u_i}{\sqrt{
 n\theta}}\FM,  \ \
 i=1,2,3,\\[3mm]
\chi_{4}=&
\frac{ 1 }{\sqrt{6n}}
 \left(\frac{|\xi-{u}|^2}{{\theta}} -3 \right)\FM,\\[3mm]
\left\langle\chi_{\alpha},\chi_{\beta}\right\rangle_{\FM}=&\delta_{\alpha \beta}, \
 \ \ {\textrm{for}} \ \alpha, \ \beta=0,1,2,3, 4.
\end{array}
\right.
\end{eqnarray*}
The macroscopic projection ${\bf P}_0$ and
microscopic projection ${\bf P}_1$ can then be
defined by
\begin{equation*}
    {\bf P}_0 h \equiv  \sum\limits_{\alpha=0}^{4}
\left\langle h,\chi_{\alpha}
 \right\rangle_{\FM} \chi_\alpha,\ \
  {\bf P}_1 h \equiv   h- {\bf
 P}_0h.
\end{equation*}

Now, we go back to the Vlasov-Maxwell-Boltzmann equation \eqref{VMB.o}.  The property of collision invariants implies
\begin{eqnarray*}
\int_{\R^3} \psi_\alpha\Big\{f_t+\xi\cdot \nabla_x
f-(E+\xi \times B)\cdot\nabla_\xi f\Big\} d\xi =0,
\end{eqnarray*}
for  $\alpha=0,1,\cdots,4$.
Using the
decomposition \eqref{1.7} further gives
\begin{equation}\label{1.11}
\left\{
\begin{array}{l}
\pa_t \rho + \na_x\cdot (nu) =0, \\[3mm]
\pa_t (nu_i) + \sum\limits_{j=1}^3 \pa_{j}(nu_iu_j)+\pa_i P+n(E+u\times B)_i\\
\hspace{4cm}+ {\displaystyle\int_{\R^3}}\psi_i(\xi)\xi\cdot \nabla_x {\bf G} d \xi=0, \
 \  i=1,2,3,\\[3mm]
\pa_t\left[\rho\left(\frac{|u|^2}{2}+e\right)\right] +
\sum\limits_{j=1}^3\pa_j\left\{u_j\left[\rho\left(\frac{|u|^2}{2}+
e\right)+p\right]\right\}+nu\cdot(E+u\times B)\\
\hspace{4cm}+
{\displaystyle \int_{\R^3}}\psi_{4}(\xi) \xi \cdot
\nabla_x {\bf G}
 d \xi=0.
\end{array}
\right.
\end{equation}
Here and hereafter we use $\pa_i=\pa_{x_i}$ for simplicity. In \eqref{1.11}, the state equation for the monatomic gas is given by
$P=\frac 2 3 n e$ and hence $P=n\theta$.
By noticing that $\pa_t{\bf G}$, $\nabla_\xi{\bf G}$, $\FL_{\bf M}{\bf
G}$, and $Q({\bf G}, {\bf G})$ are microscopic, the
microscopic equation is obtained by applying the microscopic
projection ${\bf P}_1$ to the first equation of  system \eqref{VMB.o}:
\begin{equation}\label{1.12}
 \pa_t{\bf G}+ {\bf P}_1(\xi \cdot \nabla_x {\bf M}+\xi \cdot \nabla_x
 {\bf G} )-(E+\xi\times B)\cdot\nabla_\xi {\bf G}
  = \FL_{\bf M} {\bf G} + Q({\bf G},{\bf G}),
\end{equation}
where  $ \FL_{\bf M}$ is the linearized collision operator defined
by
\begin{eqnarray*}
 \FL_{\bf M}g \equiv Q(\FM+g,\FM+g)-Q(g,g)=Q(\FM,g)+Q(g,\FM).
\end{eqnarray*}
It is known that the null space of $\FL_{\bf M}$ is exactly
$\M$. Since $\FL_{\bf M}$ is a bounded and one-to-one  operator on
$\M^\bot$, from (\ref{1.12}), one has
\begin{equation}\label{1.18}
{\bf G}=  \FL_{\bf M}^{-1} \left(
  {\bf P}_1 \left(\xi \cdot \nabla_x {\bf M}\right) \right)+\Theta
\end{equation}
with
\begin{equation}\label{A.The}
    \Theta=\FL_{\bf M}^{-1} \left( \partial_t {\bf G} +
 {\bf P}_1 \left(\xi \cdot \nabla_x {\bf G}
 \right)-(E+\xi\times B)\cdot\nabla_\xi{\bf G}
 -Q({\bf G},{\bf G}) \right).
\end{equation}
Substituting (\ref{1.18}) into (\ref{1.11}) yields the following fluid-type system
\begin{eqnarray}\label{1.19}
\left\{
 \begin{array}{l}
 \pa_tn + \na_x\cdot (nu)=0, \\[3mm]
 \pa_t (nu_i)+
 \sum\limits_{j=1}^3\pa_j\left(nu_iu_j\right)+\pa_iP+n(E+u\times B)_i\\
\quad\quad =\sum\limits_{j=1}^3\pa_j\left[\nu(\theta)\tau_{ij}\right]
-{\displaystyle\int_{\R^3}}\psi_i(\xi) \xi\cdot\nabla_x \Theta d \xi,
\ i=1,2,3,\\[5mm]
 \pa_t\left[n( \frac 12|u|^2+e)\right] +
 \na_x\cdot
 (u(n(\frac 12 |u|^2 + e)+p))
 +nu\cdot(E+u\times B) \\[3mm]
\quad\quad =\sum\limits_{i,j=1}^3\pa_j\left\{\nu(\theta)u_i\tau_{ij}\right\}
+\sum\limits_{j=1}^3\pa_j\left(\kappa(\theta)\pa_j\theta\right)
-{\displaystyle \int_{\R^3}}
\psi_{4}(\xi)\xi\cdot\nabla_x \Theta d \xi.
\end{array}
\right.
\end{eqnarray}
with
\begin{equation*}
    \tau_{ij}=\pa_ju_{i}+\pa_iu_{j}
 -\frac 23\delta_{ij}\na_x\cdot u.
\end{equation*}
Here,  the viscosity
coefficient $\nu(\theta)$ and heat conductivity coefficient
$\kappa(\theta)$ both depending only on $\theta$ can be explicitly given by introducing the Burnett functions, cf.~\cite{YZ-CMP}.

Observe that the microscopic term $\Theta$ in \eqref{A.The} is generated by the linear superposition of the first-order time-space derivative of ${\bf G}$ and the nonlinear terms. If all the terms containing  $\Theta$ in \eqref{1.19} are dropped, then by recalling \eqref{A.max}, the fluid quantities $n$, $u$, $e=\frac{3}{2}\theta$ and the electromagnetic field $E$, $B$ satisfy the following coupled Navier-Stokes-Maxwell system in the non-isentropic case
\begin{equation*}
    \left\{\begin{array}{l}
    \pa_tn + \na_x\cdot (nu)=0, \\[3mm]
 \pa_t (nu_i)+
 \sum\limits_{j=1}^3\pa_j\left(nu_iu_j\right)+\pa_i P+n(E+u\times B)_i\\
\hspace{4cm} =\sum\limits_{j=1}^3\pa_j\left[\nu(\theta)\tau_{ij}\right],
\ i=1,2,3,\\[5mm]
 \pa_t\left[n( \frac 12|u|^2+e)\right] +
 \na_x\cdot
 (u(n(\frac 12 |u|^2 + e)+p))
 +nu\cdot(E+u\times B) \\[3mm]
\hspace{4cm} =\sum\limits_{i,j=1}^3\pa_j\left\{\nu(\theta)u_i\tau_{ij}\right\}
+\sum\limits_{j=1}^3\pa_j\left(\kappa(\theta)\pa_j\theta\right),\\[3mm]
      \dis\pa_t E-\na_x\times B=nu,\\[3mm]
      \dis\pa_t B +\na_x \times E =0,\\[3mm]
      \dis\na\cdot E=n_{\rm b}-n,\ \ \na_x\cdot B=0.
    \end{array}\right.
\end{equation*}
System \eqref{NSM} that we have studied here is the simplified form in the isentropic case.

\smallskip
{\bf Acknowledgment.} This project was supported by the Direct Grant 2010/2011 from CUHK and partially by the General Research Fund (Project No.~400511) from RGC of Hong Kong.

\medskip


\end{document}